\documentclass[11pt]{article}
\usepackage{titlesec,minitoc,ulem}
\usepackage[latin1]{inputenc}
\usepackage[frenchb]{babel}
\usepackage{amssymb,amsthm,verbatim,latexsym,dsfont,amsmath}
\usepackage[all]{xy}
\usepackage[tight]{shorttoc}

\newcommand{\sommaire}{\shorttoc{Sommaire}{1}}

\numberwithin{equation}{subsubsection}

\newcommand{\0}{\mathcal{O}}

\newcommand{\G}{\mathrm{GL}}
\newcommand{\g}{\mathfrak{{gl}}}

\newcommand{\D}{\mathrm{diag}}

\newcommand{\ra}{\rightarrow}

\newcommand{\mg}{\mathds{G}}
\newcommand{\ud}{\underline{d}}

\newcommand{\lsi}{\mathcal{L}_{\psi}}
\newcommand{\lze}{\mathcal{L}_{\zeta}}
\newcommand{\fsi}{\mathfrak{F}_{\psi}}
\newcommand{\tG}{\widetilde{\mathrm{ GL}}}

\newcommand{\fp}{F_{\varpi}}
\newcommand{\can}{\buildrel can \over{\simeq}}
\newcommand{\op}{\mathcal{O}_{\varpi}}
\newcommand{\va}{\varpi}
\newcommand{\m}{\mathrm{Mat}}

\newcommand{\aS}{\mathbb{A}^S}

\newcommand{\n}{\mathfrak{n}}

\newtheorem*{thmA}{TH\'EOR\`EME A}
\newtheorem*{thmB}{TH\'EOR\`EME B}

\newtheorem*{thmA'}{TH\'EOR\`EME A$'$}

\theoremstyle{definition}
\newtheorem{theorem}{Th\'eor\`eme}[subsubsection]
\makeatother
\newtheorem{conj}[theorem]{Conjecture}
\newtheorem{proposition}[theorem]{Proposition}
\newtheorem{corollary}[theorem]{Corollaire}
\newtheorem{lemma}[theorem]{Lemme}
\newtheorem{defn}[theorem]{D\'efinition}

\newtheorem{const}[theorem]{Construction}
\newtheorem{rmk}[theorem]{Remarque}
\title{Le lemme fondamental métaplectique de Jacquet et Mao en égales caractéristiques}
\author{DO Viet Cuong}
\begin{document}
\maketitle
\sommaire
\section{Introduction}
\setcounter{section}{1}
\setcounter{subsection}{1}
\setcounter{subsubsection}{1}
Soit $\mathbb{A}$ l'anneau des adèles d'un corps global et $\tG_r(\mathbb{A})$ le revêtement métapl\-ectique de $\G_r(\mathbb{A})$ (c'est un revêtement à deux feuillets de $\G_r(\mathbb{A})$, qui est une extension centrale par $\{\pm 1\}$, cf \cite{KP}). Pour $r=2$, Jacquet \cite{J1} a montré qu'on pouvait utiliser une formule des traces relative pour caractériser l'image de l'application de relèvement automorphe de $\tG_r(\mathbb{A})$ à $\G_r(\mathbb{A})$ comme ensemble des représentations distinguées par le groupe orthogonal. Pour $r$ arbitraire, Mao \cite{M} a écrit la formule des traces relative correspondante. Pour achever la caractérisation de l'image du relèvement il reste entre autres un énoncé local à démontrer (le ``lemme fondamental métaplectique de Jacquet-Mao").
L'objet de ce travail est de donner une démonstration de cet énoncé dans le cas de caractéristique positive.

On va d'abord énoncer ce lemme fondamental. Soient $F$ un corps local non-archimédien, $\0$ son anneau des entiers et $k$ son corps résiduel, que l'on suppose de caractéristique $p$ impaire et de cardinal $q$. On fixe $\varpi$ une uniformisante de $F$. Soit $\psi :k\to \mathbb{C}^*$ un caractère additif. On note $\Psi(x)=\psi({\rm res}\,x\,{\rm d}\varpi)$; c'est un caractère additif de $F$ dans $\mathbb{C}^*$. On note $v(x)$ la valuation de l'élément $x\in F$ et $|x|=q^{-v(x)}$. On note $N_r$ le sous-groupe de $\G_r$ formé des matrices triangulaires supérieures unipotentes, $T_r$ celui formé des matrices diagonales et $S_r$ le sous-schéma de $\G_r$ formé des matrices symétriques. Soit $\theta :N_r(F)\to \mathbb{C}^*$ le caractère défini par $\theta(n)=\Psi(\frac{1}{2}\sum_{i=2}^rn_{i-1,i})$.

Le revêtement de Kazhdan-Patterson local est (canoniquement) scindé au-dessus de $N_r(F)$ ainsi qu'au-dessus de $\G_r(\0)$. On peut écrire les éléments de $\tG_r(F)$ sous la forme $\tilde{g}=(g,z)$, avec $g\in \G_r(F)$ et $z\in\{\pm1\}$ et la multiplication est alors définie par
$$(g,z).(g',z')=(gg',\chi(g,g')zz'),$$
où $\chi$ est un certain cocycle, pour la description duquel on renvoie à Kazhdan-Patterson \cite{KP}. Ces notations étant fixées, le scindage $\sigma$ au-dessus de $N_r(F)$ est simplement défini par $\sigma(n)=(n,1)$ alors que le scindage $\kappa^*$ au-dessus de $\G_r(\0)$ est défini par $\kappa^*(g)=(g,\kappa(g))$ - la fonction $\kappa :\G_r(F)\to\{\pm1\}$ n'est vraiment pas facile à calculer et son calcul dans la situation géométrique qu'on envisagera est l'un des résultats importants de ce travail.


Le groupe $N_r(F)$ agit sur $S_r(F)$ par $s\mapsto {}^tnsn$ et le groupe $N_r(F)\times N_r(F)$ agit sur $\G_r(F)$ par $g\mapsto n^{-1}gn'$. On appelle {\it pertinentes} les orbites $\dot{s}$ sous l'action de $N_r(F)$ (resp. les orbites $\dot{g}$ sous l'action de $N_r(F)\times N_r(F)$) telles que pour tout $n$ appartenant au fixateur $(N_r(F))_s$ de $s$ on ait $\theta^2(n)=1$ (resp. telles que pour toute couple $(n,n')$ appartenant au fixateur $(N_r(F)\times N_r(F))_g$ de $g$ on ait $\theta(n^{-1}n')=1$). 
Dans cet article, on s'intéresse au cas où ces fixateurs sont triviaux, qui est en fait le cas fondamental \cite{J}. Il existe alors un représentant dans l'orbite sous l'action de $N_r(F)$ qui est une matrice diagonale $t$ (resp. un représentant de la forme $w_0t$, où $w_0$ est la matrice de la permutation $(r,\dots,1)$ et $t$ est une matrice diagonale, dans l'orbite sous l'action de $N_r(F)\times N_r(F)$).

Pour chaque matrice diagonale $t\in T_r(F)$, on introduit les deux intégrales orbitales
$$I(t)=\int_{N_r(F)}\phi_0({}^tntn)\theta^2(n)dn$$
et
$$ J(t)=\int_{N_r(F)\times N_r(F)}f_0(n^{-1}w_0tn')\theta(n^{-1}n')dndn',$$
où $f_0$ est la fonction définie par
$$f_0(g)=\begin{cases}\kappa(g),&\textrm{si $g\in \G_r(\0)$}\\
0,&\textrm{sinon,}\end{cases}$$
et
$$\phi_0(g)=\begin{cases}1,&\textrm{si $g\in \G_r(\0)\cap S_r(F)$}\\
0,&\textrm{sinon,}\end{cases},$$
les mesures de Haar de $N_r(F)$ et $N_r(F)\times N_r(F)$ étant normalisées de sorte qu'elles attribuent la mesure $1$ aux sous-groupes compacts ouverts formés des matrices à coefficients dans $\0$.

Soit $\zeta :k^*\to\{\pm1\}$ le caractère quadratique non trivial ($\zeta(\lambda)=\lambda^{\frac{q-1}{2}},\,\lambda\in k^*$). On note $\gamma(a,\Psi)$ la constante de Weil, qui est définie par la formule
$$\int\Phi^{\vee}(x)\Psi\left(\frac{1}{2}ax^2\right)dx=|a|^{-1/2}\gamma(a,\Psi)\int\Phi(x)\Psi\left(-\frac{1}{2}a^{-1}x^2\right)dx,$$
où $\Psi :F\to \mathbb{C}^*$ est un caractère additif, $\Phi$ est une fonction de Schwartz sur $F$ et $\Phi^{\vee}$ est sa transformée de Fourier ($\Phi^{\vee}(x)=\int\Phi(y)\Psi(xy)dy$).
\begin{conj} (Jacquet et Mao). Soit $t=\D(t_1,\dots,t_r)$, on note $a_i=\prod_{j=1}^it_j$. On a alors $t=\D(a_1,a_1^{-1}a_2,\dots,a_{r-1}^{-1}a_r)$ et
$$J(t)=\begin{cases}\mathfrak{t}(t)I(t)\\
\mathfrak{t'}(t)I(t)
\end{cases},$$
o\`u $\mathfrak{t}(t)=|\prod_{i=1}^{r-1}a_i|^{-1/2}{\zeta(-1)}^{\sum_{j \not \equiv r({\rm mod} 2)}v(a_j)}\prod_{j\not \equiv r ({\rm mod} 2)}\gamma(a_ja_{j-1}^{-1},\Psi)$ et o\`u
$\mathfrak{t'}(t)=|\prod_{i=1}^{r-1}a_i|^{-1/2}{\zeta(-1)}^{\sum_{j \equiv r({\rm mod} 2)}v(a_j)}\prod_{j\equiv r ({\rm mod} 2)}\gamma(a_ja_{j-1}^{-1},\Psi)$ (en convenant que $a_0=1$).
De plus si $\mathfrak{t}(t)\neq\mathfrak{t'}(t)$, les deux int\'egrales $I(t)$ et $J(t)$ sont nulles.
\end{conj}
Jacquet et Mao ont démontré leur conjecture pour $\G_2$ (voir \cite{J}) et $\G_3$ (voir \cite{M}) pour tout corps local de caractéristique résiduelle $\neq 2$ (la formule de \cite{M} $\mu(a,b,c)=|a|^{-1}|b|^{-1/2}\gamma(a,\Psi)\gamma(-c,\Psi)$ doit en fait \^etre corrig\'ee en $\mu(a,b,c)=|a|^{-1}|b|^{-1/2}\gamma(-a,\Psi)\gamma(c,\Psi)$, comme on le voit en la comparant avec la formule de Jacquet \cite{J} pour $\G_2$).
\begin{thmA}\label{intro1} Si le corps local $F$ est de caractéristique positive, alors la conjecture de Jacquet et Mao est vérifiée.
\end{thmA}
Les méthodes qui sont utilisée par Jacquet et Mao dans leurs démonstrat\-ions pour $\G_2$ et $\G_3$ sont combinatoires (le calcul est déjà très complexe dans le cas $\G_3$) et il n'est pas clair qu'elles permettent d'obtenir l'énoncé pour tout rang $r$. On va donc utiliser une autre approche, en s'inspirant de la  
démonstration géométrique de B. C. Ngo pour le lemme fondamental de Jacquet et Ye (voir \cite{N}). 

On va en fait démontrer un énoncé un peu plus général que le théorème $A$, obtenu en remplaçant $\theta^2(n)$ par $\theta_{\underline{\alpha}}^2(n)$ dans la définition de l'intégrale $I$ et $\theta(n^{-1}n')$ par $\theta_{\overline{\alpha}}(n^{-1})\theta_{\underline{\alpha}}(n')$ dans la définition de l'intégrale $J$, où
$\underline{\alpha}=(\alpha_2,\dots,\alpha_r)\in (k^*)^{r-1}$, $\overline{\alpha}=(\alpha_r,\dots,\alpha_2)$ et $\theta_{\underline{\alpha}}(n)=\psi\circ\underline{h}_{\underline{\alpha}}$ avec $\underline{h}_{\underline{\alpha}}={\rm res} (\frac{1}{2}\sum_{i=2}^r\alpha_in_{i-1,i}{\rm d}\varpi)$. On notera $I(t,\underline{\alpha})$, $J(t,\underline{\alpha})$ les intégrales ainsi obtenues; bien sûr si $\underline{\alpha}=(1,\dots,1)$ on retrouve bien les intégrales $I$ et $J$ ci-dessus.

Les intégrales orbitales $I(t,\underline{\alpha})$ et $J(t,\underline{\alpha})$ portent sur des fonctions localement constantes à support compact : de ce fait elles se réduisent à des sommes finies. Lorsque le corps $F$ est de caractéristique positive, les ensembles d'indices de ces sommes finies s'interprètent naturellement comme ensembles de points à valeurs dans $k$ de variétés algébriques définies sur ce corps (qu'on appellera respectivement $X(t)$ et $Y(t)$). Plus précisément, on a :
$$I(t,\underline{\alpha})=\sum_{n\in X(t)(k)}\theta_{\underline{\alpha}}^2(n)\ \text{et }J(t,\underline{\alpha})=\sum_{(n,n')\in Y(t)(k)}\kappa(w_0{}^tntn')\theta_{\overline{\alpha}}(w_0{}^tnw_0)\theta_{\underline{\alpha}}(n'),$$
où $X(t)(k)=\{n\in N_r(F)/N_r(\0)|{}^tntn\in \G_r(\0)\cap S_r(F)\}$ et où $Y(t)(k)=\{(n,n')\in (N_r(F)/N_r(\0))^2|{}^tntn'\in\G_r(\0) \}$ ; pour la deuxième somme, on a en fait effectué le changement de variables $n^{-1}\mapsto w_0{}^tnw_0$ (pour que $X(t)(k)$ et $Y(t)(k)$ soient non vides, il faut que $a_1,\dots,a_{r-1}$ appartiennent à $\0$ et $a_r$ appartiennent à $\0^*$). 
Il reste alors à donner aussi aux fonctions une interprétation géométrique.


Du côté de l'intégrale $I$, on géométrise facilement 
le caractère $\psi$ à l'aide d'un revêtement d'Artin-Schreier (en considérant $\psi$ comme un caractère à valeurs dans $\overline{\mathbb{Q}}_{\ell}$, où $\ell$ est un nombre premier distinct de $p=\mathrm{car}(k)$). En revanche, du côté de l'intégrale $J$, il faut aussi géométriser la fonction $\kappa$. Pour cela, la construction de Kazhdan-Patterson \cite{KP} n'est pas très commode à utiliser directement et j'ai préféré une méthode plus indirecte, consistant à comparer l'extension de Kazhdan-Patterson à celle d'Arabello-De Concini-Kac, qui est de nature plus géométrique.


Arbarello, De Concini et Kac associent {\`a} chaque $g\in \G_{r}(F)$ une droite
$$D_g=(\bigwedge g\0^r/g\0^r\cap\0^r)\otimes(\bigwedge \0^r/g\0^r\cap\0^r)^{\otimes (-1)}$$
o{\`u} $\bigwedge V$ d{\'e}signe la puissance ext{\'e}rieure maximale $\bigwedge^{\dim V}V$ d'un $k$-espace vectoriel $V$ (en particulier, pour $g\in \G_r(\0)$ cette droite est canoniquement trivialisée). Cette construction fournit une extension centrale $\tG'_{r}(F)$ de $\G_{r}(F)$ par $k^*$. On utilisera plut\^ot la droite $\Delta_g=D_{\det(g)}\otimes D_g^{\otimes -1}$, ce qui revient \`a consid\'erer l'extension $\tG_{r,{\rm geo}}(F)=\det^*(\tG'_{1}(F))-\tG'_{r}(F)$ (la somme de Baer des extensions \'etant ici not\'ee additivement). On construit, \`a l'aide de la d\'ecomposition de Bruhat, une base $\delta(g)$ de $\Delta_g$, ce qui fournit une section $s_{\rm geo}$ de $\tG_{r,{\rm geo}}(F)$.
On rappelle qu'on note $\zeta :k^*\to \{\pm 1\}$ le caractère quadratique non-trivial.
\begin{proposition}On a
$\chi(g_1,g_2)=\zeta\left(\delta_{g_1}\otimes \delta_{g_2}\otimes \delta_{g_1g_2}^{\otimes -1}\right),$ o\`u $\chi$ est le $2$-cocycle de Kazhdan-Patterson \cite{KP}. En particulier, l'extension de Kazhdan-Patterson s'obtient \`a partir de notre extension $\tG_{r,{\rm geo}}(F)$ en la poussant par $\zeta$ 
\end{proposition}

La fonction $\kappa :\G_r(\0)\to\{\pm 1\}$ est donc $\zeta\circ\underline{\kappa}$ où $\underline{\kappa}$ est le quotient ${\rm triv}/\delta$ de la section triviale 
par la section $s_{\rm geo}$. 
Cette construction définit un morphisme $\underline{\kappa}: Y(t)\ra\mg_m$, qui induit une application $\kappa=\zeta\circ\underline{\kappa} : Y(t)(k)\to \{\pm1\}$.

Soit $\overline{k}$ une clôture algébrique de $k$. Soient $\mathcal{L}_{\psi}$ le faisceau d'Artin-Schreier sur $\mathbb{G}_a$ associé au caractère $\psi$ et $\mathcal{L}_{\zeta}$ le faisceau de Kummer sur $\mathbb{G}_m$ associé au revêtement $\mathbb{G}_m\ra\mathbb{G}_m,\, x\mapsto x^2$ et à $\{\pm1\}\subset \overline{\mathbb{Q}}_{\ell}$. D'après la formule des traces de Grothendieck-Lefschetz, on a :
$$I(t,\underline{\alpha})=\mathrm{Tr}(\mathrm{Fr},\mathrm{R}\Gamma_c(X(t)\otimes_k \overline{k},h^*_{\underline{\alpha}}\mathcal{L}_{\psi}));$$
$$ J(t,\underline{\alpha})=\mathrm{Tr}(\mathrm{Fr},\mathrm{R}\Gamma_c(Y(t)\otimes_k \overline{k},{h'}^*_{\underline{\alpha}}\mathcal{L}_{\psi}\otimes\underline{\kappa}^*\mathcal{L}_{\zeta})),$$
où $h_{\underline{\alpha}}$ (resp. $h'_{\underline{\alpha}}$) sont des morphismes de $X(t)$ et $Y(t)$ dans $\mg_a$ (provenant du morphisme $\underline{h}_{\underline{\alpha}}:N_r(F)\to k$ ci-dessus).

Le th\'eor\`eme $A$ est alors une cons\'equence de l'\'enonc\'e g\'eom\'etrique suivant, o\`u on note ${\cal I}(t)=\mathrm{R}\Gamma_c(X(t)\otimes_k \overline{k},h^*_{\underline{\alpha}}\mathcal{L}_{\psi})$ et ${\cal J}(t)=\mathrm{R}\Gamma_c(Y(t)\otimes_k \overline{k},{h'}^*_{\underline{\alpha}}\mathcal{L}_{\psi}\otimes\underline{\kappa}^*\mathcal{L}_{\zeta})$.
\begin{thmA'}
${\cal J}(t)\simeq{\cal T}(t)\otimes{\cal I}(t)\simeq
{\cal T'}(t)\otimes{\cal I}(t),$ o\`u ${\cal T}(t)$ et ${\cal T'}(t)$ sont des $\overline{\mathbb{Q}}_{\ell}$-espaces vectoriels de rang $1$ plac\'es en degr\'e $v(\prod_{i=1}^{r-1}a_i)$ tels que $\mathrm{Tr}(\mathrm{Fr},{\cal T}(t))=\mathfrak{t}(t)$ et  $\mathrm{Tr}(\mathrm{Fr},{\cal T'}(t))=\mathfrak{t'}(t)$.
\end{thmA'}
En géométrisant l'argument de Jacquet \cite[p. 145]{J}, on a en fait une démonstration directe du théorème $A'$ dans le cas particulier où $r=2$ et $t=\D(t_1,t_2)$ avec $v(t_1)=1,\ v(t_2)=-1$.
\begin{proposition}\label{intro2}
Le th\'eor\`eme $A'$ est vrai dans le cas particulier ci-dessus; de plus ${\cal I}(t)$ et
${\cal J}(t)$ sont alors des $\overline{{\mathbb{Q}}}_{\ell}$-espaces vectoriels de rang $2$ plac\'es respectivement en degr\'e $0$ et $1$.
\end{proposition}
Dans le cas général, on ne connaît pas explicitement ${\cal I}(t)$ (resp. ${\cal J}(t)$) car la variété ${X}(t)$ (resp. ${Y}(t)$) n'est ni lisse ni irréductible et sa dimension est très grande. Suivant une idée due à B.C. Ngo \cite{N} on va les déformer pour se ramener à une situation plus simple.

Comme dans loc. cit., on obtient ces déformations en considérant plutôt des sommes sur un corps global de caractéristique positive. Changeant de notations, on note maintenant ${\0}={k}[\varpi]$ l'anneau des polynômes en une variable $\varpi$ à coefficients dans ${k}$, et ${F}$ son corps des fractions (dans l'énoncé local ci-dessus, on va noter plutôt respectivement $\fp$, $\op$ et $k_{\varpi}$ le corps, son anneau des entiers et son corps résiduel). Pour tout $x \in{F}$, on note ${\rm sres}(x{\rm d}\varpi)=\sum_{v\in {\rm Spm}(\0)}{\rm Tr}_{k_v/k}{\rm res}_v(x{\rm d}\varpi)$ la somme des résidus en tous ses pôles à distance finie (où $k_v$ est le corps résiduel de $v$ et ${\rm res}_v$ est le résidu en $v$). Soit $\Psi: F\to\overline{\mathbb{Q}}^*_{\ell}$ le caractère défini par $\Psi(x)=\psi({\rm sres}(x\,{\rm d}\varpi))$. Soit $\underline{\alpha}=(\alpha_2,\dots,\alpha_r)\in \mg_m^{r-1}$. On note $\overline{\alpha}:=(\alpha_r,\dots,\alpha_2)$. Soit $\theta_{\underline{\alpha}} : N_r(F)\to \overline{\mathbb{Q}}^*_\ell$ le caractère défini par $\theta_{\underline{\alpha}}(n)=\Psi(\frac{1}{2}\sum_{i=2}^r\alpha_in_{i-1,i})$. Pour tout idéal maximal $v$ de $\0$, on note $\0_v$ la complété de $\0$ selon $v$, $F_v$ son corps des fractions, et $k_v$ son corps résiduel.

Du côté de l'intégrale $I$, pour toute place $v$, on modifie un peu l'intégrale $I$ locale en remplaçant la fonction caractéristique de $\G_r(\0_v)$ par celle de $\mathfrak{gl}_r(\0_v)$ (cf. \cite{N}), de sorte qu'on obtient la somme
$$I_v(t,\alpha)=\sum_{{\tiny\begin{matrix}n\in N_r(F_v)/N_r(\0_v)\\{}^tntn\in S_r(\0_v)\end{matrix}}}\theta_{\underline{\alpha},v}^2(n),$$
où $\theta_{\underline{\alpha},v}(n)=\psi(\frac{1}{2}\sum_{i=2}^r{\rm tr}_{k_v/k}{\rm res}(\alpha_in_{i-1,i}{\rm d\varpi}))$ (quand $v=\varpi$ et $a_r\in \op^*$, on retrouve bien la somme locale $I$ (qu'on va noter $I_\varpi$) dans l'énoncé du théorème $A$ ci-dessus). Ensuite, pour toute matrice diagonale $t$ à coefficients dans $F$, on définit
\begin{equation}\label{iintro} I(t,\underline{\alpha})=\prod_{v\in {\rm Spm}(\0)}I_v(t,\underline{\alpha}),
\end{equation}
 où $I_v(t,\underline{\alpha})=1$ sauf si $v$ divise $\prod_{i=1}^{r-1}a_i$. Ceci se réduit à une somme finie
$$I(t,\underline{\alpha})=\sum_{{\tiny\begin{matrix}n\in N_r(F)/N_r(\0)\\{}^tntn\in S_r(\0)\end{matrix}}}\theta_{\underline{\alpha}}^2(n)$$
laquelle est non-nulle seulement si $a_1,\dots,a_r$ sont des polynômes. 
En utilisant la même interprétation que pour la somme locale $I_\varpi$ ci-dessus, on obtient une donnée géométrique $(X(t),h_{\underline{\alpha}})$ associée à $I(t,\underline{\alpha})$ (resp. des données géométriques $(X_v(t),h_{\underline{\alpha},v})$ associées à $I_v(t,\underline{\alpha})$) et on a l'interprétation géom\-étrique suivante de l'identité (\ref{iintro})
\begin{proposition}
$$\mathrm{R}\Gamma_c(X(t)\otimes_k\overline{k},h_{\underline{\alpha}}^*\mathcal{L}_{\psi})=\bigotimes_{\lambda\in {\rm supp}(t)}\mathrm{R}\Gamma_c({X}_{v}(t)\otimes_k\overline{k},h_{\underline{\alpha},v}^*\mathcal{L}_{\psi}),$$
o\`u ${\rm supp}(t)$ est l'ensemble des racines de $\prod_{i=1}^{r-1}a_i$ dans $\overline{k}$.
\end{proposition}

Du côté de l'intégrale $J$, le problème qui se présente lorsqu'on tente d'adapter la démarche utilisée pour l'intégrale $I$ ci-dessus est que la fonction $\underline{\kappa}_{\mathrm{loc}}$  n'est définie que sur le sous-groupe compact maximal standard de $\G_{r,\mathrm{loc}}$. Pour contourner ce problème, on considère (en remarquant que  ${}^tN_r(F_v)tN_r(F_v)\cap \mathfrak{gl}_r(\0_v)={}^tN_r(F_v)tN_r(F_v)\cap \G_r(\0_v)$ quand $a_r\in \0_v^*$)
$$J_v(t,\underline{\alpha})=\begin{cases}\sum_{{\footnotesize \begin{smallmatrix}n,n'\in N_r(F_v)/N_r(\0_v)\\{}^tntn\in\mathfrak{gl}_r(\0_v)\end{smallmatrix}}}\kappa_v(w_0{}^tntn')\theta_{\overline{\alpha},v}(w_0{}^tnw_0)\theta_{\underline{\alpha},v}(n')&\text{ si $v\not|a_r$},\\
\sum_{{\footnotesize \begin{smallmatrix}n,n'\in N_r(F_v)/N_r(\0_v)\\{}^tntn\in \mathfrak{gl}_r(\0_v)\end{smallmatrix}}}\theta_{\overline{\alpha},v}(w_0{}^tnw_0)\theta_{\underline{\alpha},v}(n')&\text{ si $v |a_r$}
,\end{cases}$$
où $\kappa_v: \G_r(\0_v)\to \{\pm1\}$ provient du scindage au-dessus de $\G_r(\0_v)$ du revêtement de Kazhdan-Patterson local en $v$ 
 (quand $v=\varpi$ et $a_r\in \op^*$, on retrouve bien la somme locale $J$ (notée dorénavant $J_\varpi$) qui figure dans l'énoncé du théorème $A$ ci-dessus). Pour toute matrice diagonale $t$ à coefficients dans $F$, la somme globale en $t$ est en fait définie par
\begin{equation}\label{iintro2} J(t,\underline{\alpha})=\prod_{v\in {\rm Spm}(\0)}J_v(t,\underline{\alpha}),
\end{equation}
 où $J_v(t,\underline{\alpha})=1$ sauf si $v$ divise $\prod_{i=1}^{r-1}a_i$.

En adaptant la géométrisation de la somme local $J_{\varpi}$, on obtient des données géométriques associées à $J_v(t,\alpha)$ qui sont $(Y_v(t),h'_{\underline{\alpha},v},\underline{\kappa}_v)$ quand $v\nmid a_r$ (dans ce cas, on a que $\kappa_v=\zeta\circ\underline{\kappa}_v$) et $(Y_v(t),h'_{\underline{\alpha},v})$ quand $v|a_r$. Il reste à interpréter géométriquement la somme globale $J(t,\alpha)$. Pour cela, on introduit la variété $Y(t)$ dont l'ensemble de $k$ points est
$$Y(t)(k)=\{(n,n')\in (N_r(F)/N_r(\0))^2|{}^tntn \in \mathfrak{gl}_r(\0)\}.$$
Cette variété est munie d'un morphisme vers $\mg_a$ défini par $$h'_{\underline{\alpha}}(n,n')=\frac{1}{2}\sum_{i=2}^r\alpha_i\mathrm{sres}((n_{i-1,i}+n'_{i-1,i})\mathrm{d}\varpi).$$
La fonction $\kappa=\prod_{v\nshortmid a_r}\kappa_v$ est interprétée à l'aide du point de vue géométrique sur l'extension métaplectique évoqué ci-dessus. Plus précisément, en associant à chaque $g\in G_r(F)$ la droite $D_{\det(g)}(a_r^{-1})\otimes D_g(a_r^{-1})^{\otimes -1}$, où
$D_g(a_r^{-1}):=(\bigwedge \0[a_r^{-1}]^r/\0[a_r^{-1}]^r\cap g\0[a_r^{-1}]^r)^*\otimes(g\0[a_r^{-1}]^r/\0[a_r^{-1}]^r\cap g\0[a_r^{-1}]^r),$ on obtient une extension centrale $\tG_{r,{\rm geo}}(F)$ de
$\G_r(F)$ par $k^*$. \`A l'aide de la décomposition de Bruhat, on définit une section $s_{\rm geo}$ de cette extension, de sorte qu'on définit une fonction $\underline{\kappa}$ globale analogue à celle introduite ci-dessus (cette fonction est le produit des fonctions $\underline{\kappa}_v$ locales, i.e $\underline{\kappa}=\prod_{v\nshortmid a_r}N_{k_v/k}\underline{\kappa}_v$). De cette manière, on obtient le morphisme $\underline{\kappa}: Y(t)\to \mg_m$ (comme ${}^tN_r(F)tN_r(F)\cap \mathfrak{gl}_r(\0)\subset {}^tN_r(F)tN_r(F)\cap \G_r(\0[a_r^{-1}])$).

En utilisant la formule des traces de Grothendieck et Leftschetz, on obtient l'interprétation suivante de l'identité \ref{iintro2}
\begin{proposition}\label{intro3}
$$\mathrm{R}\Gamma_c({Y}(t)\otimes_k\overline{k},{h'}_{\underline{\alpha}}^{*}\mathcal{L}_{\psi}\otimes\underline{\kappa}^{*}\mathcal{L}_{\zeta})=\bigotimes_{\lambda\in {\rm supp}(t)}\mathrm{R}\Gamma_c({Y}_{v}(t)\otimes_k\overline{k},{h'}_{\underline{\alpha},v}^{*}\mathcal{L}_{\psi}\otimes\underline{\kappa}_{v}^*\mathcal{L}_{\zeta}),$$
o\`u ${\rm supp}(t)$ est l'ensemble des racines de $\prod_{i=1}^{r-1}a_i$ dans $\overline{k}$.
\end{proposition}

Soient $Q_{d_i}$ la vari\'et\'e affine sur ${k}$ des polyn\^omes unitaires de degr\'e $d_i$ et $V_{\ud}=\{(a_1,\dots,a_r)\in \prod_{i=1}^rQ_{d_i}|{\rm pgcd}(\prod_{i=1}^{r-1}a_i,a_r)=1\}$ avec $\ud=(d_1,\dots,d_r)$ (la raison pour laquelle on considère cet ouvert est que les sommes locales en les places divisant $a_r$ y sont triviales, du côté de l'intégrale $I$ comme du côté de l'intégrale $J$). Soit $t=\D(a_1,a_2/a_1,\dots,a_r/a_{r-1})$ tel que $(a_1,\dots,a_r)\in V_{\ud}$ et $\underline{\alpha}=(\alpha_2,\dots,\alpha_r)\in (k^*)^{r-1}$. Le couple $({X}(t),h_{\underline{\alpha}})$ et le triple $({Y}(t),h'_{\underline{\alpha}},\underline{\kappa})$ se mettent en familles  de  sorte  qu'on  obtient  des  vari\'et\'es ${X}_{\ud}$ et ${Y}_{\ud}$   de  type  fini  sur
 ${k}$ munies de morphismes $f_{\ud}^X:{X}_{\ud}\times\mg_m^{r-1}\ra  V_{\ud}\times\mg_m^{r-1}$, $f_{\ud}^Y:{Y}_{\ud}\times\mg_m^{r-1}\ra V_{\ud}\times\mg_m^{r-1}$, $h_{\ud}:{X}_{\ud}\times\mg_m^{r-1}\ra\mg_a$, $h'_{\ud}:{Y}_{\ud}\times\mg_m^{r-1}\ra\mg_a$ et $\underline{\kappa}_{\ud}:{Y}_{\ud}\times\mg_m^{r-1}\ra\mg_m$ tels que ${X}(t)$ et $ \mathrm{R}\Gamma_c({X}(t),h_{\underline{\alpha}}^*\mathcal{L}_{\psi})$ (resp.  ${Y}(t)$ et $ \mathrm{R}\Gamma_c({Y}(t),{h'}_{\underline{\alpha}}^{*}\mathcal{L}_{\psi}\otimes\underline{\kappa}^{*}\mathcal{L}_{\zeta})$) sont respectivement les fibres en $t$ de $f_{\ud}^X$ et de $\mathrm{R}f^X_{\ud,!}h_{\ud}^*\mathcal{L}_{\psi}$ (resp. de $f_{\ud}^Y$ et de $\mathrm{R}f^Y_{\ud,!}({h'}_{\ud}^{*}\mathcal{L}_{\psi}\otimes\underline{\kappa}_{\ud}^{*}\mathcal{L}_{\zeta}$)).

Les fibres de ${\cal I}_{\ud}=\mathrm{R}f^X_{\ud,!}h_{\ud}^*\mathcal{L}_{\psi}$ (resp. ${\cal J}_{\ud}=\mathrm{R}f^Y_{\ud,!}({h'}_{\ud}^{*}\mathcal{L}_{\psi}\otimes\underline{\kappa}_{\ud}^{*}\mathcal{L}_{\zeta})$ au-dessus des points $(t,\underline{\alpha})$ tels que le polynôme $\prod_{i=1}^ra_i$ n'a pas de racine multiples se factorisent en produits tensoriels des complexes locaux ``simples" considérés dans la proposition~\ref{intro2}, lesquels sont des $\overline{\mathbb{Q}}_\ell$-espaces vectoriels de rang $2$ placée en degré $0$ (resp. en degré $1$). De plus, les points $(t,\underline{\alpha})$ de cette espèce forment un ouvert $U_{\ud}$ de $\prod_{i=1}^rQ_i\times (\mg_m)^{r-1}$, et les restrictions de ${\cal I}_{\ud}$ et ${\cal J}_{\ud}$ à cet ouvert sont des systèmes locaux de rang $2^{\deg(\prod_{i=1}^ra_i)}$ placés respectivement en degré $0$ et $\deg(\prod_{i=1}^ra_i)$. Les deux restrictions sont reliées par un système local ${\cal T}_{\ud}$ de rang $1$, placé en degré $\deg(\prod_{i=1}^ra_i)$ au-dessus de $U_{\ud}$, géométriquement constant et provenant d'un certain caractère $\tau$ de ${\rm Gal}_{\overline{k}/k}$.
\begin{proposition}\label{intro4}
Pour $\ud=(d_1,\dots,d_r)$, ${\cal T}_{\ud}$ est géométriquement constant et est défini par le caractère $\tau$ de $\mathrm{Gal}_{\overline{k}/k}$ :
$$\tau(\mathrm{Fr}_q)=\left\{\begin{smallmatrix}
(-1)^{\sum_{i=1}^r d_i}q^{\sum_{i=1}^rd_i/2}\zeta(-1)^{\sum_{i=0}^{s-1}d_{2i+1}}\gamma_\infty(\varpi,\Psi_\infty)^{-\sum_{i=1}^{s-1}p(d_{2i}-d_{2i+1})}\,
\text{si $r=2s$},\\
(-1)^{\sum_{i=1}^r d_i}q^{\sum_{i=1}^rd_i/2}\zeta(-1)^{\sum_{i=1}^{s}d_{2i}}\gamma_\infty(\varpi,\Psi_\infty)^{-\sum_{i=1}^{s}p(d_{2i}-d_{2i-1})}\,
\text{si $r=2s+1$},
\end{smallmatrix}\right.$$
où $p(x)=\begin{cases} 1&\text{si $x$ est impair}\\
0&\text{si $x$ est pair}.\end{cases}$

\end{proposition}
La formule de $\tau$ vient du fait que le produit des constantes de Weil en toutes les places du corps global $F$ est trivial \cite{W}. 
\`A l'aide du théorème de Chebotarev, le système local $\mathcal{T}_{\ud}$ est géométriquement constant et se prolonge alors de manière évidente à $V_{\ud}\times \mg_m^{r-1}$. 


Dans le cas où $\ud=(1,2,\dots,r)$, d'après \cite{N}, on trouve des présentations simples pour $(X_{\ud},h_{\ud})$  et pour $(Y_{\ud},h'_{\ud},\underline{\kappa}_{\ud})$ (voir les propositions \ref{ngo1} et \ref{ngo2}) qui vont nous simplifier les calculs. On obtient alors le théorème suivant (qui peut être considéré comme un analogue global du lemme fondamental de Jacquet-Mao) :
\begin{thmB}\label{intro7} Pour $\ud=(1,2,\dots,r)$,
$\mathcal{J}_{\ud}={\cal T}_{\ud}\otimes \mathcal{I}_{\ud}$. Les deux membres de cette \'egalit\'e sont, \`a d\'ecalage pr\`es, des faisceaux pervers isomorphes au prolongement interm\'edi\-aire de leur restriction \`a $U_{\ud}$.
\end{thmB}

Ce théorème résulte immédiatement du suivant :

\begin{theorem}\label{intro5}\begin{enumerate}
\item Le complexe de faisceaux
 ${\cal I}_{\ud}[\frac{r(r+1)}{2}+r-1]$ est un faisceau pervers, prolongement       interm\'ediaire de sa restriction \`a l'ouvert $U_{\ud}$.
\item Le complexe de faisceaux
${\cal J}_{\ud}[r^2+r-1]$ est un faisceau pervers, prolonge\-ment interm\'ediaire de sa restriction \`a l'ouvert $U_{\ud}$.
\end{enumerate}
\end{theorem}
Le point $(1)$ du th\'eor\`eme~\ref{intro5} s'obtient assez facilement par l'argument de \cite[``le pas de récurence", p.515]{N} en rempla\c cant $\G_i$ par le groupe orthogonal associ\'e \`a la forme quadratique $(x_1,\dots,x_i)\mapsto\sum_{j=1}^ix_j^2$. En revanche, on ne peut pas adapter directement l'argument de loc. cit. pour démontrer le point $(2)$ car la fonction 
$\kappa_{\ud}$ n'est pas invariante sous l'action de $\G_{r-1}$. Le théorème suivant est le point crucial pour résoudre cette difficulté (en remarquant que $Y_{\ud} \simeq \{y+\varpi{\rm Id}_r,\, y\in \mathfrak{gl}_r|{\rm pgcd}(\prod_{i=1}^{r-1}a_i(y),a_r(y))\break=1\}$, où $a_i(y)=\det(s_i(y)+\varpi{\rm Id}_i)$, et où $s_i(y)$ est la sous-matrice de $y$ constituée des $i$ premières lignes et des $i$ premières colonnes, cf. \ref{ngo1} et \ref{ngo2}).
\begin{theorem}\label{intro6} $\underline{\kappa}_{\ud}(w_0(y+\varpi{\rm Id}_r))$ est en fait un polyn\^ome en les coefficients de la matrice $y$. De plus on a : $\underline{\kappa}_{\ud}(w_0(y+\varpi{\rm Id}_r))\underline{\kappa}_{\ud}(w_0{}^t(y+\varpi{\rm Id}_r))=(-1)^{\sum_{i=1}^{r-1}(i+i(i+1))}{\rm result}(a_{r-1}(y),a_r(y))$.
\end{theorem}
La d\'emonstration de ce th\'eor\`eme repose sur l'interprétation géométrique de l'extension de Kazhdan-Patterson \'evoqu\'ee ci-dessus. Ce th\'eor\`eme implique que $\underline{\kappa}(y)=\underline{\kappa}_{\ud}(w_0(y+\varpi{\rm Id}_r))$ est alors un produit de facteurs irr\'educ\-tibles de ${\rm result}(a_{r-1}(y),a_r(y))$. En faisant agit $g\in \G_{r-1}$ par $$y\mapsto \D(g,1)^{-1}\,y\,\D(g,1),$$ $g$ transforme alors $\underline{\kappa}$ en la multipliant par une puissance de $\det(g)$ (pour l'argument détaillé, on renvoie à la démonstration de la perversité de $\mathcal{J}_{\ud}$ à la fin de la section 4); l'extension $G_{r-1}$ de $\G_{r-1}$ obtenue en extrayant une racine carr\'ee de $\det(g)$ laisse alors invariant le faisceau $\underline{\kappa}^*{\cal L}_{\zeta}$ et l'argument de loc. cit. s'adapte alors  en rempla\c cant $\G_i$ par l'extension $G_i$ (le plongement $G_i\hookrightarrow G_{i+1}$ est fourni aussi par $g\mapsto \D(g,1)$ puisque les deux ont m\^eme d\'eterminant).

On verra ensuite comment déduire le théorème $A'$ du théorème $B$. Les complexes locaux obtenus par les factorisations de la proposition \ref{intro3} sont d'autant plus compliqués que les multiplicités des racines du polynôme $\prod_{i=1}^{r-1}a_i$ sont grandes. D'après B.C. Ngo \cite{N}, tous les complexes locaux peuvent être obtenu en considérant seulement ${\cal I}_{\ud}$ et ${\cal J}_{\ud}$, où $\ud=(1,2,\dots,r)$ quitte à remplacer $r$ par $r+r'$ pour $r'$ assez grand. Plus précisé\-ment on considère la situation suivante (
pour simplifier les notations, désormais on va supprimer l'indice $\ud$) : 

Soit $t'\in T_s(F_{\varpi})$. D'après B.C. Ngo \cite[prop 3.5.1, p. 505]{N}, pour $r$ assez grand, il existe $t^{\circ}=\D(a_1^{\circ},a_2^{\circ}/a_1^{\circ}, \dots,a_r^{\circ}/a_{r-1}^{\circ})$, avec $(a_1^{\circ},\dots,a_r^{\circ})\in V_{(1,2,\dots,r)}(k)$ tel que ${\cal I}_\varpi(t^{\circ})\simeq {\cal I}_\varpi(t')$ et ${\cal J}_\varpi(t^{\circ})\simeq {\cal J}_\varpi(t')$.
On a $a^{\circ}_i={a'}^{\circ}_i{a''}^{\circ}_i$, o\`u ${a'}^{\circ}_i$ est \`a racines simples différentes de $0$ et o\`u  ${a''}^{\circ}_i$ sont des puissances de $\varpi$.
Soient alors $\ud'=(\deg({a'}_i^{\circ}))_i$ et $\ud''=(\deg({a''}_i^{\circ}))_i$. On fait varier $(a'_i)_i$ et $(a''_i)_i$ en introduisant l'ouvert $(V_{\ud'}\times V_{\ud''})^{{\rm dist}}$ de $V_{\ud'}\times V_{\ud''}$ au-dessus duquel ${\rm pgcd}(\prod_{i=1}^ra'_i,\prod_{i=1}^ra''_i)=1$ et les $a'_i$ sont \`a racines simples.  On a alors un morphisme \'etale $\mu:(V_{{\ud}^{ '}}\times V_{{\ud}^{''}})^{{\rm dist}}\ra V_{\ud}$.

Pour $(t',t'')=(\D(a'_1,a'_2/a'_1,\dots,a'_r/a'_{r-1}),\D(a''_1,a''_2/a''_1,\dots,a''_r/a''_{r-1}))$ où $((a'_i)_i,(a''_i)_i)\in (V_{\ud'}\times V_{\ud''})^{\rm dist}$, on g\'en\'eralise les sommes globales $I$ et $J$ ci-dessus en introduisant $$X_1(t',t'')=\prod_{w|a'_1\dots a'_{r-1}}{\rm Res}_{k_w/k}X_w(t)\text{ et }X_2(t',t'')=\prod_{w|a''_1\dots a''_{r-1}}{\rm Res}_{k_w/k}X_w(t) $$
($Y_1(t',t'')$ et $Y_2(t',t'')$ sont d\'efinies par des formules analogues). Celles-ci se mettent en familles quand $\ud'$ et $\ud''$ sont fixées.
On d\'efinit de cette mani\`ere des complexes ${\cal I}_1,{\cal I}_2,{\cal J}_1$ et ${\cal J}_2$ v\'erifiant $\mu^*{\cal I}={\cal I}_1\otimes^{{\mathbb L}}{\cal I}_2$ et $\mu^*{\cal J}={\cal J}_1\otimes^{{\mathbb L}}{\cal J}_2$.
En fait ${\cal I}_1$ et ${\cal J}_1$ sont des syst\`emes locaux. De plus, en utilisant la formule du produit des constantes de Weil, on définit un système local $\mathcal{T}_1$ (ce système local n'est plus géométriquement constant) tel que ${\cal J}_1={\cal T}_1\otimes{\cal I}_1$ ; on pose ${\cal T}_2={\cal T}\otimes {\cal T}_1^{\otimes -1}$.

 \`A l'aide du théorème B et des propriétés des faisceaux pervers et du prolongement intermédiaire,
on obtient que ${\cal I}_2$ et ${\cal J}_2$ sont pervers et prolongement interm\'ediaire de leur restriction \`a l'ouvert $\mu^*U_{\ud}$. De plus, on a ${{\cal J}_2}_{|\mu^*U_{\ud}}={{\cal T}_2}_{|\mu^*U_{\ud}}\otimes{{\cal I}_2}{|\mu^*U_{\ud}}$, de sorte qu'on a
${\cal J}_2={\cal T}_2\otimes{\cal I}_2$.

En sp\'ecialisant en $t=t^{\circ}$ on obtient alors le th\'eor\`eme $A'$.


\vskip 5mm
Voici maintenant le plan détaillé de ce travail. La section 2 contient la définition de la somme locale $I_\varpi(t,\underline{\alpha})$ et de la somme globale $I(t,\underline{\alpha})$ pour tout $\underline{\alpha}\in (k^*)^{r-1}$. On démontre que la somme globale est le produit des sommes locales
$$I(t,\underline{\alpha})=\prod_{v\in {\rm Spm}(\0)}I_v(t,\underline{\alpha}).$$ On expose l'interprétation géométrique de la somme locale et de la somme globale et exprime géométriquement la formule de produit ci-dessus.
On construit aussi le complexe $\mathrm{R}f^X_{\ud,!}h_{\ud}^*\mathcal{L}_{\psi}$ et on démontre qu'il est pervers quand $\ud=(1,2\dots,r)$ (voir le point $(1)$ du théorème~\ref{intro5}).

Dans la section 3, on rappelle la définition du groupe métaplectique local de Kazhdan-Patterson. L'extension de ACK et son scindage au-dessus $N_r(F_\varpi)$ et au-dessus de $\G_r(\0_\varpi)$ sont introduits en 3.1.2 . La construction géométrique du groupe métaplectique à l'aide de l'extension de ACK figure dans 3.2.3. On exprime la fonction $\underline{\kappa}$ locale géométriquement dans 3.1.4. Le 3.2 contient la définition du groupe métaplectique global et de la fonction $\kappa$ globale. On montre dans 3.2.2 que le fonction $\underline{\kappa}$ globale est un produit de fonctions $\underline{\kappa}$ locales.

La section 4 contient la définition de la somme locale $J_\varpi(t,\underline{\alpha})$ et de la somme globale $J(t,\underline{\alpha})$ pour tout $\underline{\alpha}\in (k^*)^{r-1}$. On démontre que la somme globale est le produit des sommes locales
$$J(t,\underline{\alpha})=\prod_{v\in {\rm Spm}(\0)}J_v(t,\underline{\alpha}).$$ On expose l'interprétation géométrique de la somme locale et de la somme globale et exprime géométriquement la formule de produit ci-dessus.
On construit aussi le complexe
$\mathrm{R}f^Y_{\ud,!}({h'}_{\ud}^{*}\mathcal{L}_{\psi}\otimes\underline{\kappa}_{\ud}^{*}\mathcal{L}_{\zeta})$. La sous-section 4.3 contient l'étude de ce complexe dans le cas particulier où $\ud=(1,2,\dots,r)$. Dans cette sous-section figure le calcul de la fonction $\underline{\kappa}$ dans ce cas particulier, ainsi que la formule cruciale $\underline{\kappa}_{\ud}(w_0(y+\varpi{\rm Id}_r))\underline{\kappa}_{\ud}(w_0{}^t(y+\varpi{\rm Id}_r))=(-1)^{\sum_{i=1}^{r-1}(i+i(i+1))}{\rm rest}(a_r(y),a_{r-1}(y))$. La perversité de $\mathrm{R}f^Y_{\ud,!}({h'}_{\ud}^{*}\mathcal{L}_{\psi}\otimes\underline{\kappa}_{\ud}^{*}\mathcal{L}_{\zeta})$ est montrée à la fin de la section 4.

Dans la section 5, on donne la formule du facteur de transfert; dans la section 6, on déduit l'énoncé local de l'énoncé global (dans cette introduction on a déjà donné une idée de leur contenu).

\textbf{Remerciements :}    J'adresse mes plus sincères remerciements à Alain Genestier pour son soutien constant durant la préparation de ce travail. Je tiens à exprimer aussi ma profonde gratitude à Bao Chau Ngo pour ses remarques. Une très belle formule de ce travail (cf. \ref{ka7}) été obtenue lorsque j'ai eu la chance de travailler avec lui à Chicago.

\section{Intégrale $I$}
\setcounter{theorem}{0}
\subsection{Somme locale}
\quad Soit $k$ un corps fini de caractéristique $p$ et à $q$ éléments. On fixe un caractère additif non trivial $\psi: k \rightarrow \overline{\mathbb{Q}}^{*}_{\ell}$, où $\ell$ est un nombre premier différent de $p$ et où $\overline{\mathbb{Q}}_{\ell}$ est une clôture algébrique de $\mathbb{Q}_{\ell}$. On note $\0_{\varpi}=k[[\varpi]]$ l'anneau des séries formelles en une indéterminée $\varpi$ et à coefficients dans $k$, $F_{\varpi}=k((\varpi))$ son corps des fractions. On notera $\Psi$ le caractère de $\fp$ défini par $\Psi(x)=\psi(\mathrm{res}(x\mathrm{d}\varpi))$.

Pour tout $\underline{\alpha}=(\alpha_2,\dots,\alpha_r)\in {(k^*)}^{r-1}$, on note $\theta_{\underline{\alpha}} : N_r(\fp)\rightarrow \overline{\mathbb{Q}}^{*}_{\ell}$ le caractère défini par $\theta_{\underline{\alpha}}(n)=\Psi(\sum^{r}_{i=2}\alpha_i n_{i-1,i})$. Sa restriction à $N_r(\op)$ étant triviale, il induit une fonction $\theta_{\underline{\alpha}}$ sur $N_r(\fp)/N_r(\op)$ à valeurs dans $\overline{\mathbb{Q}}^{*}_{\ell}$.

Pour chaque $t=\D(a_1,a_2/a_1,\dots,a_r/a_{r-1})\in T_r(\fp)$, on considère l'ensemble fini (cf. \cite[proposition 1.1.2]{N})
$$X_{\varpi}(t)(k)=\{n\in N_r(\fp)/N_r(\op) | {}^t ntn \in S_r(\op) \}.$$
L'intégrale orbitale $I$ de Jacquet-Mao peut s'écrire (cf. \cite[p. 484]{N})
$$I_{\varpi}(t,\alpha)=\sum_{n \in X_{\varpi}(t)(k)}\theta_{\underline{\alpha}}(n).$$

L'ensemble $X_{\varpi}(t)(k)$ est de manière naturelle l'ensemble des points à valeurs dans $k$ d'une variété algébrique $X_{\varpi}(t)$ de type fini sur $k$. Cette variété est munie d'un morphisme $h_{\underline{\alpha}}:X_{\varpi}(t)\ra \mathbb{G}_a$ défini par $h_{\underline{\alpha}}(n)=\mathrm{res}(\sum_{i=2}^{r}\alpha_in_{i-1,i}\mathrm{d}\varpi)$.

Soit $\overline{k}$ une clôture algébrique de $k$. On note $\overline{X}_{\varpi}(t)=X_{\varpi}(t)\otimes_k\overline{k}$. Soient $\mathcal{L}_{\psi}$ le faisceau d'Artin-Schreier sur $\mathbb{G}_a$ associé au caractère $\psi$. D'après la formule des traces de Grothendieck-Lefschetz, on a :
$$I_{\varpi}(t,\alpha)=\mathrm{Tr}(\mathrm{Fr},\mathrm{R}\Gamma_c(\overline{X}_{\varpi}(t),h_{\underline{\alpha}}^*\mathcal{L}_{\psi})).$$
\subsection{Somme globale}
\setcounter{subsubsection}{1}
Dans le cas général, on ne connaît pas explicitement $\mathrm{R}\Gamma_c(\overline{X}_{\varpi}(t),h_{\underline{\alpha}}^*\mathcal{L}_{\psi})$ car la variété $\overline{X}_{\varpi}(t)$ est trop compliquée.

D'après Ngo Bao Chau, on va introduire une somme sur un corps global qui permettra de se ramener par déformation à une situation plus simple. Soient ${\0}={k}[\varpi]$ l'anneau des polynômes en une variable $\varpi$ à coefficients dans ${k}$ et ${F}$ son corps des fractions. Pour tout $x \in {F}$, on note $\mathrm{sres}(x\mathrm{d}\varpi)$ la somme des résidus en tous ses pôles à distance finie (c.à.d.  en les points de $\mathbb{A}^1=\mathrm{Spec}(\0)$).
On notera $\Psi$ le caractère de $F$ défini par $\Psi(x)=\psi(\mathrm{sres}(x\mathrm{d}\varpi))$. Pour tout $\underline{\alpha}=(\alpha_2,\dots,\alpha_{r})\in {(k^*)}^{r-1}$, on note $\theta_{\underline{\alpha}} : N_r(F)\rightarrow \overline{\mathbb{Q}}^{*}_{\ell}$ le caractère défini par $\theta_{\underline{\alpha}}(n)=\Psi(\sum^{r}_{i=2}\alpha_i n_{i-1,i})$. Sa restriction à $N_r(\0)$ étant triviale, il induit une fonction $\theta_{\underline{\alpha}}$ sur $N_r(F)/N_r(\0)$ à valeurs dans $\overline{\mathbb{Q}}^{*}_{\ell}$. Pour tous $t\in T_r({F})$ et $\underline{\alpha}\in({k}^*)^{r-1}$, on construit un couple  $({X}(t),h_{\underline{\alpha}}) $ de manière analogue à \cite[proposition 3.1.1]{N} où ${X}(t)$ est une variété de type fini sur ${k}$ dont l'ensemble des ${k}$-points est
$${X}(t)({k})=\{n\in N_r({F})/N_r({\0}) | {}^t ntn \in S_r(\0) \},$$
munie d'un morphisme $h_{\underline{\alpha}}(n)=\sum_{i=2}^{r}\alpha_i\mathrm{sres}(n_{i-1,i}\mathrm{d}\varpi)$. L'intégrale orbitale $I$ globale est alors :
$$I(t,\alpha)=\sum_{n \in X(t)(k)}\theta_{\underline{\alpha}}(n).$$

Pour toute place $v$, on note ${\0}_v$ le complété de ${\0}$ en $v$, ${F}_v$ son corps des fractions, et $k_v$ son corps résiduel. On note $n_v$ l'image de $n$ dans $N_r(F_v)/N_r(\0_v)$. Pour $t=\D(a_1,\frac{a_2}{a_1},\dots,\frac{a_r}{a_{r-1}})\in T_r({F}_v)$, on introduit la variété de type fini dont l'ensemble de $k$ points est
$$X_v(t)(k)=\{n\in N_r({F_v})/N_r({\0_v}) | {}^t ntn \in S_r(\0_v) \}.$$ Cette variété est munie d'un morphisme $h_{\underline{\alpha},v}(n)=\sum_{i=2}^{r}\alpha_i \mathrm{tr}_{k_v/k}(\mathrm{res}(n_{i-1,i}))$. Lorsque $v=\varpi$ et $t\in T_r(\fp)$ on retrouve bien le couple $(X_{\va}(t),h_{\underline{\alpha},\va})$. L'intégrale orbitale $I$ de Jacquet-Mao en la place $v$ est :
 $$I_v(t,\underline{\alpha})=\sum_{n_v\in X_v(t)(k)}\theta_{\underline{\alpha}}(n_v),$$
où $\theta_{\underline{\alpha}}(n_v)=\psi(\sum_{i=2}^{r}\alpha_i\mathrm{tr}_{k_v/k}(\mathrm{res}(n_{i-1,i})))$.

On note $\mathrm{supp}(t)=\mathrm{Spm}(\0/\prod_{i=1}^{r-1} a_i)$.
\begin{lemma}
Si $v\not \in \mathrm{supp}(t)$, alors $I_v(t,\underline{\alpha})=1$.
\end{lemma}
\begin{proof}[Démonstration] D'après \cite[corollaire 1.1.5]{N}, $X_v(t)(k)$ est réduit à l'élément $n={\rm Id_r}$ (car $a_i\in \0_v^*\ \forall i\in\{1,\dots,r-1\}$), de sorte qu'on a $I_v(t,\underline{\alpha})=1$.
\end{proof}La somme globale est reliée aux sommes locales par la formule de produit suivante (cf. \cite[proposition 1.3.2]{N})
\begin{equation}\label{i1}
I(t,\underline{\alpha})=\prod_{v\in \mathrm{supp}(t)}I_v(t,\underline{\alpha}).
\end{equation}

On ajoute une barre pour indiquer le changement de corps de $k$ à $\overline{k}$. On peut définir les couples $(\overline{X}(t),h_{\underline{\alpha}})$ et $(\overline{X}_v(t),h_{\underline{\alpha},v})$, où $v\in \mathrm{Spm}(\overline{\0})$. On note encore $\mathrm{supp}(t)=\mathrm{Spm}(\overline{\0}/\prod_{i=1}^{r-1}a_i\overline{\0})$.On a alors la forme cohomologique suivante de~\ref{i1} (cf. \cite[corollaire 3.2.3]{N}) :
$$\mathrm{R}\Gamma_c(\overline{X}(t),h_{\underline{\alpha}}^*\mathcal{L}_{\psi})=\bigotimes_{v\in \mathrm{supp}(t)}\mathrm{R}\Gamma_c(\overline{X}_{v}(t),h_{\underline{\alpha},v}^*\mathcal{L}_{\psi}).$$

Supposons $t=\D(a_1,\frac{a_2}{a_1},\dots,\frac{a_r}{a_{r-1}})$, où les $a_i$ sont des polynôme unitaires dont on fixera les degrés $d_i=\deg(a_i)$. Soient $Q_{d_i}$, la variété affine sur ${k}$ des polynômes unitaires de degré $d_i$ et $Q_{\underline{d}}=\prod_{i=1}^rQ_{d_i}$ avec $\ud=(d_1,\dots,d_r)$. Le couple $({X}(t),h_{\underline{\alpha}})$ se mettent en famille  de  sorte  qu'on  obtient  une  variété ${X}_{\ud}$ de  type  fini  sur
${k}$ munie de deux morphismes $f_{\ud}^X:{X}_{\ud}\times\mathbb{G}_m^{r-1}\ra Q_{\ud}\times \mathbb{G}_m^{r-1}$ et $h_{\ud}:{X}_{\ud}\times\mathbb{G}_m^{r-1}\ra\mathbb{G}_a$ tels que $\overline{X}(t)$ et $ \mathrm{R}\Gamma_c(\overline{X}(t),h_{\underline{\alpha}}^*\mathcal{L}_{\psi})$ sont respectivement les fibres en $(t,\underline{\alpha})\in (Q_{\ud}\times \mg_m^{r-1})(\overline{k})$ de $f_{\ud}^X$ et de $\mathrm{R}f^X_{\ud,!}h_{\ud}^*\mathcal{L}_{\psi}$. Plus pr\'ecis\'ement :
\begin{lemma}(cf. \cite[proposition 3.3.1]{N})
\begin{enumerate}
\item Pour tout $\underline{d}\in \mathbb{N}^r$ le foncteur $X_{\underline{d}}$ qui associe à toute $k$-algèbre $R$ l'ensemble
    $$X_{\underline{d}}(R)=\{g\in S_r(\0\otimes_k R)|\det(g_i)\in Q_{d_i}(R)\}/N_r(\0\otimes_k R),$$
    où $g_i$ est la sous-matrice de $g$ faite des $i$-premières lignes et des $i$-premières colonnes de $g$, est représenté par une variété affine de type fini sur $k$, qu'on note aussi $X_{\underline{d}}$.
    Soit
    $$f^{X}_{\underline{d}}: X_{\underline{d}}\times \mg_m^{r-1}\ra Q_{\underline{d}}\times \mg_m^{r-1}$$
    le morphisme défini par $f^{X}_{\underline{d}}(g,\underline{\alpha})=((a_i)_{1\leq i\leq r},\underline{\alpha})$ où $a_i=\det(g_i)$.
\item Pour tout $i$ avec $2\leq i \leq r$, l'application $h_{i}:S_r(\0\otimes_kR)\times (R^*)^{r-1}\ra R$ définie par
    $$    h_{i}=\mathrm{res}\left(a_{i-1}^{-1}\left((g_{i,1},\dots,g_{i,i-1})a_{i-1}g_{i-1}^{-1}\left(\begin{matrix}0\\\vdots\\0\\\alpha_i\end{matrix}\right)\right)\right),
    $$
    où $a_ig_i^{-1}$ est la matrice des cofacteurs de $g_i$, lesquels sont dans $\0_{R}:=\0\otimes_kR$ et où $\mathrm{res}(a_{i-1}^{-1} b)$ est le coefficient de $\varpi^{d_{i-1}-1}$ dans l'expression polynomiale en la variable $\varpi$ du reste de la
    division euclidienne de $b$ par $a_{i-1}$ (cette division euclidienne a un sens puisque le coefficient dominant de $a_{i-1}$ est égal à $1$), induit un morphisme $h_{i} : X_{\underline{d}}\times \mg_m^{r-1}\ra\mg_a$.
\item Soit $h_{\underline{d}}=\sum_{i=2}^rh_{i}$. Alors pour tous $t\in Q_{\underline{d}}(k)$ et $\underline{\alpha} \in \mg_m^{r-1}$, le couple $(X(t),h_{\underline{\alpha}})$ est isomorphe à la fibre $(f^{X}_{\underline{d}})^{-1}(t,\underline{\alpha})$ munie de la restriction de $h_{\underline{d}}$ à cette fibre.
\end{enumerate}
\end{lemma}
\subsection{Le cas $\emph{d}=(1,2,\dots,r)$}
\setcounter{theorem}{0}
\setcounter{subsubsection}{1}
Soit $U_{\ud}$ l'ouvert de $\prod_{i=1}^rQ_i\times (\mg_m)^{r-1}$ form\'e des couples $(t,\underline{\alpha})$ tels que le polyn\^ome $\prod_{i=1}^ra_i$ n'ait pas de racines multiples.
\begin{theorem}\label{px}
Pour $\ud=(1,2,\dots,r)$ le complexe de faisceaux \\$\mathrm{R}f^X_{\ud,!}h_{\ud}^*\mathcal{L}_{\psi}[\frac{r(r+1)}{2}+r-1]$ est un faisceau pervers, prolongement intermédiaire de sa restriction à l'ouvert $U_{\ud}$.
\end{theorem}
On rappelle que $S_r$ désigne l'espace affine des matrices symétriques de taille $r$.
\begin{proposition} (cf. \cite[proposition 4.2.1]{N})\label{ngo1}
Pour $\underline{d}=(1,2,\dots,r)$ le triplet $(X_{\underline{d}},f^{X}_{\underline{d}},h_{\underline{d}})$ est isomorphe au triplet $(S_r,f^X,h)$ où le morphisme $f^X:S_r\times \mg_m^{r-1}\ra\prod_{i=1}^rQ_i\times\mg_m^{r-1}$ est défini par
$$f^X(x,\underline{\alpha})=(a_1(x),\dots,a_r(x),\underline{\alpha}),$$ où $a_i(x)=\det(s_i(x)+\varpi {\rm Id}_i)$, $s_i(x)$ étant la sous-matrice faite des $i$ premières lignes et des $i$ premières colonnes de $x$, et où le morphisme $h:S_r\times \mg_m^{r-1}\ra \mg_a$ est défini par
$$h(x,\underline{\alpha})=\sum_{i=2}^r\alpha_ix_{i-1,i}.$$
\end{proposition}
\begin{proof}[Démonstration] Soit $R$ une $k$-algèbre. On note $\0_R:=\0\otimes_k R=R[\varpi]$.

On montre d'abord que l'on peut réduire toute matrice $g\in S_r(\0_R)$ dont le déterminant de la sous-matrice $s_i(g)$ est un polynôme unitaire de degré $i$ à une matrice  de la forme $x+\va {\rm Id}_r$, avec $x$ une matrice symétrique à coefficients dans $R$, par récurrence sur $r$. L'assertion est évidente pour $r=1$. D'après l'hypothèse de récurrence, on peut supposer que $s_{r-1}(g)=x_{r-1}+\va {\rm Id}_{r-1}$ (à l'aide de l'action du groupe $N_{r-1}$, qu'on voit comme le sous-groupe de $N_r$  formé des matrices unipotentes dont les coefficients non diagonaux de la $r$-ième colonne sont nuls). On écrit
$$g=\begin{pmatrix}
s_{r-1}(g)&y\\
{}^ty& z
\end{pmatrix},$$
où $y$ est un vecteur colonne appartenant à ${\0_R}^{r-1}$ et $z$ un élément de ${\0_R}$.
Comme pour toute matrice $x$ à coefficient dans $R$ le $R$-module ${\0_R}^r$ se décompose en une somme directe $R^r\oplus(x+\va {\rm Id}_r){\0_R}^r$, (cf. \cite[lemme 4.1.2]{N}) il existe un unique vecteur $v$ tel que $y+s_{r-1}(g)v\in R^{r-1}$.
En considèrant la matrice $u_r=\begin{pmatrix}{\rm Id}_{r-1}&v\\0&1\end{pmatrix}$ on a alors ${}^tu_rgu_r=x+\va {\rm Id}_r$, où $x$ est une matrice symétrique à coefficients dans $R$.
De plus, si $x+\va {\rm Id}_r={}^tntn$, on a nécessairement $\mathrm{sres}(n_{i-1,i})=x_{i-1,i}$.
\end{proof}

Soit $S_i$ la variété affine des matrices symétriques de taille $i$. On définit les variétés $R_i$ en posant $R_r=\mathrm{Spec}(k)$ et $R_{i-1}=R_i\times Q_i \times \mg_m$. Soit $f^X_i:\mg_m\times S_i\times R_i\ra S_{i-1}\times R_{i-1}$ le morphisme défini par $f^X_i(\alpha_i,x_i,r_i)=(s_{i-1}(x_i),r_{i-1})$, où $r_{i-1}=(r_i,\Delta_i(x_i),\alpha_i)$. Soit $h_i :\mg_m\times S_i \times R_i\ra \mg_a$ le morphisme défini par $h_i(\alpha_i,x_i,r_i)=\alpha_ix_{i-1,i}$. Soit $\mathrm{pr}_i:\mg_m\times S_i\times R_i\ra S_i\times R_i$ la projection évidente. On définit les complexes $\mathcal{I}_i$ sur $S_i\times R_i$ en posant $\mathcal{I}_r=\overline{\mathbb{Q}}_{\ell}[\frac{r(r+1)}{2}]$ et $\mathcal{I}_{i-1}=Rf^X_{i,!}(\mathcal{I}_i\otimes h^*_{i}\lsi)$ (voir loc. cit.). On a un isomorphisme $S_1\times R_1 \simeq Q_{\ud}\times\mg_m^{r-1}$ pour lequel $\mathcal{I}_1\simeq \mathrm{R}f^X_{\ud,!}h_{\ud}^*\mathcal{L}_{\psi}[\frac{r(r+1)}{2}+r-1]$.

On note $\mathrm{O}_i$ le groupe orthogonal de degré $i$ de la forme quadratique $q(x_1,\dots,x_i)=\sum_{j=1}^ix_j^2$. Ce groupe agit dans $\mg_m\times S_i\times R_i$ par l'action adjointe sur $S_i$ et par l'action triviale sur les autres facteurs. On identifiera le groupe $\mathrm{O}_{i-1}$ au sous-groupe $\D(\mathrm{O}_{i-1},1)$ de $\mathrm{O}_i$ de sorte que $\mathrm{O}_{i-1}$ agit aussi sur $\mg_m\times S_i\times R_i$ par l'action induite. Comme l'action adjointe laisse invariant le polynôme caractéristique, le morphisme $f^X_i$ est $\mathrm{O}_{i-1}$-équivariant. Le morphisme $h_i$ n'est pas $\mathrm{O}_{i-1}$-équivariant mais est néanmoins $\mathrm{O}_{i-2}$-équivariant. Le théorème~\ref{px} résulte alors de la proposition suivante
\begin{proposition}(cf. \cite[proposition 5.2.2]{N})
Soient $U_i$ et $U_{i-1}$ les images réciproques de $U_{\ud}$ dans $S_i\times R_i$ et dans $S_{i-1}\times R_{i-1}$. Si $\mathcal{I}$ est un faisceau pervers sur $S_i\times R_i$, $\mathrm{O}_{i-1}$-équivariant et isomorphe au prolongement intermédiaire à restriction à l'ouvert $U_i$, alors
$$\mathcal{I}'=Rf^X_{i,!}(\mathcal{I}\otimes h^*_{i}\lsi)[1]$$ est aussi un faisceau pervers sur $S_{i-1}\times R_{i-1}$, $\mathrm{O}_{i-2}$-équivariant et isomorphe au prolongement intermédiaire de sa restriction à l'ouvert $U_{i-1}$.
\end{proposition}
\begin{proof}[Démonstration]
Les morphismes qui interviennent dans la formation de $\mathcal{I}'$ sont tous $\mathrm{O}_{i-2}$-équivariants donc $\mathcal{I}'$ l'est aussi.

On utilise la transformation de Fourier-Deligne (\cite{L}) pour démontrer la perversité et le prolongement intermédiaire. Soit $E=S_i\times Q_i\times R_i$. Il est clair que $E\times \mg_m\simeq S_{i-1}\times R_{i-1}$. Soient $V$ le fibré trivial $E\times \mathds{A}^{i-1}$ et $V^{\vee}$ son fibré dual. On note $\iota:S_i\times R_i\ra V$ l'immersion fermée définie par $\iota(x_i,r_i)=(x_{i-1},a_i(x_i),r_i,y)$, où $x_i$ est de la forme
$\left(\begin{smallmatrix}
x_{i-1}&y\\
{}^ty&*
\end{smallmatrix}\right)$. On note $\epsilon :E\times\mg_m\ra V^{\vee}$ l'immersion fermé définie par $\epsilon(e,\alpha_i)=(e,{}^t(0,\dots,0,\alpha_i))$. On vérifie que $\mathcal{I}'=\epsilon^*\fsi(\iota_*\mathcal{I})[2-i]$ (cf. \cite[proposition 5.3.2]{N}), où $\fsi$ est la transformation de Fourier-Deligne. Puisque $\mathcal{I}$ est un faisceau pervers et $\iota$ est une immersion fermée, $\iota_*\mathcal{I}$ est un faisceau pervers. D'après \cite{L}, $\fsi(\iota_*\mathcal{I})$ en est un aussi. L'action de $\mathrm{O}_{i-1}$ sur $S_i\times R_i$ s'\'etend à $V$ de la manière suivante :
$$\pi(g,(x_{i-1},a_i,r_i),y)=(({}^tgx_{i-1}g,a_i,r_i),{}^tgy).$$
Cela induit donc une action sur $V^{\vee}$
$$\check{\pi}(g,(x_{i-1},a_i,r_i),\check{y})=(({}^tgx_{i-1}g,a_i,r_i),{}^tg\check{y}).$$
Par rapport à cette action, $\fsi(\iota_*\mathcal{I})$ est $\mathrm{O}_{i-1}$-équivariant.
On va utiliser le lemme suivant
\renewcommand{\qedsymbol}{}
\end{proof}
\begin{lemma}(cf. \cite[lemme 5.4.3]{N})\label{i2}
Le morphisme composé
$$\xymatrix{\mathrm{O}_{i-1}\times E\times \mg_m\ar@{^{(}->}[r]^-{\epsilon}&\mathrm{O}_{i-1}\times V^{\vee}\ar[r]^-{\check{\pi}}&V^{\vee}}$$ est un morphisme lisse de dimension relative $\frac{i(i-1)}{2}+1-(i-1)$.
\end{lemma}
\begin{proof}[Démonstration du lemme]
On note $Z$ l'image de l'immersion localement fermée $\epsilon$ dans $V^{\vee}$. Le morphisme composé s'écrit donc :
$$\mathrm{O}_{i-1}\times Z\buildrel {\check{\pi}}\over{\ra} V^{\vee}.$$

En oubliant les composantes $Q_i$ et $R_i$ on obtient un diagramme cartésien évident
$$\xymatrix{
\mathrm{O}_{i-1}\times Z\ar[r]^{\check{\pi}}\ar[d]&V^{\vee}\ar[d]\\
\mathrm{O}_{i-1}\times \mg_m\times S_{i-1}\ar[r]^-{\xi}&S_{i-1}\times \mathbb{A}^{i-1}
}$$
où $\xi(g_{i-1},\alpha_i,x_{i-1})=({}^tg_{i-1}x_{i-1}g_{i-1},g_{i-1}{}^t(0,\dots,0,\alpha_i))$.

En factorisant le morphisme $\zeta$ et oubliant le facteur $S_{i-1}$, l'assertion se ram\`ene à démontrer que
le morphisme
$$\mathrm{O}_{i-1}\times\mg_m\ra \mathbb{A}^{i-1},\ (g_{i-1},\alpha_i)\mapsto (g_{i-1}{}^t(0,\dots,0,\alpha_i))$$
est lisse et de dimension relative $\frac{i(i-1)}{2}+1-(i-1)$. Cela résulte de l'assertion suivante :
``L'orbite de l'\'el\'ement  ${}^t(0,\dots,0,1)\in \mathbb A^{{i-1}}$ sous l'action du  groupe $\mathrm{O}_{i-1}\times\mg_m$  sur $\mathbb{A}^{i-1}$ d\'efinie par
$$ (g_{i-1},\lambda,y)\mapsto  \lambda g_{i-1}y$$ est l'ouvert défini par l'équation ${}^tyy\neq 0$."
Cette assertion est évidente.
\renewcommand{\qedsymbol}{}
\end{proof}
\begin{proof}[Fin de la démonstration]

D'après \cite[4.2.5]{BBP}, $\check{\pi}^*(\fsi(\iota_*\mathcal{I}))[\frac{i(i-1)}{2}+1-(i-1)]$ est un faisceau pervers sur $O_{i-1}\times Z$. Grâce à la $O_{i-1}$-équivariance on a un isomorphisme
  $$\check{\pi}^*(\fsi(\iota_*\mathcal{I}))=\mathrm{pr}_Z^*(\fsi(\iota_*\mathcal{I})_{|Z}).$$
D'après loc. cit. $\fsi(\iota_*\mathcal{I})_{|Z}[1-(i-1)]$ est un faisceau pervers sur $Z$, la projection $\mathrm{pr}_Z: O_{i-1}\times Z \to Z$ étant clairement un morphisme lisse surjectif de dimension relative $\dim(O_{i-1})=\frac{i(i-1)}{2}$. Alors $\mathcal{I}'$ est un faisceau pervers sur $S_{i-1}\times R_{i-1}$ .

De la m\^eme mani\`ere, $\mathcal{I}'$ est le prolongement intermédiaire de sa restriction à $U_{i-1}$.
\end{proof}

\section{Le groupe métaplectique}
\subsection{Le groupe m\'etaplectique local}
\subsubsection{La construction du groupe m\'etaplectique local (cf. \cite{KP})}
Soit $k$ un corps fini de caract\'eristique $p\neq 2$ et \`a $q$ \'el\'ements. On note $\0_{\varpi}=k[[\varpi]]$ l'anneau des s\'eries formelles \`a une ind\'etermin\'ee $\varpi$ et \`a coefficients dans $k$, $\fp=k((\varpi))$ son corps des fractions.

Dans \cite{KP} le groupe m\'etaplectique $\tG_r$ est construit comme une extension centrale non-triviale du groupe $\G_r(\fp)$ par $\{\pm 1\}$
$$\xymatrix{1 \ar[r] &\{\pm1\} \ar[r] &\tG_r(\fp) \ar[r]& \G_r(\fp) \ar[r]& 1}. $$

L'extension $\tG_r(\fp)$ est d\'efinie par une loi de groupe sur le produit $\G_r(\fp)\times \{\pm1\}$
$$(g,z)(g',z')=(gg',zz'\chi(g,g')), $$
o\`u  $ \chi$ est un certain $2$-cocycle, ce qui veut dire qu'il v\'erifie l'identit\'e
$$\chi(g',g'')\chi(gg',g'')^{-1}\chi(g,g'g'')\chi(g,g')^{-1}=1,$$
qui traduit l'associativit\'e de la loi de groupe.


\begin{proposition}(cf. \cite[Proposition 0.1.2]{KP})\label{scindage}
L'extension centrale $\tG_r(\fp)$ est scind\'ee au-dessus de $\G_r(\0_{\varpi})$.
\end{proposition}
On a une section canonique $\kappa^*$ de l'extension $\G_r(\fp)$ au-dessus de $\G_r(\mathcal{O}_{\varpi})$ qui s'écrit $\kappa^*(g)=(g,\kappa(g))$ pour une certaine application $$\kappa: \G_r(\mathcal{O}_{\varpi})\rightarrow \{\pm1\}$$
qui satisfait les relations suivantes (c.f \cite[lemme 2]{M}):
\begin{enumerate}
\item $\kappa |_{T_r(\fp)\cap \G_r(\op)}=\kappa |_{W_r}=\kappa |_{N_r(\fp)\cap \G_r(\op)} =1$ (ceci d\'etermine $\kappa$ de mani\`ere unique).
\item $\kappa(g_1g_2)=\kappa(g_1)\kappa(g_2)\chi(g_1,g_2)$\quad  $(g_1,g_2 \in \G_r(\op))$.
\item $\kappa\left(\left[\begin{matrix}g_1&\\&g_2\end{matrix}\right]\right)=\kappa\left(\left[\begin{matrix}&g_1\\g_2&\end{matrix}\right]\right)=\kappa(g_1)\kappa(g_2)$ \ $\left(\left[\begin{matrix}g_1&\\&g_2\end{matrix}\right]\in \G_r(\op)\right)$.
\end{enumerate}

Dans le cas $r=2$ Kubota (\cite[p. 19]{Ku}) a montr\'e que
$$\kappa(g)=\begin{cases}\left[c,\frac{d}{\det(g)}\right]&\left(g=\left(\begin{matrix}a&b\\c&d\end{matrix}\right),\,c\neq 0,\,c\not\in \0_{\varpi}^*\right)\\1&(c=0 \textrm{ ou }c\in \0_{\varpi}^*), \end{cases}$$
o\`u $[.,.]$ est le symbole de Hilbert.

En rempla\c cant le symbole de Hilbert par le symbole mod\'er\'e dans la d\'efinition de $\chi$ (on note encore $\chi$ le cocycle ainsi obtenu), on obtient une extension $\tG_{r,KP}(\fp)$ de $\G_r(\fp)$ par $k^*$. Le symbole de Hilbert peut s'exprimer en termes du symbole mod\'er\'e, de sorte que cette extension $\tG_{r,KP}(\fp)$ permet de retrouver celle ci-dessus.

Le nouveau $2$-cocycle $\chi$ est bien d\'efini par la proposition suivante.
\begin{proposition}(\cite[p. 112]{M})\label{chi1}
Soit $ g \in \G_r(\fp)$. D'apr\`es la d\'ecom\-position de Bruhat, on \'ecrit $g=n_1mn_2$, o\`u $n_1,\,n_2\in N_r$ et $m\in T_rW_r(\fp)$ ($m$ est uniquement d\'etermin\'e). On note $B(g)=m$. Soit $\{.,.\}$ le symbole mod\'er\'e, i.e. $\{f,g\}=(-1)^{v(f)v(g)}\frac{f^{v(g)}}{g^{v(f)}}(0)$.  On a alors:
\begin{enumerate}
\item $\chi(t,t')=\displaystyle\prod_{i<j}\{t_i,t_j'\}$,
    \\ o\`u $t=\D[t_i]$ et $t'=\D[t_i']$.
\item $\chi(w,w')=1$\quad $(w,w' \in W_r)$.
\item $\chi(t,w)=1$\quad $(w\in W_r$, $t \in T_r(\fp))$.
\item $\chi(\alpha,t)=\displaystyle\{t_{\ell},t_{\ell+1}\}^{-1}\{-1,\frac{t_{\ell}}{t_{\ell+1}}\}\{-1,\det(t)\}$,
    \\ o\`u  $\alpha$ est la matrice de la transposition $(\ell,\ell+1)$.
\item $\chi(ng,g'n')=\chi(g,g')$ \quad $(n,n'\in N_r(\fp))$.
\item $\chi(t,g)=\chi(t,B(g))$ \quad $(t \in T_r(\fp))$.
\item La formule de loc. cit.  $\sigma(s_{\alpha},g)=\sigma(R(s_{\alpha}g)R(g))$ doit en fait \^etre corrig\'ee en $\chi(\alpha,g)=\chi(B(\alpha g)B(g)^{-1},B(g))$, comme on le voit en la comparant avec la construction de Matsumoto du groupe m\'etaplectique dans \cite[p. 40]{KP}.
\end{enumerate}
\end{proposition}

Pour comprendre g\'eom\'etriquement l'extension $\tG_{KP}(\fp)$, on la compare \`a une autre extension due \`a Arbarello, De Concini et Kac (ACK), dont la d\'efinition est plus g\'eom\'etrique.
\subsubsection{Le symbole $(A|B)$ (cf. \cite{ACK})}
Soit $V$ un $k$-espace vectoriel de dimension finie. On note $\bigwedge V=\bigwedge^{\dim V}(V)$ la puissance ext\'erieuse maximale de $V$.
Soient $(v_1,\dots,v_m)$ une base de $V$ et $({v}^*_1,\dots,{v}^*_m)$ sa base duale. L'\'el\'ement $v_1\wedge\dots\wedge v_m$ est une base de $\bigwedge V$ et on normalise l'identification $(\bigwedge V)^*\simeq \bigwedge V^*$ de telle sorte que ${(v_1\wedge\dots\wedge v_m)}^{*}={v}^*_1\wedge\dots\wedge {v}^*_m$ soit sa base duale.
\begin{proposition}\label{ack1}
Soit  $0 \ra V \ra V'' \ra V' \ra 0$ une suite exacte d'espaces vectoriels de dimension finie. On a des isomorphismes canoniques
$$\phi(V,V'): \bigwedge V \otimes \bigwedge V' \ra \bigwedge V''\, ;$$ $$\space \phi(V',V): \bigwedge V' \otimes \bigwedge V \ra \bigwedge V'',$$
d\'efinis par
$$\phi(V,V')(v_1\wedge \dots \wedge v_m \otimes v'_1\wedge\dots\wedge v'_n)=v_1\wedge\dots\wedge v_m\wedge\widetilde{v'_1}\wedge\dots\wedge\widetilde{v'_n},$$
$$\phi(V',V)(v'_1\wedge\dots\wedge v'_n\otimes v_1\wedge \dots \wedge v_m)=\widetilde{v'_1}\wedge\dots\wedge\widetilde{v'_n}\wedge v_1\wedge\dots\wedge v_m,$$
o\`u $(v_1,\dots,v_m)$ (resp. $(v'_1,\dots,v'_n)$) est une base de $V$ (resp. de $V'$) et $\widetilde{v'_1},\dots,\widetilde{v'_n}$ sont les pr\'eimages dans $V''$ de $v'_1,\dots,v'_n$.
\end{proposition}

D'apr\`es Deligne (\cite{DL}), pour \'eviter les probl\`emes de signes, on va consid\'erer $\bigwedge V$ comme l'espace vectoriel gradu\'e $\bigwedge^{\dim(V)}V$ plac\'e en degr\'e $\dim V$, qui est muni un isomorphisme de sym\'etrie
$$\bigwedge V\otimes \bigwedge V'\buildrel \thicksim \over{\ra}\bigwedge V'\otimes \bigwedge V$$ par la r\`egle de Koszul (l'isomorphisme \'evident est multipli\'e par $(-1)^{\dim V.\dim V'})$. La suite exacte $0\ra V\ra V\oplus V'\ra V'\ra 0$ donne alors un diagramme commutatif :
$$\xymatrix{
\bigwedge (V\oplus V')\ar[d]_-{\wr}&\bigwedge V\otimes \bigwedge V'\ar[d]^-{Koszul}\ar[l]_-{\phi(V,V')}\\
\bigwedge (V'\oplus V)&\bigwedge V'\otimes \bigwedge V\ar[l]^-{\phi(V',V)}
},$$
o\`u l'isomorphisme $\bigwedge (V\oplus V') \buildrel \sim \over{\ra}\bigwedge (V'\oplus V)$ vient de l'isomorphisme de sym\'etrie $V\oplus V' \buildrel \sim \over{\ra}V'\oplus V$
\begin{defn}
Soient $V$ un espace-vectoriel et $A,B$ deux sous-espaces de $V$. On dit que $A$ et $B$ sont \textit{commensurables} (et on note $A\sim B$) si et seulement si $\dim\left(\frac{A+B}{A\cap B}\right)<\infty$.
\end{defn}

On s'int\'eresse au cas o\`u $V=\fp^r$ et $A,B$ sont deux $\0_{\varpi}^r$-r\'eseaux. Il est clair que $A \sim B$. On a donc $\dim(A/A\cap B)<\infty$ et $\dim(B/A\cap B)<\infty$.
\begin{defn}
Soient $A,B$ deux $\0_{\varpi}^r$-r\'eseaux. On d\'efinit la droite  $$(A|B)=\left(\bigwedge A/A \cap B\right)^*\otimes \left(\bigwedge B/A\cap B\right).$$

\end{defn}
Clairement $(A|B)\otimes(B|A)\can k$, et si $A=B$, on a $(A|B) \can k$.
\begin{proposition}\label{ACK1}Soient $A,B,C,D$ des $\0_{\varpi}^r$-r\'eseaux.
\begin{itemize}
\item On a un isomorphisme de permutation \'evident  $$(A|B)\otimes(C|D)\ra(C|D)\otimes(A|B).$$
Pour \'eviter le probl\`eme de signe, on va voir $(A|B)$ comme une droite gradu\'ee plac\'ee en degr\'e $[A:B]=\dim(A/A\cap B)-\dim(B/A\cap B),$ et on red\'efinira l'isomorphisme de sym\'etrie :$$(A|B)\otimes(C|D)\ra(C|D)\otimes(A|B),$$
en multipliant le morphisme \'evident par $(-1)^{[A:B][C:D]}$ (r\`egle de Koszul). On notera $\mathrm{Sym}^{\bullet}$ l'isomorphisme de sym\'etrie ainsi obtenu.
\item On a un isomorphisme canonique de contraction $$\beta :(A|B)\otimes (B|C) \ra (A|C).$$

\end{itemize}
\end{proposition}
On renvoie \`a (\cite[\S4]{ACK}) pour la construction $\beta$ \`a l'aide de l'isomorphisme $\phi$ de la proposition \ref{ack1}.
\subsubsection{L'extension de ACK}
\begin{lemma} \label{n1} Soient $g\in \G_r(\fp)$ et $L$ un $\op^r$-r\'eseau. Il existe alors un isomorphisme canonique : $$\gamma:(\0_{\varpi}^r|g\0_{\varpi}^r) \can (L|gL).$$
\end{lemma}
\begin{proof}
L'isomorphisme $\gamma$ est d\'efini par le diagramme commutatif~:
$$\xymatrix{(L|gL)&(L|gL)\otimes(\0_{\varpi}^r|L)\otimes(L|\op^r)\ar[l]\\
&(\0_{\varpi}^r|L)\otimes(L|gL)\otimes(L|\op^r)\ar[u]_-{\mathrm{Sym}^{\bullet}}\\
(\0_{\varpi}^r|g\0_{\varpi}^r)\ar[uu]^{\gamma}\ar[r]_-{\beta^{-1}}&(\0_{\varpi}^r|L)\otimes (L|gL)\otimes(gL|g\0_{\varpi}^r)\ar[u]_-{\times g^{-1}}.
}
$$
\end{proof}
Arbarello, De Concini, Kac associent \`a chaque $g \in \G_r(\fp)$ une droite $D_g:=(\0_{\varpi}^r|g\0_{\varpi}^r)$.
Cette contruction fournit une extension centrale de $\G_r(\fp)$ par $k^*$
$$1 \rightarrow k^*\rightarrow \tG_r'(\fp) \rightarrow \G_r(\fp) \rightarrow 1.$$
Le groupe $\tG_r'(\fp)$ est form\'e des \'el\'ements
$$\{(g,v)|g \in \G_r(\fp),\,v\in D_g-\{0\}\}$$ et est muni de la loi de groupe (cf. \cite{ACK} pour la
v\'erification de l'associativit\'e et de l'inverse) d\'efinie par l'isomorphisme de multiplication :
$$D_g\otimes D_{g'}\buildrel {\times g} \over{\ra}(\op^r|g\op^r)\otimes(g\op^r|gg'\op^r)\buildrel\beta\over{\ra}D_{gg'}.$$\label{chi3} La restriction de cette extension \`a $\mathrm{SL}_r(\fp)$ est celle de Beauville-Laszlo \cite{BL}.

\begin{rmk}\label{rmk1}
Soit $k'$ une extension finie de $k$. On note $\op'=k'[[\va]]$ et $\fp'$ son corps de fractions. On a aussi une extension de ACK
$$1\ra (k')^*\ra \G'_{r}(\fp')\ra\G_r(\fp')\ra 1,$$
comme celle d\'efinie ci-dessus dans le cas particulier o\`u $k'= k$. En poussant cette extension via le morphisme de norme $N_{k'/k}:{k'}^*\ra k$ on obtient une extension m\'etaplectique de $\G'_r(\fp')$ par $k^*$ :
$1\ra k^*\ra N_*\tG'_{r}(\fp')\ra \G_r(\fp')\ra 1.$ Cette extension est aussi obtenue en associant \`a chaque $g\in \G_r(\fp')$ la droite ${D'}_{g}=({\op'}^r|g{\op'}^r)_k$ (o\`u
${\op'}^r$ et $g{\op'}^r$ sont vus comme $k$-espaces vectoriels).
\end{rmk}
D'apr\`es la d\'efinition de $D_g$, on a $D(g)\can k\,,\ \forall g\in\G_r(\op)$, de sorte que cette extension est scind\'ee canoniquement sur $\G_r(\op)$, i.e on a le morphisme de groupes $$\G_r(\op)\ra\tG'_r(\op)\quad g\mapsto (g,1),$$
o\`u $1$ désigne le pr\'eimage dans $D_g$ de $1\in k$ par l'isomorphisme canonique $D_g\can k\,,\ \forall \, g\in \G_r(\op)$.
\begin{lemma} \label{n2}
Soit $n \in N_r(\fp)$, il existe alors un $\op^r$-r\'eseau $M$  tel que $nM=M$.
\end{lemma}
\begin{proof}
 Supposons que $n=\left(
\begin{smallmatrix}
1&n^1_{2}&n^1_3&\dots&n^1_r\\
0&1&n^2_3&\dots&n^2_r\\
\vdots &&\ddots&&\vdots\\
0&&\dots& &1
\end{smallmatrix}
\right)$.
On consid\`ere le $\op^r$-r\'eseau $M$ de la forme $M=\D(\beta_1,\dots,\beta_r)\op^r$. On a alors
$$nM=\left(\begin{matrix}\beta_1\op+n^1_2\beta_2\op+n^1_3\beta_3\op\dots +n^1_r\beta_r\op\\
\beta_2\op+n^2_3\beta_3\op+\dots+n^2_r\beta_r\op\\
\vdots\\
\beta_r\op\end{matrix}\right).$$

On choisit les $\beta_i\in \fp$ de telle sorte que \`a $v(\beta_{i-1})\ll v(\beta_{i})\,,\ \forall i=2,\dots,r$. Cela implique $\beta_i\op+\sum_{j=i+1}^rn^i_j\beta_j\op=\beta_i\op\,,\ \forall i$. On en d\'eduit : $nM=M$.
\end{proof}

\begin{lemma}\label{n3}
Soient $M$, $M'$ deux $\op^r$-r\'eseaux tels que   $nM=M$ et $nM'=M'$. Le diagramme
$$\xymatrix{(M|nM)\ar[rr]^-{(~\ref{n1})}\ar[dr]_-{\textrm{\'evident}}&&(M'|nM')\ar[dl]^-{\textrm{\'evident}}\\
&k.
}$$
est alors commutatif.
\end{lemma}
\begin{proof}
On peut supposer que $M'\subset M$ (comme on le voit en consid\'erant le r\'eseau $M''=M\cap M''$, qui v\'erifie lui aussi $nM''=M''$).

L'assertion r\'esulte imm\'ediatement du fait que l'automorphisme $(M'|M)\break\buildrel{\times n}\over{\ra}(nM'|nM)=(M'|M)$ est l'identit\'e.  Ceci r\'esulte du fait que l'automorphisme $\times n$ de $M/M'$ est unipotent.
\end{proof}

D'après le lemme~\ref{n2}, on peut choisir un $\op^r$-réseau M qui satisfait $nM=M$, ce qui induit un isomorphisme $D_n\can k$ (qui est l'isomorphimse composé de $D_n\buildrel ~\ref{n1}\over{\ra}(M|nM)$ et de $(M|nM)\can k$). On note $\mathfrak{d}_n$ l'image réciproque de $1\in k$ dans $D_n$ par cet isomorphisme. \`A l'aide du lemme~\ref{n3}, on voit que $\mathfrak{d}_n$ ne dépend pas du choix de $M$.
\begin{lemma}\label{n4}
La fonction $\sigma_N:N_r(\fp)\ra\widetilde{N}'_r(\fp)\ n\mapsto (n,\mathfrak{d}_n)$ est un morphisme de groupes, i.e l'extension de ACK est scindée au-dessus de $N_r(\fp)$.
\end{lemma}
\begin{proof}[Démonstration]Soient $n,n'\in N_r(\fp).$ On choisit $M=\D(\beta_1,\dots,\beta_r)\0_\varpi^r$ tel que $v(\beta_i)\ll v(\beta_{i+1})$, de telle sorte que : $nM=n'M=nn'M=M$. Le lemme résulte immédiatement de $((\op^r|M)\otimes(nM|n\op^r))\otimes ((\op^r|M')\otimes(nM'|n\op^r))\can ((\op^r|M)\otimes(nn'M|nn'\op^r))$ (à l'aide de l'isomorphisme canonique $D_n\can((\op^r|M)\otimes(nM|n\op^r))$ du lemme~\ref{n1}, ici avec  $g=n$).
\end{proof}

\subsubsection{Construction geom\'etrique du groupe m\'etaplectique}
En utilisant la construction de ACK ci-dessus dans le cas $r=1$, on obtient aussi une extension centrale
$$1\ra k^*\ra \tG'_1(\fp)\ra\G_1(\fp)\ra 1.$$ On note $\det^*(\tG'_1(\fp))$ l'image réciproque de cette extension par le morphisme de d\'eterminant $\det:\G_r(\fp)\ra \G_1(\fp)$. Dans le groupe d'extensions $H^2(\G_r(\fp),k^*)$ (dont la loi de groupe sera not\'ee additivement), on consid\`ere la classe d'\'equivalence $\det^*(\tG'_1(\fp))-\tG'_r(\fp)$. C'est la classe d'\'equivalence de l'extension $$\tG_{r,{\rm geo}}(\fp)=\{(g,v)|g\in \G_r(\fp),\,v\in \Delta_g-\{0\}\}\,,$$
o\`u l'on note $\Delta_{g}=D_{\det(g)}\otimes D_g^{\otimes (-1)}$.
Les scindages de $\tG_r'(\fp)$ au-dessus de $N_r$ et $K_r$ d\'efinissent encore des scindages de $\tG_{r,\rm geo}(\fp)$ au-dessus de $N_r$ et $K_r$.
\begin{rmk}\label{rmk2}
Comme pour l'extension ACK, soit $k'$ une extension finie de $k$. On note $\op'=k'[[\va]]$ et $\fp'$ son corps de fractions. On a aussi une extension m\'etaplectique g\'eom\'etrique
$$1\ra (k')^*\ra \tG_{r,\rm geo}(\fp')\ra\G_r(\fp')\ra 1,$$
comme celle d\'efinie  ci-dessus dans le cas particulier o\`u $k'= k$. En poussant cette extension via le morphisme de norme $N_{k'/k}:{k'}^*\ra k$ on obtient une extension m\'etaplectique de $\G_r(\fp')$ par $k^*$ :
$1\ra k^*\ra N_*\tG_{r,\rm geo}(\fp')\ra \G_r(\fp')\ra 1.$ Cette extension est aussi obtenue en associant \`a chaque $g\in \G_r(\fp')$ la droite $\Delta'_g=D'_{\det(g)}\otimes_k{D'}_{g}^{\otimes -1}$.
\end{rmk}

\begin{theorem}\label{chi4} $\tG_{r,\rm geo}(\fp)$ et $\tG_{r,KP}(\fp)$ sont isomorphes.
\end{theorem}

Pour montrer ce th\'eor\`eme, on va construire une section ensembliste $s_{{\rm geo}}:\G_r(\fp)\ra\tG_{r,{\rm geo}}(\fp)$. Ensuite on associera \`a cette section  un $2$-cocycle $\chi_{\rm geo}  :\G_r(\fp)\times \G_r(\fp)\ra k^*$, d\'efini par
$$({\rm Id}_r, \chi_{\rm geo}(g,g'))=s_{\rm geo}(g)s_{\rm geo}(g')/s_{\rm geo}(gg').$$ Le th\'eor\`eme~\ref{chi4} r\'esultera  de l'\'enonc\'e plus pr\'ecis $$\chi_{\rm geo}(g,g')=\chi(g,g').$$
\begin{const}({\it La section ensembliste $s':\G_r(\fp)\ra \tG'_r(\fp)$}).

Soient $(e_i)_i$ la base canonique de $\fp^r$ et $(e^*_i)_i$ sa base duale. On note $\varpi^j_i=\varpi^j e_i$ et $\varpi^{j*}_i=\varpi^je^*_i$. Soient $t=\D(t_1,\dots,t_r)\in T_r(\fp)$ et $w\in W_r$. Soit $\mathfrak{n}\gg 0 \in 2\mathbb{Z}$ tel que $\varpi^\mathfrak n\op^r\subset \op^r\cap tw\op^r$, par d\'efinition on a un isomorphisme canonique
$$D_{tw}=(\op^r|tw\op^r)=(\op^r|t\op^r)\can (\bigwedge \op^r/\varpi^{\mathfrak{n}}\op^r)^*\otimes (\bigwedge t\op^r/\varpi^{\mathfrak{n}}\op^r).$$

On choisit  $(\bigwedge_{i=1}^r\bigwedge_{j=1}^{\n-1}\varpi^j_i)^*\otimes \bigwedge_{i=1}^r\bigwedge_{j=v(t_i)}^{\n-1}\varpi^j_i$ comme base du facteur de gauche et on note $\mathfrak{d}_{tw}$ son image réciproque dans $D_{tw}$ par l'isomorphisme ci-dessus. L'\'el\'ement $\mathfrak{d}_{tw}$ ne d\'epend pas du choix de $\n$. En pratique, on notera ``=" au lieu de l'isomorphisme canonique, i.e $\mathfrak{d}_{tw}=(\bigwedge_{i=1}^r\bigwedge_{j=1}^{\n-1}\varpi^j_i)^*\otimes \bigwedge_{i=1}^r\bigwedge_{j=v(t_i)}^{\n-1}\varpi^j_i$.

Soit $g\in \G_r(\fp)$, \`a l'aide de la d\'ecomposition de Bruhat, on a $g=nB(g)n'$. En utilisant le scindage au-dessus de $N_r(\fp)$ de l'extension de ACK et l'isomorphisme de multiplication, on obtient l'isomorphime canonique
$D_g\can D_{B(g)}$ (c'est l'isomorphisme compos\'e de $D_g\can D_n\otimes D_{B(g)}\otimes D_{n'}$ et de $D_n\otimes D_{B(g)}\otimes D_{n'}\can k\otimes D_{B(g)}\otimes k \can D_{B(g)}$). On note $\mathfrak{d}_{g}$ l'image réciproque de $\mathfrak{d}_{B(g)}\in D_{B(g)}$ dans $D_g$ par cet isomorphisme.
\begin{lemma}\label{s1}$\mathfrak{d}_{g}$ ne depend pas du choix de l'écriture
$g=nB(g)n'$.
\end{lemma}
\begin{proof}[Démonstration]
On suppose que $g=n_1B(g)n_1'=n_2B(g)n_2'$. On note $\mathfrak{d}_g^{i}\ (i=1,2)$ l'élément de $D_g$  défini comme ci-dessus relativement aux deux \'ecritures  de $g$. D'après la définition, on a $\mathfrak{d}_g^i=\mathfrak{d}_{n_i}\otimes\mathfrak{d}_{B(g)}\otimes\mathfrak{d}_{n'_i}$.

On a $\mathfrak{d}_g^1=\mathfrak{d}_g^2\Leftrightarrow\mathfrak{d}_{n_2}^{\otimes-1}\otimes\mathfrak{d}_{n_1}\otimes\mathfrak{d}_{B(g)}=\mathfrak{d}_{B(g)}\otimes\mathfrak{d}_{n'_2}\otimes\mathfrak{d}_{n'_1}^{\otimes-1}$. Comme $\sigma_N$ est un morphisme de groupes (cf. le lemme~\ref{n4}), on a $\mathfrak{d}_{n_2}^{\otimes-1}\otimes\mathfrak{d}_{n_1}=\mathfrak{d}_{n_2^{-1}n_1}$ et $\mathfrak{d}_{n'_2}\otimes\mathfrak{d}_{n'_1}^{\otimes-1}=\mathfrak{d}_{n'_2{n'}_1^{-1}}$. On note $n=n_2^{-1}n_1$ et $n'=n'_2{n'}_1^{-1}$. Le lemme résultera immédiatement du fait  que l'image de $\mathfrak{d}_{n}\otimes\mathfrak{d}_{tw}$ dans $D_{ntw}$ par l'isomorphisme canonique $D_n\otimes D_{tw}\ra D_{ntw} $ et celle de $\mathfrak{d}_{tw}\otimes\mathfrak{d}_{n'}$ dans $D_{twn'}$ par l'isomorphimse de $D_{tw}\otimes D_{n'}\ra D_{twn'}$ sont identiques quand $ntw=twn'$, ce qu'on va maintenant d\'emontrer. Plus précisement on va montrer que l'automorphisme $D_{ntw}\can D_{tw} \can D_{twn'}$ est trivial.

Soit $M=\D(\beta_1,\dots,\beta_r)\0_\varpi^r$ tel que $v(\beta_i)\ll v(\beta_{i+1})$ (de telle sorte que $nM=n'M=M$). Comme $ntw=twn'$, on a $ntwM=twn'M=twM$. Comme dans la démonstration du lemme~\ref{n3}, l'automorphisme $(M|twM)\buildrel {\times n} \over{\ra} (nM|ntwM)$ est trivial. On note $u\neq 0$ une base de $(M|twM)$ (donc $u=(\times n)(u)$) et $v\neq 0$ une base de $(\op^r|M)$. D'après la définition l'image de $u$ dans $D_{ntw}$ par l'isomorphisme canonique $(M|twM)\can D_{tw}\can D_{ntw}$ est l'image de $v\otimes (\times n)(u)\otimes (\times ntw)(v^*)\in (\op^r|M)\otimes(nM|ntwM)\otimes(ntwM|ntw\op^r)$ par l'isomorphisme de multiplication. De même, l'image de $u$ dans $D_{twn'}$ est l'image de $v\otimes u\otimes (\times twn')(v^*)\in (\op^r|M)\otimes(M|twM)\otimes(twn'M|twn'\op^r)$ par l'isomorphimse de multiplication. Par ailleurs, comme $ntw=twn'$ et $u=(\times n)(u)$, on obtient $v\otimes (\times n)(u)\otimes (\times ntw)(v^*)=v\otimes u\otimes (\times twn')(v^*)$, ce qui achève la démonstration du lemme.\qedhere

\end{proof}

La section ensembliste $s':\G_r(\fp)\ra \tG'_r(\fp)$ est alors d\'efinie par $s'(g)=(g,\mathfrak{d}_{g})$. On note $\chi'(g,g'):=s'(g)s'(g')/s'(gg')$ le $2$-cocycle associé à $s'$. On note $\mathfrak{d}_g*\mathfrak{d}_{g'}$ l'image de $\mathfrak{d}_g\otimes\mathfrak{d}_{g'}$ dans $D_{gg'}$ par l'isomorphisme de multiplication $D_g\otimes D_{g'}\to D_{gg'}$.
\end{const}
\begin{lemma}\label{chi6}
\begin{enumerate}
\item $\chi'(t,t')=\displaystyle(-1)^{\sum_{i <j}v(t_i)v(t'_j)}\prod_{i=1}^r{\left(\frac{t_i}{\varpi^{v(t_i)}}\right)}^{v(t'_i)}(0)$ \\  o\`u $t=\D(t_i)_i$ et $t'=\D(t_i')_i$.
\item $\chi'(w,w')=1$\quad $(w,w'\in W_r)$.
\item $\chi'(t,w)=1$\quad $(w\in W_r$, $t \in T_r(\fp))$.
\item $\chi'(\alpha,t)=\displaystyle(-1)^{v(t_{\ell})v(t_{\ell+1})} $\\
 o\`u $\alpha$ est la matrice de la permutation $(\ell, \ell+1)$ et $t\in T_r(\fp)$.
\item $\chi'(ng,g'n')=\chi'\quad (n,n'\in N_r(\fp))$
\item $\chi'(t,g)=\chi'(t,B(g))\quad (t\in T_r(\fp))$
\end{enumerate}
\end{lemma}
\begin{proof}[D\'emonstration]
\begin{enumerate}
\item On rappelle la d\'efinition du morphisme canonique de multiplication sur $T_r(\fp)$ :
$$\xymatrix{D_{t}\otimes D_{t'}\ar[r]^-{\times t}&(\op^r|t\op^r)\otimes(t\op^r|tt'\op^r)\ar[r]^-{\beta} &D_{tt'}.
}$$
Soit $\n\in 2\mathbb{Z}$ tel que $\n>{\rm max}\{v(t_i),v(t'_i),v(t_i)+v(t'_i)\}\,,\  \forall i$.
On a $t\op^r/\varpi^{\n}t\op^r=\displaystyle\oplus_{i=1}^rt_i\op/\varpi^{\n}t_i \op$. Cela nous donne  une base  $\left\{\varpi_i^j\right\}$ de $t\op^r/\varpi^{\n}\op^r$, o\`u $i$ varie de $1$ \`a $r$ et o\`u $j$ varie de $v(t_i)$ \`a $\n-1+v(t_i)$. La matrice de passage de $\left\{\varpi_i^j\right\}^{i=1...r}_{j=v(t_i)...\n-1+v(t_i)}$ \`a $\left\{t\varpi_i^j\right\}^{i=1...r}_{j=0...\n-1}$  est la matrice carr\'ee  $\m=\D(\m_i)$ de taille $\n r$, où
$$\m_i=\left(\begin{matrix}
\frac{t_i}{\varpi^{v(t_i)}}(0)&*&*&\dots&*\\
&\frac{t_i}{\varpi^{v(t_i)}}(0)&*&\dots&*\\
&&&\ddots&\vdots\\
&&&&\frac{t_i}{\varpi^{v(t_i)}}(0)
\end{matrix}\right)\text{ de taille $\n$}.$$

On a aussi $tt'\op^r/\varpi^{\n}t\op^r=\displaystyle\oplus_{i=1}^rt_it_i'\op/\varpi^{\n} t_i\op$. Cela nous donne  une base  $\left\{\varpi_i^j\right\}^{i=1\dots r}_{j=v(t_it_i')\dots \n-1+v(t_i)}$ de $tt'\op^r/\varpi^{\n}t\op^r$. La matrice de passage de $\left\{\varpi_i^j\right\}^{i=1...r}_{j=v(t_it'_i)...\n-1+v(t_i)}$ \`a $\left\{t\varpi_i^j\right\}^{i=1...r}_{j=v(t'_i)...\n-1}$  est la matrice carr\'ee  $\m'=\D(\m'_i)$ de taille $\n r-\sum_{i=1}^rv(t'_i)$, où
$${\m'}_i=\left(\begin{matrix}
\frac{t_i}{\varpi^{v(t_i)}}(0)&*&*&\dots&*\\
&\frac{t_i}{\varpi^{v(t_i)}}(0)&*&\dots&*\\
&&&\ddots&\vdots\\
&&&&\frac{t_i}{\varpi^{v(t_i)}}(0)
\end{matrix}\right)\text{ de taille $(\n-v(t'_i))$}.$$
On a (on se contente de faire le calcul dans le cas   o\`u $v(t_i)>0$, le cas g\'en\'eral \'etant laiss\'e au lecteur) :
{\allowdisplaybreaks\footnotesize
\begin{eqnarray*}
\mathfrak{d}_{t}*\mathfrak{d}_{t'}&=&\displaystyle\left(\bigwedge_{i=1}^r\bigwedge_{j=1}^{\n-1}\varpi^j_i\right)^*\otimes\bigwedge_{i=1}^r\bigwedge_{j=v(t_i)}^{\n-1}\varpi^j_i\otimes\left(\bigwedge_{i=1}^r\bigwedge_{j=1}^{\n-1}t\varpi^j_i\right)^*\otimes\bigwedge_{i=1}^r\bigwedge_{j=v(t'_i)}^{\n-1}t\varpi^j_i\\
&=&\displaystyle\left(\bigwedge_{i=1}^r\bigwedge_{j=1}^{\n-1}\varpi^j_i\right)^*\otimes\bigwedge_{i=1}^r\bigwedge_{j=v(t_i)}^{\n-1}\varpi^j_i\otimes \frac{1}{\det(\m)}\left(\bigwedge_{i=1}^r\bigwedge_{j=v(t_i)}^{\n-1+v(t_i)}\varpi^j_i\right)^*\otimes\\
&&\otimes\bigwedge_{i=1}^r\bigwedge_{j=v(t'_i)}^{\n-1}t\varpi^j_i\\
&=&\displaystyle\left(\bigwedge_{i=1}^r\bigwedge_{j=1}^{\n-1}\varpi^j_i\right)^*\otimes\frac{(-1)^{\sum_{i<j}(\n-v(t_j))v(t_i)}}{\det(\m)}\left(\bigwedge_{i=1}^r\bigwedge_{j=\n}^{\n-1+v(t_i)}\varpi^j_i\right)^*\otimes\\
&&\otimes\bigwedge_{i=1}^r\bigwedge_{j=v(t'_i)}^{\n-1}t\varpi^j_i\\
&=&\displaystyle\left(\bigwedge_{i=1}^r\bigwedge_{j=1}^{\n-1}\varpi^j_i\right)^*\otimes\frac{(-1)^{\sum_{i<j}(\n-v(t_j))v(t_i)}}{\det(\m)}\left(\bigwedge_{i=1}^r\bigwedge_{j=\n}^{\n-1+v(t_i)}\varpi^j_i\right)^*\otimes\\
&&\otimes\det(\m')\bigwedge_{i=1}^r\bigwedge_{j=v(t_it'_i)}^{\n-1+v(t_i)}\varpi^j_i\\
&=&\frac{(-1)^{\sum_{i<j}(\n-v(t_j))v(t_i)}(-1)^{\sum_{i< j}(\n-v(t_jt_j'))v(t_i)}\det(\m')}{\det(\m)}\times\\
&&\times\displaystyle\left(\bigwedge_{i=1}^r\bigwedge_{j=1}^{\n-1}\varpi^j_i\right)^*\otimes\bigwedge_{i=1}^r\bigwedge_{j=v(t_it'_i)}^{\n-1}\varpi^j_i\\
&=&(-1)^{\sum_{i<j}v(t_i)v(t'_j)}\prod_{i=1}^r\frac{\det(\m^i_i)}{\det({\m'}^i_i)}.\mathfrak{d}_{tt'}\\
&=&(-1)^{\sum_{i<j}v(t_i)v(t'_j)}\prod_{i=1}^r\left(\frac{t_i}{\varpi^{v(t_i)}}\right)^{v(t'_i)}(0).\mathfrak{d}_{tt'}.
\end{eqnarray*}}

On en d\'eduit $\chi'(t,t')=\displaystyle(-1)^{\sum_{i<j}v(t_i)v(t'_j)}\prod_{i=1}^r\left(\frac{t_i}{\varpi^{v(t_i)}}\right)^{v(t'_i)}(0)$.
\item Les assertions (2) et (3) r\'esultent imm\'ediatement de la construction de la section $s'$.
\item En notant ${}^\alpha t=\alpha t\alpha^{-1}\in T_r(\fp)$, on a
$$\mathfrak{d}_{\alpha t}=\mathfrak{d}_{{}^{\alpha} t\alpha}=(\bigwedge_{i=1}^r\bigwedge_{j=1}^{n-1}\varpi^j_i)^*\otimes \bigwedge_{i=1}^r\bigwedge_{j=v(t_{\alpha(i)})}^{n-1}\varpi^j_i.$$

D'autre part, 
on a alors :
{\small\begin{eqnarray*}
\mathfrak{d}_{\alpha}*\mathfrak{d}_{ t}&=&(\bigwedge_{i=1}^r\bigwedge_{j=1}^{\n-1}\alpha\varpi^j_{i})^*\otimes \bigwedge_{i=1}^r\bigwedge_{j=v(t_{i})}^{\n-1}\alpha\varpi^j_i\\
&=&(\bigwedge_{i=1}^r\bigwedge_{j=1}^{\n-1}\varpi^j_{\alpha(i)})^*\otimes \bigwedge_{i=1}^r\bigwedge_{j=v(t_{i})}^{\n-1}\varpi^j_{\alpha(i)}\\
&=&(-1)^{\n²}(\bigwedge_{i=1}^r\bigwedge_{j=1}^{\n-1}\varpi^j_i)^*\otimes(-1)^{(\n-v(t_{\ell}))(\n-v(t_{\ell+1}))} \bigwedge_{i=1}^r\bigwedge_{j=v(t_{\alpha(i)})}^{n-1}\varpi^j_i\\
&=&(-1)^{v(t_{\ell})v(t_{\ell+1})}\mathfrak{d}_{\alpha t}
\end{eqnarray*}}

Par cons\'equent, on a : $\chi'(\alpha,t)=(-1)^{v(t_{\ell})v(t_{\ell+1})}.$
\item En utilisant la propri\'et\'e de $2$-cocycle $\chi'(g_1,g_2)\chi'(g_1g_2,g_3)=\chi'(g_1,g_2g_3)\break\chi'(g_2,g_3)$, le point (5) r\'esulte imm\'ediatement de $\chi'(n,g)=\chi'(g,n')=1, \  \forall g\in \G_r(\fp)$. On montre que $\chi'(n,g)=1$ et l'autre est pareil.

Soit $g=n_1twn_2$, d'apr\`es la construction on a $s'(g)=s'(n_1)s'(tw)s'(n_2)$ et $s'(ng)=s'(nn_1)s'(tw)s'(n_2)$. Comme $s'(n)s'(n_1)=s'(nn_1)$, on obtient $s'(n)s'(g)=s'(ng)$, i.e $\chi'(n,g)=s'(n)s'(g)/s'(ng)=1$.
\item Soit $g=n_1B(g)n_2$, on a $tg=n_1'tB(g)n_2=n_1'B(tg)n_2$. On a
\begin{eqnarray*}
s'(t)s'(g)&=&s'(t)s'(n_1)s'(B(g))s'(n_2)\\
&=&(tn_1,\mathfrak{d}_{tn_1})s'(B(g))s'(n_2)\ (\text{comme } \chi(t,n_1)=1)\\
&=&(n_1't,\mathfrak{d}_{n'_1t})s'(B(g))s'(n_2)\\
&=&s'(n'_1)s'(t)s'(B(g))s'(n_2)\ (\text{comme } \chi(n'_1,t)=1)\\
&=&\chi'(t,B(g))s'(n'_1)s'(B(tg))s'(n_2)\\
&=&\chi'(t,B(g))s'(tg)
\end{eqnarray*}
Par cons\'equent, $\chi'(t,g)=\chi'(t,B(g))$.
\end{enumerate}
\end{proof}
\begin{const}\label{chi5}
Soit $t=\D(t_1,\dots,t_r)\in T_r(\fp)$. On a alors un isomorphisme canonique
$$D_t\displaystyle\can\bigotimes_{i=1}^rD_{t_i}$$
d\'efini comme suit.
\end{const}
Soit $\n\gg 0 \in 2\mathbb{Z}$ tel que : $\va^\n\op \subset t_i\op\cap \op\,,\  \forall \, i\in \{1,\dots,\n\}$ et $\va^\n\op^r\subset t\op^r\cap\op^r$. On a
les isomorphismes canoniques :
$$D_{t_i}\buildrel \beta\over{\leftarrow} (\op|\va^\n\op)\otimes(\va^\n\op|t_i\op)=\left(\bigwedge(\op/\va^\n\op)\right)^*\otimes\bigwedge(t_i\op/\va^\n\op),$$
$$D_t\buildrel \beta\over{\leftarrow} (\op^r|\va^\n\op^r)\otimes(\va^\n\op|t_i\op)=\left(\bigwedge(\op^r/\va^\n\op^r)\right)^*\otimes\bigwedge(t\op^r/\va^\n\op^r).$$

En utilisant la suite exacte : $0\ra t_i\op/\va^\n\op\ra \bigoplus_{j=1}^it_j\op/\va^\n\op\ra\bigoplus_{j=1}^{i-1}t_j\op/\va^\n\op\ra 0$, on a l'isomorphisme canonique :
$$\bigwedge\left(\bigoplus_{j=1}^it_j\op/\va^\n\op\right)\buildrel\phi_i\over{\leftarrow}\bigwedge \left(\bigoplus_{j=1}^{i-1}t_j\op/\va^\n\op \right)\otimes\bigwedge(t_i\op/\va^\n\op),$$
o\`u $\phi_i=\phi(\bigoplus_{j=1}^{i-1}t_j\op/\va^\n\op,t_j\op/\va^\n\op)$.

En utilisant la suite exacte : $0\ra \op/\va^\n\op\ra \bigoplus_{j=1}^i\op/\va^\n\op\ra\bigoplus_{j=1}^{i-1}\op/\va^\n\op\ra 0$, on a l'isomorphisme canonique :
$$\bigwedge\left(\bigoplus_{j=1}^i\op/\va^\n\op\right)\buildrel\phi'_i\over{\leftarrow}\bigwedge \left(\bigoplus_{j=1}^{i-1}\op/\va^\n\op \right)\otimes\bigwedge(\op/\va^\n\op),$$
o\`u $\phi'_i=\phi(\bigoplus_{j=1}^{i-1}\op/\va^\n\op,\op/\va^\n\op)$.

Par ailleurs, 
on a l'isomorphisme canonique :
$$\xymatrix{\bigotimes_{i=1}^r\left[\left(\bigwedge(\op/\va^\n\op)\right)^*\otimes\bigwedge(t_i\op/\va^\n\op)\right]\ar[d]\\
\bigotimes_{i=1}^r\left(\bigwedge(\op/\va^\n\op)\right)^*\otimes\bigotimes_{i=1}^r\bigwedge(t_i\op/\va^\n\op),}$$
(comme on utilise la règle de commutation de Koszul, c'est l'isomorphisme \'evident multipli\'e par $(-1)^{\sum_{i=2}^r\n.\sum_{j=1}^{i-1}(\n-v(t_j))}=1.$)

Par cons\'equent, en utilisant les relations $t\op^r/\va^\n\op^r=\bigoplus_{i=1}^rt_i\op/\va^\n\op$, et $\op^r/\va^\n\op^r=\bigoplus_{i=1}^r\op/\va^\n\op$, on obtient le diagramme
$$\xymatrix{\bigotimes_{i=1}^r\left[\left(\bigwedge(\op/\va^\n\op)\right)^*\otimes\bigwedge(t_i\op/\va^\n\op)\right]\ar[d]\ar[r]^-{\beta}&\bigotimes_{i=1}^rD_{t_i}\ar[dd]\\
\bigotimes_{i=1}^r\left(\bigwedge(\op/\va^\n\op)\right)^*\otimes\bigotimes_{i=1}^r\bigwedge(t_i\op/\va^\n\op)\ar[d]_-{\phi}\\
\left(\bigwedge(\op^r/\va^\n\op^r)\right)^*\otimes\bigwedge(t\op^r/\va^\n\op^r)\ar[r]_-{\beta}&D_t,
}$$
o\`u $\phi=(\phi'_r\circ\dots\circ\phi'_2)\otimes(\phi_r\circ\dots\circ\phi_2)$.

Soit $tw \in T_rW_r(\fp)$, o\`u $t=\D(t_1,\dots,t_r) \in T_r(\fp)$. Clairement, on a
$$\Delta_{tw}\can Hom_k(D_{tw},D_{\det(tw)}).$$
On d\'efinit alors $\delta_{tw}$ comme l'isomorphisme rendant le diagramme suivant  
$$\xymatrix{D_{tw}=D_t\ar[d]_-{\delta_{tw}}\ar[r]_-{can}^-{(~\ref{chi5})}&\bigotimes_{i=1}^r D_{t_i} \ar[rrr]^-{\times(-1)^{\sum_{i<j}v(t_i)v(t_j)}}&&&\bigotimes_{i=1}^r D_{t_{r+1-i}}\ar[d]^{\mu}\\
D_{\det(tw)}=D_{\det(t)}&&&&D_{\det(t)}\ar[llll]^-{\times(-1)^{\sum_{i=1}^r(r-i)v(t_i)+\sum_{i<j}v(t_i)v(t_j)}},
} $$
commutatif,
o\`u l'isomorphisme $\mu$ est l'isomorphisme compos\'e des isomorphismes de multiplication $ D_{\prod_{j=1}^{r-i}t_{r+1-j}} \otimes D_{t_i} \buildrel{\mu_i}\over{\ra} D_{\prod_{j=1}^{r-i+1}t_{r+1-j}}$, c.\`a.d $\mu=\mu_{1}\circ\dots\circ\mu_{r-1}$.

Soit $g\in \G_r(\fp)$, \`a l'aide de la d\'ecomposition de Bruhat on a $g=ntwn'$. On a un isomorphisme canonique $\Delta_n\otimes \Delta_{tw} \otimes \Delta_{n'}\can \Delta_g$ provenant du diagramme
{\small$$\xymatrix{ \Delta_n\otimes \Delta_{tw} \otimes \Delta_{n'}\ar[d]\ar@{=}[r]&D_{\det(n)}\otimes D_n^{\otimes -1}\otimes D_{\det(tw)}\otimes D_{tw}^{\otimes -1}\otimes D_{\det(n')}\otimes D_{n'}^{\otimes -1}\ar[d]\\
D_{\det(g)}\otimes D_{g}^{\otimes -1}&(D_{\det(n)}\otimes D_{\det(tw)}\otimes D_{\det(n')})\otimes(D_n\otimes D_{tw}\otimes D_{n'})^{\otimes -1}\ar[l]}$$} o\`u l'isomorphisme vertical de gauche est l'isomorphisme compos\'e des isomorphismes de commutation (la règle de Koszul est ici triviale car $D_n$ et $D_{\det(n)}$ sont des droites gradu\'ees de degr\'e 0), et o\`u l'isomorphisme horizontal du bas est le compos\'e des isomorphismes de multiplication. Puisque $\Delta_n\can k\,,\ \forall n\in N_r(\fp)$ on a
$\Delta_g\can \Delta_{tw}$. On note $\delta_{g}$ l'image r\'eciproque de $\delta_{tw}$ par cet isomorphisme.
\begin{proposition}$\delta_g$ ne d\'epend pas du choix l'écriture $g=nB(g)n'$.
\end{proposition}
\begin{proof}
Elle est semblable à celle du lemme~\ref{s1}.
\end{proof}Grâce à cette proposition, on obtient une section ensembliste $s_{\rm geo}: \G_r(\fp)\ra\tG_{r,\rm geo}(\fp)$ d\'efinie par
$$s_{\rm geo}(g)=(g,\delta_{g}).$$
 On note $\chi_{\rm geo}$ le 2-cocycle associé à cette section. On va comparer $\chi_{\rm geo}$ au cocycle de Kazhdan-Patterson et le th\'eor\`eme~\ref{chi4} r\'esulte alors de la proposition suivante :
\begin{proposition}\label{chi2}
\begin{enumerate}
\item $\chi_{{\rm geo}}(t,t')=\displaystyle(-1)^{\sum_{i<j}v(t_i)v(t_j')}\prod_{i<j}\frac{t_i^{v(t_j')}}{{t_j'}^{v(t_i)}}(0)$ \\  o\`u $t=\D[t_i]$ et $t'=\D[t_i']$.
\item $\chi_{{\rm geo}}(w,w')=1$\quad $(w,w' \in W_r)$.
\item $\chi_{{\rm geo}}(t,w)=1$\quad $(w\in W_r$, $t \in T_r)$.
\item $\chi_{{\rm geo}}(\alpha,t)=\displaystyle(-1)^{v(t_{\ell})v(t_{\ell+1})+\sum_{i\neq\ell,\ell+1}v(t_i)}\frac{t_{\ell+1}^{v(t_{\ell})}}{t_{\ell}^{v(t_{\ell+1})}}(0) $\\
 o\`u $\alpha$ est la matrice de la permutation $(\ell,\ell+1)$ et $t\in T_r$.
\item $\chi_{{\rm geo}}(ng,g'n')=\chi_{{\rm geo}}(g,g')\quad (n,n'\in N_r)$
\item $\chi_{{\rm geo}}(t,g)=\chi_{{\rm geo}}(t,B(g))\quad (t\in T_r)$
\item $\chi_{{\rm geo}}(\alpha,g)=\chi_{{\rm geo}}(B(\alpha g)B(g)^{-1},B(g))$
 \end{enumerate}

 \vskip 2mm
En particulier, $\chi_{\rm geo}(g,g')=\chi(g,g')$.
\end{proposition}
\begin{proof}

On a le diagramme suivant :
$$\xymatrix{D_{g}\otimes D_{g'}\ar[r]^-{\delta_{g}\otimes\delta_{g'}}\ar[d]&D_{\det(g)}\otimes D_{\det(g')}\ar[d]\\
D_{gg'}\ar[r]_{\delta_{gg'}}&D_{\det(gg')}\,,
}$$
o\`u les isomorphismes verticaux  ci-dessus sont des isomorphismes de multiplication.

D'apr\`es la d\'efinition de $\delta_{g}$ on trouve
$$\delta_{g}=[\delta_{g}]\mathfrak{d}_{{\rm det}(g)}\otimes\mathfrak{d}_g^{\otimes -1}\,,$$
o\`u $[\delta_{g}]=(-1)^{\sum_{i=1}^r(r-i)v(t_i)}\prod_{i=1}^{r-1}\chi'(\prod_{j=i+1}^{r}t_j,t_i)$.
Le cocycle $\chi_{\rm geo}$ est alors donn\'e par la formule
$$\chi_{\rm geo}(g,g')=\frac{[\delta_{g}].[\delta_{g'}].\chi'(\det(g),\det(g'))}{[\delta_{gg'}].\chi'(g,g')}.$$

Par cons\'equent, les six premi\`eres  assertions de la proposition r\'esultent imm\'ediatement des six assertions du lemme~\ref{chi6} et de la d\'efinition de $s_{\rm geo}$ 

On va maintenant donner la d\'emonstration du point (7) de la proposition.

On note $m_{\alpha}(u)$ la matrice $\D({\rm Id}_{\ell-1},\begin{pmatrix}
0&u\\
-\frac{1}{u}&0
\end{pmatrix},{\rm Id}_{r-\ell-1})$, $e_{\alpha}^{u}$ la matrice $\D({\rm Id}_{\ell-1},\begin{pmatrix}
1&u\\
0&1
\end{pmatrix},{\rm Id}_{r-\ell-1})$ et $d_{\alpha}(u)$ la matrice $\D({\rm Id}_{\ell-1},-\frac{1}{u},u,{\rm Id}_{r-\ell-1})$. On note aussi $e_{-\alpha}^u$ la matrice de transposition de $e_{\alpha}^u$, i.e $e_{-\alpha}^u=\break\D({\rm Id}_{\ell-1},\begin{pmatrix}
1&0\\
u&1
\end{pmatrix},{\rm Id}_{r-\ell-1})$.


Soit $g=nwtn'$, o\`u $t=\D(t_1,\dots,t_r)$. On \'ecrit $n=n_{\alpha}e_\alpha^u$, o\`u $n_\alpha \in N_r(\fp)$ est telle que le coefficient situ\'e sur la $\ell$-i\`eme  ligne et la $\ell+1$-i\`eme colonne est nul.

On consid\`ere $\alpha g=\alpha n_\alpha e_\alpha^uB(g)n'=n''\alpha e_\alpha^uB(g)n'$.
\begin{itemize}\item Si $u=0$, on a $\alpha g=n''\alpha B(g)n'$, donc $B(\alpha g)=\alpha B(g)$. On voit que
\begin{eqnarray*}
\delta_\alpha*\delta_g&=&\delta_\alpha*(\delta_{n_{\alpha}}*\delta_{B(g)}*\delta_{n'})\\
&=&\delta_{\alpha n_\alpha}*\delta_{B(g)}*\delta_{n'}\\
&=&\delta_{n''\alpha}*\delta_{B(g)}*\delta_{n'}\\
&=&\delta_{n''}*\delta_{\alpha}*\delta_{B(g)}*\delta_{n'}\\
&=&\chi_{\rm geo}(\alpha,B(g))\delta_{n''}*\delta_{\alpha B(g)}*\delta_{n'}\\
&=&\chi_{\rm geo}(\alpha,B(g))\delta_{\alpha g}
\end{eqnarray*}
Par cons\'equent, $\chi_{\rm geo}(\alpha,g)=\chi_{\rm geo}(B(\alpha g)B(g)^{-1},B(g))$.
\item Si $w^{-1}(\ell)<w^{-1}(\ell+1)$, comme $B(g)^{-1}e_\alpha^uB(g)\in N_r(\fp)$, on a $\alpha g=n''\alpha B(g)(B(g)^{-1}e_\alpha^uB(g)n')$, de sorte qu'on obtient $B(\alpha g)=\alpha B(g)$. On a
{\allowdisplaybreaks\begin{eqnarray*}
\delta_\alpha*\delta_g&=&\delta_\alpha*\delta_{n}*\delta_{B(g)}*\delta_{n'}\\
&=&\delta_{\alpha n}*\delta_{B(g)}*\delta_{n'}\\
&=&\delta_{n''\alpha e^u_\alpha}*\delta_{B(g)}*\delta_{n'}\\
&=&\delta_{n''}*\delta_{\alpha e^u_\alpha}*\delta_{B(g)}*\delta_{n'}\\
&=&\delta_{n''}*\delta_{\alpha}*\delta_{e^u_\alpha}*\delta_{B(g)}*\delta_{n'}\\
&=&\delta_{n''}*\delta_{\alpha}*\delta_{e^u_\alpha B(g)}*\delta_{n'}\\
&=&\delta_{n''}*\delta_{\alpha}*\delta_{B(g)B(g)^{-1}e^u_\alpha B(g)}*\delta_{n'}\\
&=&\delta_{n''}*\delta_{\alpha}*\delta_{B(g)}*\delta_{B(g)^{-1}e^u_\alpha B(g)}*\delta_{n'}\\
&=&\chi_{\rm geo}(\alpha,B(g))\delta_{n''}*\delta_{\alpha B(g)}*\delta_{B(g)^{-1}e^u_\alpha B(g)n'}\\
&=&\chi_{\rm geo}(\alpha,B(g))\delta_{\alpha g}
\end{eqnarray*}}
Par cons\'equent, $\chi_{\rm geo}(\alpha,g)=\chi_{\rm geo}(B(\alpha g)B(g)^{-1},B(g))$.
\item Si $w^{-1}(\ell)>w^{-1}(\ell+1)$, et $u\neq 0$, on a $e_\alpha^u=m_\alpha(u)e_{\alpha}^{-u}e_{-\alpha}^{1/u}$. Ceci implique (en utilisant $\alpha m_{\alpha}(u)=d_{\alpha}(u)$)
{\allowdisplaybreaks\begin{eqnarray*}
\alpha g&=&n''d_{\alpha}(u)e_{\alpha}^{-u}e_{-\alpha}^{1/u}B(g)n'\\
&=&n'''(d_{\alpha}(u)B(g))(B(g)^{-1}e_{-\alpha}^{1/u}B(g))n'.
\end{eqnarray*}}
Comme $B(g)^{-1}e_{-\alpha}^{1/u}B(g)\in N_r(\fp)$, on a $B(\alpha g)=d_{\alpha}(u)B(g)$. On consid\`ere
{\allowdisplaybreaks\begin{eqnarray*}
\delta_\alpha*\delta_g&=&\delta_\alpha*\delta_{n}*\delta_{B(g)}*\delta_{n'}\\
&=&\delta_{\alpha n}*\delta_{B(g)}*\delta_{n'}\\
&=&\delta_{n''\alpha e^u_\alpha}*\delta_{B(g)}*\delta_{n'}\\
&=&\delta_{n''d_{\alpha}(u)e_\alpha^{-u}e_{-\alpha}^{1/u}}*\delta_{B(g)}*\delta_{n'}\\
&=&\delta_{n'''d_{\alpha}(u)e_{-\alpha}^{1/u}}*\delta_{B(g)}*\delta_{n'}\\
&=&\delta_{n'''}*\delta_{d_{\alpha}(u)e_{-\alpha}^{1/u}}*\delta_{B(g)}*\delta_{n'}\\
&=&\chi_{\rm geo}(d_{\alpha}(u),e_{-\alpha}^{1/u})^{-1}\times\\
&&\times\delta_{n'''}*\delta_{d_{\alpha}(u)}* \delta_{e_{-\alpha}^{1/u}}*\delta_{B(g)}*\delta_{n'}\\
&=&\chi_{\rm geo}(d_{\alpha}(u),B(e_{-\alpha}^{1/u}))^{-1}\chi_{\rm geo}(e_{-\alpha}^{1/u},B(g))\times\\
&&\times\delta_{n'''}*\delta_{d_{\alpha}(u)}
 *\delta_{e_{-\alpha}^{1/u}B(g)}*\delta_{n'}\\
&=&\chi_{\rm geo}(d_{\alpha}(u),d_{\alpha}({1/u}))^{-1}\chi_{\rm geo}(e_{-\alpha}^{1/u},B(g))\times\\
&&\times\delta_{n'''}*\delta_{d_{\alpha}(u)}
 *\delta_{B(g)B(g)^{-1}e_{-\alpha}^{1/u}B(g)}*\delta_{n'}\\
&=&\chi_{\rm geo}(e_{-\alpha}^{1/u},B(g))\chi_{\rm geo}(d_{\alpha}(u),B(g))\delta_{\alpha g}.
\end{eqnarray*}}
On calcule $\chi_{\rm geo}(e_{-\alpha}^{1/u},B(g))$.
{\allowdisplaybreaks\begin{eqnarray*}
\chi_{\rm geo}(e_{-\alpha}^{1/u},B(g))&=&\chi_{\rm geo}(e_{-\alpha}^{1/u},wt)\\
&=&\chi_{\rm geo}(e_{-\alpha}^{1/u}w,t)\chi_{\rm geo}(e_{-\alpha}^{1/u},w)/\chi_{\rm geo}(w,t)\\
&=&\chi_{\rm geo}(ww^{-1}e_{-\alpha}^{1/u}w,t)\chi_{\rm geo}(e_{-\alpha}^{1/u},w)/\chi_{\rm geo}(w,t)\\
&=&\frac{\chi_{\rm geo}(w,w^{-1}e_{-\alpha}^{1/u}wt)\chi_{\rm geo}(w^{-1}e_{-\alpha}^{1/u}w,t)\chi_{\rm geo}(e_{-\alpha}^{1/u},w)} {\chi_{\rm geo}(w,w^{-1}e_{-\alpha}^{1/u}w)\chi_{\rm geo}(w,t)}\\
&=&\chi_{\rm geo}(e_{-\alpha}^{1/u},w)\text{ (comme $w^{-1}e_{-\alpha}^{1/u}w\in N_r(\fp))$.}
\end{eqnarray*}}
On va prouver que $\chi_{\rm geo}(e_{-\alpha}^{1/u},w)=1$. On a $e_{-\alpha}^{1/u}=e_{\alpha}^ud_{\alpha}(1/u)\alpha e_{\alpha}^u$ et $w^{-1}e_{-\alpha}^{1/u}w=e^{1/u}_{(w^{-1}(\ell+1),w^{-1}(\ell))}\in N_r(\fp)$.
\begin{itemize}
\item Si $v(u)>0$, i.e $v(1/u)<0$, on a $e_{\alpha}^u\op^r=\op^r$. Donc $\mathfrak{d}_{e_{-\alpha}^{1/u}}=(\bigwedge_{i=1}^{v(u)-1}\varpi^i e_{\ell})^*\otimes\left(\bigwedge_{i=-v(u)}^{-1}\varpi^i\begin{pmatrix}u\\1\end{pmatrix}_{\ell}\right)\in (e_\alpha^u\op^r|e_\alpha^ud_\alpha(1/u)\op^r)$, o\`u $\begin{pmatrix}u\\1\end{pmatrix}_{\ell}$ est le vecteur $ue_{\ell}+e_{\ell+1}\in\op^r$.

    On choisit $M=\oplus_{i=1}^{w^{-1}(\ell)-1}\op\oplus\va^{v(u)}\op\oplus\oplus_{i=w^{-1}(\ell)+1}^{r}\op$, de sorte qu'on a $e^{1/u}_{(w^{-1}(\ell+1),w^{-1}(\ell))}M=M$. Donc $\mathfrak{d}_{e^{1/u}_{-\alpha}w}=\mathfrak{d}_{we^{1/u}_{(w^{-1}(\ell+1),w^{-1}(\ell))}}\break=(\bigwedge_{i=1}^{v(u)-1}\varpi^ie_{\ell})^*
    \otimes \left(\bigwedge_{i=0}^{v(u)-1}\frac{1}{u}\varpi^i\begin{pmatrix}u\\1\end{pmatrix}_{\ell}\right)\in (w\op^r|w M)\otimes\break(wM|we^{1/u}_{(w^{-1}(\ell+1),w^{-1}(\ell))}\op^r)$.

    La matrice de passage de $\left(\varpi^i\begin{pmatrix}u\\1\end{pmatrix}_{\ell}\right)_{i=-v(u)}^{-1}$ \`a
    $\left(\frac{1}{u}\varpi^i\begin{pmatrix}u\\1\end{pmatrix}_{\ell}\right)_{i=0}^{v(u)-1}$ est une matrice de taille $v(u)$ : $\m=\left(\begin{smallmatrix}\frac{\varpi^{v(u)}}{u}(0)&*&\cdots&*\\
0&\frac{\varpi^{v(u)}}{u}(0)&\cdots&*\\
\vdots&&\ddots&\vdots\\
0&\cdots&&\frac{\varpi^{v(u)}}{u}(0)\end{smallmatrix}\right).$ Ceci implique $\chi'(e^{1/u}_{-\alpha},w)=\left(\frac{u}{\va^{v(u)}}\right)^{v(u)}(0).$
\item Si $v(u)\leq 0$, i.e $v(1/u)\geq 0$, on a $e^{1/u}_{(w^{-1}(\ell+1),w^{-1}(\ell))}\op^r=\op^r$, donc $\mathfrak{d}_{e_{-\alpha}^{1/u}w}=\mathfrak{d}_w\otimes\mathfrak{d}_{e^{1/u}_{(w^{-1}(\ell+1),w^{-1}(\ell))}}=1$.

    On choisit $M=\oplus_{i=1}^{\ell}\op\oplus\va^{-v(u)}\op\oplus\oplus_{i=\ell+2}^{r}\op$, de sorte qu'on a $e^u_\alpha M=M$. Donc $\mathfrak{d}_{e^{1/u}_{-\alpha}}=(\bigwedge_{i=0}^{-v(u)-1}\va^ie_{\ell+1})^*\otimes
    \left(\bigwedge_{i=0}^{-v(u)-1}\va^i\begin{pmatrix}u\\1\end{pmatrix}_{\ell}\right)\otimes
    \left(\bigwedge_{i=0}^{-v(u)-1}\va^i\begin{pmatrix}u\\1\end{pmatrix}_{\ell}\right)^*\otimes
    \left(\bigwedge_{i=v(u)}^{-1}\va^ie_\ell\right)
    \otimes \left(\bigwedge_{i=0}^{-v(u)-1}-u\va^ie_\ell\right)^*\otimes
    \left(\bigwedge_{i=0}^{-v(u)-1}\va^ie_{\ell+1}\right)
    \in (\op^r|M)\otimes(e^u_\alpha M|e^u_\alpha \op^r)\otimes (e^u_\alpha\op^r|e^u_\alpha d_\alpha(1/u)\op^r)\otimes(e^u_\alpha d_\alpha(1/u)\alpha\op^r|e^u_\alpha d_\alpha(1/u)\alpha M)\otimes(e^{1/u}_{-\alpha}M|e^{1/u}_{-\alpha}\op^r)$.

    La matrice de passage de $\left(\varpi^ie_{\ell}\right)_{i=v(u)}^{-1}$ \`a
    $\left(-u\varpi^ie_{\ell}\right)_{i=0}^{-v(u)-1}$ est une matrice de taille $-v(u)$ : $\m=\left(\begin{smallmatrix}\frac{-u}{\va^{v(u)}}(0)&*&\cdots&*\\
0&\frac{-u}{\va^{v(u)}}(0)&\cdots&*\\
\vdots&&\ddots&\vdots\\
0&\cdots&&\frac{-u}{\va^{v(u)}}(0)\end{smallmatrix}\right).$ Ceci implique $\mathfrak{d}_{e^{1/u}_{-\alpha}}=\left(\frac{-u}{\va^{v(u)}}\right)^{v(u)}(0)$, de sorte qu'on a $$\chi'(e^{1/u}_{-\alpha},w)=\left(\frac{-u}{\va^{v(u)}}\right)^{v(u)}(0).$$
\end{itemize}
Par ailleurs on a
{\allowdisplaybreaks
\begin{eqnarray*}
 \chi_{\rm geo}(e^{1/u}_{-\alpha},w)&=&\frac{\overline{\delta}_{e^{1/u}_{-\alpha}}.\overline{\delta}_w.\chi'(\det(e^{1/u}_{-\alpha}),\det(w))}{\overline{\delta}_{e^{1/u}_{-\alpha}w}.\chi'(e^{1/u}_{-\alpha},w)}\\
 &=&(-1)^{v(u)}\left(\frac{u}{\varpi^{v(u)}}\right)^{-v(u)}/\chi'(e^{1/u}_{-\alpha},w)\\
 &=&1
\end{eqnarray*}}
\end{itemize}
\end{proof}
\subsubsection{Fonction $\kappa$ locale}
Le scindage $\underline{\kappa}^*$ de la proposition~\ref{scindage} provient en fait d'un scindage de l'extension $\tG_{r, KP}(\fp)$ au-dessus de ${\rm GL}_{r}(\op)$.
Ceci d\'efinit une fonction  $\underline{\kappa} : \G_{r}(\op)\to k^*$ telle que (cf. \cite{KP}) $$\underline{\kappa}^*(g)=(g,\underline{\kappa}(g))=(1,\underline{\kappa}(g)).s(g)\,,\ \forall g\in\G_r(\op)\,,$$ (où $s(g)=(g,1)$ est la section de $\G_r(\fp)\ra \tG_r(\fp)$ de Kazhdan-Patterson).

En termes géométriques,  le scindage trivial de $\tG_{r,\rm geo}(\fp)$ au-dessus de $\G_r(\op)$ qui vient de l'isomorphisme canonique $${\rm triv}: \Delta_g\can k\,,\ \forall g\in \G_r(\op)$$ co\"\i ncide avec $s_{\rm geo}$ au-dessus des trois sous-groupes $T_{r}(\op)$,  $W_{r}$ et $N_{r}(\op)$ (cf. la premi\`ere condition de la proposition~\ref{scindage}) et s'identifie donc au scindage $\underline{\kappa}^*$.  Par
 cons\'equent, on a $\underline{\kappa}(g)={\rm triv}/\delta_g$.

\begin{rmk}
En rempla\c cant  $k$ par  une extension finie de $k'$ de $k$, on a une extension de $\G_r(\fp')$ par $(k')^*$ :
$1\ra (k')^*\ra\tG_{r,\rm geo}(\fp')\ra \G_r(\fp)\ra 1$. On obtient alors de la m\^eme mani\`ere une  fonction $\kappa': \G_r(\op')\ra (k')^*$. En composant cette fonction $\underline{\kappa}'$  avec le  morphisme de norme $N_{k'/k}$ on obtient une fonction $\underline{\kappa} :\G_r(\op')\ra k^*$. Cette fonction satisfait que $g\mapsto (g,\underline{\kappa}(g))$ est un scindage de $1\ra k^*\ra N_*\tG_{r,\rm geo}(\fp')\ra\G_r(\fp')\ra 1$ au-dessus de $\G_r(\op')$.
\end{rmk}
\`A l'aide de ce point de vue, on obtient une nouvelle démonstration de la formule pour $\underline{\kappa}$ de Kubota pour $r=2$ (on plutôt, du renforcement où en remplace le symbole de Hilbert par le symbole modéré).
\begin{proposition}\label{kubo}
Soit $g=\begin{pmatrix}a&b\\c&d\end{pmatrix}\in \G_2(\op)$, on a alors $$\underline{\kappa}(g)=\begin{cases}1&\text{si }c=0 \text{ ou } c\in \op^*\\
\left\{c,\frac{d}{\det(g)}\right\}&\text{sinon}. \end{cases}$$
\end{proposition}
\begin{proof}
\begin{itemize}
\item Si $c=0$, la décomposition de Bruhat de $g$ est
$$g=\D(a,d)\begin{pmatrix}1&\frac{b}{a}\\0&1\end{pmatrix}.$$
Comme $g\in \G_2(\op)$, donc $a,d\in \op^*$, de sorte qu'on a $n=\begin{pmatrix}1&\frac{b}{a}\\0&1\end{pmatrix}\in \G_2(\op)$ et $t=\D(a,d)\in \G_2(\op)$. Par conséquent, on a $\delta_g={\rm triv}$, i.e $\underline{\kappa}(g)=1$.
\item Si $c\neq 0$, la décomposition de Bruhat de $g$ est
$$g=\begin{pmatrix}1&\frac{a}{c}\\0&1\end{pmatrix}\D\left(\frac{-\det(g)}{c},c\right)w_0\begin{pmatrix}1&\frac{d}{c}\\0&1\end{pmatrix}.$$
\begin{itemize}
\item Si $c\in \op^*$, les facteurs de la décomposition de Bruhat de $g$ ci-dessus sont dans $\G_2(\op)$, donc on obtient facilement que $\delta_g={\rm triv}$, de sorte qu'on a $\underline{\kappa}(g)=1$.
\item Si $c\not \in \op^*$ (i.e $v(c)>0$), on a $[\delta_g]=(-1)^{v(c)}\left(\frac{c}{\varpi^{v(c)}}\right)^{-v(c)}(0)$. On va noter chaque facteur de la décomposition ci-dessus respectivement (de gauche à droite) $n,t,w_0,n'$. Comme $g\in \G_2(\op)$ et $v(c)>0$, on a $a,d\in \op^*$. Pour trouver $\mathfrak{d}_g$, on choisit $M=\op\oplus\varpi^{v(c)}\op$ (ce réseau satisfait la condition que $nM=M$ et $n'M=M$), de sorte qu'on obtient :
    \begin{multline*}
    \mathfrak{d}_g=(\bigwedge_{j=0}^{v(c)-1}\varpi^je_2)^*\otimes (\bigwedge_{j=0}^{v(c)-1}\varpi^jne_2)\otimes (\bigwedge_{j=0}^{v(c)-1}\varpi^jne_2)^*\otimes \bigwedge_{j=-v(c)}^{-1}\varpi^j ne_1\\\otimes(\bigwedge_{j=0}^{v(c)-1}\varpi^jntw_0e_2)^*\otimes (\bigwedge_{j=0}^{v(c)-1}\varpi^jge_2)\\
    \in (\op^2|M)\otimes (nM|n\op^2)\otimes(n\op^2|nM)\otimes(nM|ntw_0\op^2)\\\otimes(ntw_0\op^2|ntw_0M)\otimes(gM|g\op^2).
    \end{multline*}
    On considère le $k$-espace vectoriel $ntw_0\op^2/ntw_0M$. Celui-ci est muni de  deux bases  : $\{ntw_0\varpi^je_2\}_{j=0\dots v(c)-1}$, qui est l'image de $\{\varpi^je_2\}_{j=0\dots v(c)-1}\in \op^2/M$ par l'isomorphisme $\times ntw_0$, et \break$\{n\varpi^je_1\}_{j=-v(c)\ldots -1}$ qui est l'image de $\{\varpi^je_1\}_{j=-v(c)\ldots -1}\in \break tw_0\op^2/tw_0M$ par l'isomorphisme $\times n$.
    La matrice de passage de la seconde vers la première est la matrice carrée de taille $v(c)$
    $$\m=\begin{pmatrix}\frac{-\det(g)}{c\varpi^{-v(c)}}(0)&*&*&\dots&*\\
&\frac{-\det(g)}{c\varpi^{-v(c)}}(0)&*&\dots&*\\
&&&\ddots&\vdots\\
&&&&\frac{-\det(g)}{c\varpi^{-v(c)}}(0)
\end{pmatrix}.$$
Par ailleurs, on a $ntw_0M=M$ (ce qui implique $gM=M$) et $g\op^2=\op^2$, donc $g\op^2/gM=\op^2/M$, de sorte que le $k$-espace vectoriel $\op^2/M$ est muni des deux bases $\{\varpi^je_2\}_{j=0\ldots v(c)}$ et $\{g\varpi^je_2\}_{j=0\ldots v(c)}$ qui est l'image de $\{\varpi^je_2\}_{j=0\ldots v(c)}\in \op^2/M$ par l'isomorphisme $\times g$. La matrice de passage de la seconde vers la première est la matrice carrée de taille $v(c)$
    $$\m'=\begin{pmatrix}d(0)&*&*&\dots&*\\
&d(0)&*&\dots&*\\
&&&\ddots&\vdots\\
&&&&d(0)
\end{pmatrix}.$$
Par conséquent, on a $\mathfrak{d}_g=\frac{\det(\m')}{\det(\m)}{\rm triv}$. En utilisant la formule  $\delta_g=[\delta_g]\mathfrak{d}_{\det(g)}\otimes\mathfrak{d}_g^{\otimes (-1)}$, on a donc que $\underline{\kappa}(g)=\left\{c,\frac{d}{\det(g)}\right\}$.
\end{itemize}
\end{itemize}
\end{proof}
\subsection{Le groupe métaplectique global}\label{glokap}
\subsubsection{La construction du groupe métaplectique $S$-global}
Soient $\0:=k[\va]$ l'anneau des polynômes à une indéterminée $\va$ et à coefficients dans $k$, $F=k(\va)$ son corps des fractions. Soit $S$ un ensemble fini des points fermés de $\mathbb{P}^1$ (qui seront considérés comme des places de $F$). On associe à chaque $g\in \G_r(F)$ la droite  $\Delta^{S}_g:=\otimes_v \Delta_{g,v}$ 5(grâce à la règle de Koszul cette définition ne dépend pas du choix d'un ordre sur les places $v$), où $\Delta_{g,v}=(\0_v^r|g\0_v^r)_k^{\otimes -1}\otimes (\0_v|\det(g)\0_v)_k$ ($\0_v$ est l'anneau d'entiers de $F$ en place $v$ ; $\0_v^r$, $g\0_v^r$, $\0_v$ et $\det(g)\0_v$ sont vu comme des $k$-espaces vectoriels).
Par ailleurs, dans le cas où $S\neq \emptyset$, on a le lemme suivant :
\begin{lemma}\label{ka3} Pour $S\neq \emptyset$, soit $\0(S)$ l'anneaux des fractions rationnelles dont les pôles sont dans $S$. On a un isomorphisme canonique $D^S_g\can (\0(S)^r|g\0(S)^r)$.
Par conséquent, on a un isomorphisme canonique $\Delta_{g}^S\can (\0(S)|\det(g)\0(S))\otimes(\0(S)^r|g\0(S)^r)^{\otimes -1}$.
\end{lemma}
\begin{proof} On note $L=\0(S)^r\cap g\0(S)^r$. En utilisant le lemme chinois, si $M'\subset M$ sont deux $\0(S)$-réseaux dans $F^r$ on obtient des isomorphismes canoniques
$$M/M'\simeq\bigoplus_{v\not\in S}(M\otimes_{\0(S)}\0_v)/(M'\otimes_{\0(S)}\0_v),$$
d'où des isomorphismes canoniques
\begin{eqnarray*}
(\0(S)^r|g\0(s)^r)&\simeq&(\bigwedge \0(S)^r/L)^*\otimes (\bigwedge g\0(S)^r/L)\\
&\simeq&\left(\bigwedge \bigoplus_{v\not\in S}(\0_v^r)/(L\otimes_{\0(S)}\0_v)\right)^*\otimes\\
&&\otimes\left(\bigwedge\bigoplus_{v\not\in S}(g\0_v^r)/(L\otimes_{\0(S)}\0_v)\right)\\
&\simeq&{\bigotimes}_{v\not\in S}'D_{g,v}\,.\qedhere
\end{eqnarray*}
\end{proof}

En utilisant ce lemme,  on obtient de nouveau la définition de la dro\-ite $\Delta_g^S$. Cette construction fournit une extension centrale de $\G_r(F)$ par $k^*$ :
$$1\ra k^*\ra \tG_r^S(F) \ra \G_r(F) \ra 1.$$
On va appeler $\tG_r^S(F)$ \textit{le groupe métaplectique $S$-global}.
Comme dans  le cas local envisagé dans la partie précédente, cette extension est scindée au-dessus de $N_r(F)$, et au-dessus de $\G_r(\0(S))$. On construit aussi comme dans  le cas local une section ensembliste $$s_{\rm geo} : \G_r(F)\ra\tG_r^S(F),\ g\mapsto (g,\delta_g).$$

Pour toute place $v$ de $F$, on note $F_v$ la complétion en $v$ de $F$, $\0_v$ son anneau des entiers et $k_v$ son corps résiduel. En utilisant que $k_v$ est une extension finie de $k$ on a une extension (qu'on appelle le groupe métaplectique local en la place $v$) :
$$1\ra k^*\ra N_*\tG_r(F_v)\ra\G_r(F_v)\ra 1,$$ une fonction $\underline{\kappa}_v :\G_r(\0_v)\ra k^*$ (voir ci-dessus pour la définition) et un scindage $\underline{\kappa}^*_v:\G_r(\0_v)\ra N_*\tG_r(F_v) \ g\mapsto (g,\underline{\kappa}_v(g))$.

D'après Kazhdan-Patterson, on a l'extension suivante (voir \cite[section 2]{KP} pour la construction quand on remplace le symbole de Hilbert par le symbole modéré) :
$$\xymatrix{ 1 \ar[r]&\prod_{v\not\in S}'k_v^*\ar[r]&\prod_{v\not\in S}'\tG_{r}(F_v)\ar[r]&\prod_{v\not\in S}'\G_r(F_v)\ar[r]&1,
}$$
où
\begin{itemize}
\item $\prod_{v\not\in S}'k_v^*$ est le groupe form\'e des éléments $(x_v)_{v\not \in S}$ où $x_{v}\in k_v^*$ et $x_{v}=1$ pour presque tout $v$.
\item $\prod_{v\not \in S}'\G_{r}(F_v)$ est le groupe form\'e des éléments $(g_v)_{v\not \in S}$ où $g_v\in \G_r(F_v)$ et $g_v\in \G_r(\0_v)$ pour presque tout $v$.
\item $\prod_{v\not \in S}'\tG_{r}(F_v)$ est le groupe form\'e des éléments $(\widetilde{g}_v)_{v\not \in S}$ où $\widetilde{g}_v\in \tG_r(F_v)$ et, pour presque tout $v$, $\widetilde{g}_v\in \tG_r(\0_v)$ et $\widetilde{g}_{v}=\underline{\kappa}^*_v(g_v)$.
\end{itemize}
En poussant cette extension via le morphisme $\prod'_{v \not \in S}k_v^*\ra k^*,\ (x_v)\mapsto \prod_{v\not\in S}N_{k_v/k}(x_v)$, on obtient une extension de $\G_r(\aS)$ par $k^*$ (où $\aS:=\prod_{v\not \in S}'F_v$) :
$$1\ra k^*\ra \tG_r(\aS)\ra\G_r(\aS)\ra 1.$$
\begin{proposition}\label{ka4} On a un morphisme d'extensions :
$$\xymatrix{ 1\ar[r]&k^*\ar[r]\ar[d]&\ar @{} [dr] |{\square}\tG_{r}^S(F)\ar[r]\ar[d]&\G_r(F)\ar[r]\ar@{_{(}->}[d]&1\\
1\ar[r]&k^*\ar[r]&\tG_r(\aS)\ar[r]&\G_r(\aS)\ar[r]&1,
}$$
pour lequel le carr\'e de droite est cart\'esien.
\end{proposition}
\begin{proof}
L'assertion résulte immédiatement de l'identité $\tG_r(F_v)\simeq\tG_{r,\rm geo}(F_v)$ (cf. théorème~\ref{chi4}) et du lemme~\ref{ka3}.
\end{proof}
\begin{rmk}\label{rmk3}
L'extension globale est triviale au-dessus de $\G_r(F)$ lorsqu'on prend $S=\emptyset$ (cf.\cite[p. 50-51]{KP}), i.e $\tG_r(F):=\tG_r^\emptyset(F)$ est une extension scindée. On va le voir géométriquement en la reliant au déterminant de la cohomologie.

Soient $\mathcal{C}$ une courbe lisse (connexe) projective sur $k$, $F$ son corps des fonctions rationelles. On note $\mathbb{A}_{\mathcal{C}}=\prod'_{v\in |\mathcal{C}|}F_v$ et $\0_{\mathbb{A}_{\mathcal{C}}}=\prod_{v\in |\mathcal{C}|}\0_v$ où $\0_v$ est l'anneau des entiers de $F_v$. Soient $\mathcal{F}$, $\mathcal{G}$ deux fibrés vectoriels de rang $r$ sur $\mathcal{C}$ munis un isomorphisme générique entre eux (i.e un isomorphisme $\mathcal{F}_{|U}\simeq \mathcal{G}_{|U}$ défini sur un ouvert $U$ assez petit).\`A l'aide de cet isomorphisme on a que $\mathcal{G}\otimes \0_{\mathbb{A}_{\mathcal{C}}}$ est un $\0_{\mathbb{A}_{\mathcal{C}}}$-réseau dans $\mathcal{F}\otimes \mathbb{A}_{\mathcal{C}}$ et vice versa, de sorte qu'on définit $(\mathcal{G}\otimes \0_{\mathbb{A}_{\mathcal{C}}}|\mathcal{F}\otimes \0_{\mathbb{A}_{\mathcal{C}}})$ comme dans la section 3.1. On a la proposition suivante
\begin{proposition}\label{dc}
Il existe un isomorphisme canonique :
$$\det (R\Gamma\mathcal{G})\otimes (\det(R\Gamma\mathcal{F}))^{\otimes -1}\simeq (\mathcal{F}\otimes \0_{\mathbb{A}_{\mathcal{C}}}|\mathcal{G}\otimes \0_{\mathbb{A}_{\mathcal{C}}})$$
\end{proposition}
\begin{proof}
On peut supposer que $\mathcal{F}\subset \mathcal{G}$ (sinon, on considère un troisième fibré vectoriel $\mathcal{H}$ muni d'un isomorphisme générique avec $\mathcal{G}$ tel que $\mathcal{H}\otimes \0_{\mathbb{A}_{\mathcal{C}}}\subset (\mathcal{F}\otimes\0_{\mathbb{A}_{\mathcal{C}}}\cap\mathcal{G}\otimes \0_{\mathbb{A}_{\mathcal{C}}})$ dans $\mathcal{G}\otimes \0_{\mathbb{A}_{\mathcal{C}}}$ et est ramené à démontrer la proposition pour les deux couples $(\mathcal{H},\mathcal{F})$ et $(\mathcal{H},\mathcal{G})$).

On considère la suite exacte :
$$0\to \mathcal{F}\hookrightarrow \mathcal{G}\to \mathcal{G}/\mathcal{F} \to 0.$$
Par conséquent, on obtient la suite exacte longue :
$$0\to H^0(\mathcal{F})\to H^0(\mathcal{G}) \to H^0(\mathcal{G}/\mathcal{F}) \to H^1(\mathcal{F})\to H^1(\mathcal{G})\to 0.$$
Elle implique {\small$(\det(H^0(\mathcal{F})))^{\otimes -1}\otimes \det(H^0(\mathcal{G})\otimes (\det(H^0(\mathcal{G}/\mathcal{F})))^{\otimes -1}\otimes \det(H^1(\mathcal{F}))\otimes (\det(H^1(\mathcal{F})))^{\otimes -1}\simeq k$}, de sorte qu'on a:
$$\det (R\Gamma\mathcal{G})\otimes (\det(R\Gamma\mathcal{F}))^{\otimes -1}\simeq (\det(H^0(\mathcal{G}/\mathcal{F})))^{\otimes -1}\simeq (\mathcal{F}\otimes \0_{\mathbb{A}_{\mathcal{C}}}|\mathcal{G}\otimes \0_{\mathbb{A}_{\mathcal{C}}}). $$
\end{proof}
En utilisant cette proposition dans le cas où $\mathcal{F}=\0_{\mathcal{C}}^r$ et $\mathcal{G}$ est le fibré vectoriel associé à $g\in \G_r(F)\setminus \G_r(\mathbb{A}_{\mathcal{C}})/\G_r(\0_{\mathbb{A}_{\mathcal{C}}})$, on a :
$$\det(R\Gamma\mathcal{G})\otimes \det(R\Gamma\0_{\mathcal{C}}^r)^{\otimes -1}\simeq D_g.$$
Quand $g\in \G_r(F)$, on a $\mathcal{G}\simeq \0_{\mathcal{C}}^r$, de sorte qu'on obtient des isomorphisme canoniques
$D_g\can k$ et $\Delta_g \can k$.

Par ailleurs, on a le diagramme commutatif suivant :
$$\xymatrix{D_g\otimes D_{g'}\ar[r]\ar[d]_-{\times g'}& (\det(R\Gamma\mathcal{G})\otimes \det(R\Gamma\0_{\mathcal{C}}^r)^{\otimes -1})\otimes (\det(R\Gamma\mathcal{G}')\otimes \det(R\Gamma\0_{\mathcal{C}}^r)^{\otimes -1}\ar[d]^-{\times g'}\\
D_{gg'}&(\det(R\Gamma\mathcal{G})\otimes \det(R\Gamma\0_{\mathcal{C}}^r)^{\otimes -1})\otimes (\det(R\Gamma\mathcal{G}'')\otimes \det(R\Gamma\mathcal{G})^{\otimes -1},\ar[l]}$$
où $\mathcal{G'}$ et $\mathcal{G}''$ sont des fibrés vectoriels de rang $r$ associés à $g'$ et $gg'$.
Alors, le morphisme $D_g\can k$ qu'on vient de défnir est compatible avec la multiplication du groupe $\tG_r(F)$, de sorte qu'on obtient un scindage au dessus de $\G_r(F)$.

\end{rmk}
\subsubsection{Fonction $\emph{\kappa}$ globale}
Clairement, d'après la construction du groupe métaplectique $S$-global on a un scindage canonique $\underline{\kappa}^*_S$ au-dessus $\G_r(\0(S))$ (qui vient de l'isomorphisme $\Delta_g(S)\can k, \,\forall g\in \G_r(\0(S))$). De plus, on a :
\begin{itemize}
\item ${\underline{\kappa}_S^*}_{|T_r(\0(S))}=s_{{\rm geo}_{|T_r(\0(S))}}$,
\item ${\underline{\kappa}_S^*}_{|W_r}=s_{{\rm geo}_{|W_r}}$
\item ${\underline{\kappa}_S^*}_{|N_r(\0(S))}=s_{{\rm geo}_{|N_r(\0(S))}}$.
\end{itemize} D'après \cite[proposition 0.1.3]{KP}, $\underline{\kappa}^*_S$ est alors un scindage canonique au sens de Kazhdan-Patterson. On considère donc l'application $\underline{\kappa}_S : \G_r(\0(S))\break\ra k^*$ qui est définie par :
$$\underline{\kappa}_S(g)=\underline{\kappa}_S^*(g).s(g)^{-1}$$
(la loi du groupe de $\tG_r^S(F)$ étant notée multiplicativement.)

On note $g_v$ la matrice $g$ vue comme matrice à coefficients dans $F_v$. D'après la proposition~\ref{ka4}, on a :
\begin{proposition}\label{glo4}
$$\underline{\kappa}_S(g)=\prod_{v\not\in S}\underline{\kappa}_v(g_v)\quad \forall\,g\in \G_r(\0(S)).$$
\end{proposition}

\section{Intégrale $J$}
\subsection{Somme locale}
\setcounter{theorem}{0}
\setcounter{subsubsection}{1}
\quad \  Pour tout $\underline{\alpha}=(\alpha_2,\dots,\alpha_{r})\in {(k^*)}^{r-1}$, on note $\theta'_{\underline{\alpha}} : N_r(\fp)\rightarrow \overline{\mathbb{Q}}^{*}_{\ell}$ le caractère défini par $\theta'_{\underline{\alpha}}(n)=\Psi(\frac{1}{2}\sum^{r}_{i=2}\alpha_i n_{i-1,i})$, et $\overline{\alpha}=(\alpha_{r},\dots,\alpha_2)$. La restriction de $\theta'_{\underline{\alpha}}$ à $N_r(\op)$ étant triviale, elle induit une fonction $\theta'_{\underline{\alpha}}$ sur $N_r(\fp)/N_r(\op)$ à valeurs dans $\overline{\mathbb{Q}}^{*}_{\ell}$. Soit $\zeta:k^*\ra\{\pm1\}$ le caract\`ere quadratique non trivial ($\zeta(\lambda)=\lambda^{\frac{q-1}{2}},\,\lambda\in k^*$).

Pour chaque $t=\D(a_1,a_2/a_1,\dots,a_r/a_{r-1})\in T_r(\fp)$, on considère l'ensemble fini (cf. \cite[proposition 1.1.2]{N})
$$Y_{\varpi}(t)(k)=\{(n,n')\in (N_r(\fp)/N_r(\op))^2 | {}^t ntn' \in \G_r(\op) \}.$$
En utilisant le changement de variable $n\mapsto {}^tn''=w_0n^{-1}w_0$, l'intégrale orbitale $J$ de Jacquet-Mao peut s'écrire :
\begin{eqnarray*}
J_{\varpi}(t,\underline{\alpha})&=&\sum_{(n'',n') \in Y_{\varpi}(t)(k)}\underline{\kappa}(w_0{}^tn''tn)\theta'_{\overline{\alpha}}(w_0{}^tn''w_0)\theta'_{\underline{\alpha}}(n')\\
&=&\sum_{(n,n') \in Y_{\varpi}(t)(k)}\underline{\kappa}(w_0{}^tntn)\theta'_{\overline{\alpha}}(w_0{}^tnw_0)\theta'_{\underline{\alpha}}(n').
\end{eqnarray*}

L'ensemble $Y_{\varpi}(t)(k)$ est de manière naturelle l'ensemble des points à valeurs dans $k$ d'une variété algébrique $Y_{\varpi}(t)$ de type fini sur $k$. Cette variété est munie d'un morphisme $h'_{\underline{\alpha}}:Y_{\varpi}(t)\ra \mathbb{G}_a$ défini par
\begin{eqnarray*}
h'_{\underline{\alpha}}(n,n')&=&\sum_{i=2}^{r}\frac{1}{2}\mathrm{res}(\alpha_{r+2-i} n_{r+1-i,r+2-i}d\va)+\sum_{i=2}^{r}\frac{1}{2}\mathrm{res}(\alpha_in'_{i-1,i}d\va)\\
&=&\sum_{i=2}^{r}\frac{1}{2}\mathrm{res}(\alpha_{i} (n_{i-1,i}+n'_{i-1,i})d\va).
\end{eqnarray*} et d'un morphisme $\underline{\kappa} :Y_{\va}(t)\ra \mg_m$ qu'on va maintenant construire (la fonction $\kappa: Y_{\va}(t)(k)\to\{\pm 1\}$ est alors définie par $\kappa=\zeta\circ\underline{\kappa}$).
\begin{const}\label{geokap} Soit $R$ une $k$-algèbre. On note $\0_{\varpi,R}=\op\otimes_k R$ et $F_{\varpi,R}=\fp\otimes_k R$.
On va géométriser la fonction $\underline{\kappa}$ à l'aide des résultats de la section 3.
\begin{enumerate}
\item Soit $g\in \G_r(F_{\varpi,R})$. On va définir une $R$-droite (plus précisément un $R$-module inversible) associée à $g$. Pour cela on a besoin du lemme suivant :
\begin{lemma}(cf.\cite{AF})\label{morphisme1}
Soient $M$ et $M'$ deux $\0_{\varpi,R}^r$-réseaux dans $F_{\varpi,R}^r$ tels que $M'\subset M$. Le $R$-module $M/M'$ est alors localement libre.
\end{lemma}

Soit $g\in \G_r(F_{\varpi,R})$. On choisit un $\0_{\varpi,R}^r$-réseau $M$ tel que $M\subset g\0_{\varpi,R}^r\cap\0_{\varpi,R}^r$. En utilisant le lemme~\ref{morphisme1} ci-dessus, $g(\0_{\varpi,R}^r)/M$ et $\0_{\varpi,R}/M$ sont des $R$-modules localements libres. On considère alors les puissances extérieures de degré maximal $\bigwedge (g(\0_{\varpi,R}^r)/M)$, $\bigwedge(\0_{\varpi,R}^r/M)$ et on note $D_{g,R}$ la $R$-droite  $$(\0_{\varpi,R}^r|g\0_{\varpi,R}^r)=(\bigwedge(\0_{\varpi,R}^r/M)^*\otimes_R\bigwedge (g(\0_{\varpi,R}^r)/M).$$
Cette définition ne dépend pas du choix de $M$. En particulier, pour $g\in \G_r(\0_{\varpi,R})$ cette droite est canoniquement trivialisée.

Comme dans la section 3, on pose $\Delta_{g,R}=D_{\det(g),R}\otimes D_{g,R}^{\otimes -1}$.
\item  On trouve une trivialisation naturelle de $\Delta_{n,R}$ pour $n\in N_r(F_{\varpi,R})$, à l'aide des deux généralisations suivantes des lemmes~\ref{n2} et~\ref{n3} (les démonstrations données plus haut s'étendent sans problème).
\begin{lemma}
Soit $n \in N_r(F_{\varpi,R})$, il existe alors un $\0_{\varpi,R}^r$-r\'eseau $M$  tel que $nM=M$
\end{lemma}
\begin{lemma}
Soit $n\in N_r(F_{\varpi,R})$.Soient $M$, $M'$ deux $\0_{\varpi,R}^r$-r\'eseaux tels que   $nM=M$ et $nM'=M'$. Le diagramme
$$\xymatrix{(M|nM)\ar[rr]^-{(~\ref{n1})}\ar[dr]_-{\textrm{\'evident}}&&(M'|nM')\ar[dl]^-{\textrm{\'evident}}\\
&k.
}$$
est alors commutatif.
\end{lemma}
\item On construit une base $\delta_R(w_0t)$ de $\Delta_{w_0t,R}$ en prenant la construction~\ref{chi5}.
\item Comme $\Delta_{nw_0tn'}\can \Delta_{n}\otimes\Delta_{w_0t}\otimes\Delta_{n'}$, on obtient une base $\delta_R(nw_0tn')$ de $\Delta_{nw_0tn',R}$ et on pose $\underline{\kappa}(nw_0tn')=\delta_R(nw_0tn')/{\rm triv}$ pour $nw_0tn'\in \G_r(F_{\varpi,R})$.
\end{enumerate}
La fonction $\underline{\kappa} : Y(t)(R)\to R^*$ est définie par $(n,n')\mapsto \underline{\kappa}((w_0{}^tnw_0)w_0tn')$. Comme la construction ci-dessus commute au changement de base et généralise celle de la section 3 (dans le cas où $R$ est un corps) on obtient bien un morphisme $\underline{\kappa} :Y(t)\to \mg_m$.
\end{const}

On note $\overline{Y}_{\varpi}(t)=Y_{\varpi}(t)\otimes_k\overline{k}$. Soit $\mathcal{L}_{\zeta}$ le faisceau de Kummer sur $\mathbb{G}_m$ associé au revêtement $\mathbb{G}_m\ra\mathbb{G}_m,\, x\mapsto x^2$ et au caractère non trivial $\zeta$ de $\{\pm1\}$. D'après la formule des traces de Grothendieck-Lefschetz, on a :
$$ J_{\varpi}(t,\underline{\alpha}):=\mathrm{Tr}(\mathrm{Fr},\mathrm{R}\Gamma_c(\overline{Y}_{\varpi}(t),{h'}_{\underline{\alpha}}^*\mathcal{L}_{\psi}\otimes\underline{\kappa}^*\mathcal{L}_{\zeta})).$$
\subsection{Somme globale}
\setcounter{theorem}{0}
\setcounter{subsubsection}{1}
\quad Comme pour  l'intégrale $I$, on ne connaît pas explicitement $\mathrm{R}\Gamma_c(\overline{Y}_{\varpi}(t),\break{h'}_{\underline{\alpha}}^*\mathcal{L}_{\psi}\otimes\underline{\kappa}^*\mathcal{L}_{\zeta})$ car la variété $\overline{Y}_{\varpi}(t)$ est trop compliquée, donc on a besoin de  globaliser. Dans cette partie, on va introduire deux d\'efinitions de  l'intégrale $J$ globale. La première utilise les int\'egrales locales  et la propriété de multiplicativit\'e et la deuxième est géométrique. Bien entendu, ces deux d\'efinitions donnent le m\^eme r\'esultat (voir le lemme~\ref{compint}).

\begin{enumerate}

\item Première d\'efinition

On rappelle que ${\0}={k}[\varpi]$ est l'anneau des polynômes en une variable $\varpi$ à coefficients dans ${k}$, et ${F}$ est son corps des fractions. Pour toute place $v$, ${\0}_v$ est le complété de ${\0}$ en $v$, ${F}_v$ est son corps des fractions, et $k_v$ est son corps résiduel. On note $t_v$ l'image de $t$ dans  $T_r(F_v)$.

Soit $t=\D(a_1,a_2/a_1,\dots,a_r/a_{r-1})$ où $a_i\in \0-\{0\}$. Désormais, on note
$Y_v(t)=\{(n,n')\in (N_r(F_v)/N_r(\0_v))^2| {}^tnt_vn'\in \mathfrak{gl}_r(\0_v) \}$ et $\theta'_{\underline{\alpha}}(n_v)=\psi(\sum_{i=2}^r\alpha_i tr_{k_v/k}\mathrm{res}_{v}(n_{i-1,i}{\rm d}\varpi))$.

Pour toute place $v\not | a_r $, on pose
$$J_v(t,\underline{\alpha})=\int_{N_r(F_v)\times N_r(F_v)}f_v(w_0{}^tnt_vn')\theta'_{\overline{\alpha}}(w_0{}^tnw_0)\theta'_{\underline{\alpha}}(n')dndn',$$
où $f_v(w_0{}^tnt_vn')=\begin{cases}\zeta(\underline{\kappa}_v(w_0{}^tnt_vn'))& \textrm{si }w_0{}^tnt_vn'\in \G_r(\0_v)\\0&\textrm{sinon}\end{cases}$.\\En utilisant que $a_r\in \0_v^*$ implique ${}^tN_r(F_v)t_vN_r(F_v)\cap \G_r(\0_v)={}^tN_r(F_v)t_vN_r(F_v)\cap \mathfrak{gl}_r(\0_v)$ (cf. \cite[proposition 1.1.1]{N}), on peut aussi écrire :
$$J_v(t,\underline{\alpha})=\sum_{(n,n')\in Y_v(t)}\zeta(\underline{\kappa}_v(w_0{}^tnt_vn'))\theta'_{\overline{\alpha}}(w_0{}^tnw_0)\theta'_{\underline{\alpha}}(n').$$

Pour toute place $v|a_r$, on pose
$$J_v(t,\underline{\alpha})=\int_{N_r(F_v)\times N_r(F_v)}\mathbb{I}_{\mathfrak{gl}_r(\0_v)}(w_0{}^tnt_vn')\theta'_{\overline{\alpha}}(w_0{}^tnw_0)\theta'_{\underline{\alpha}}(n')dndn',$$
où $\mathbb{I}_{\mathfrak{gl}_r(\0_v)}$ est la fonction caractéristique de l'ensemble $\mathfrak{gl}_r(\0_v)$. Cette intégrale s'\'ecrit aussi :
$$J_v(t,\underline{\alpha})=\sum_{(n,n')\in Y_v(t)}\theta'_{\overline{\alpha}}(w_0{}^tnw_0)\theta'_{\underline{\alpha}}(n').$$
\begin{defn} \label{glo1}
On introduit l'intégrale orbitale $J$ globale :
$$J^{(1)}(t,\underline{\alpha})=\displaystyle\prod_{v\neq\infty} J_{v}(t,\underline{\alpha}).$$
\end{defn}
\begin{lemma}\label{glo3}Si $v\not |\prod_{i=1}^{r-1}a_i$, alors $J_v(t,\underline{\alpha})=1.$
\end{lemma}
\begin{proof}[Démonstration]
D'après \cite[corollaire 1.1.5]{N}, $Y_v(t)$ est réduit à l'élément $(n,n')=(Id_r,Id_r)$ (car $a_i\in\0_v^*\ \forall i\in\{1\dots,r-1\}$. Donc
$$ J_v(t,\underline{\alpha})=\begin{cases}\zeta(\underline{\kappa}_v (w_0t_v))&\text{si }v\not|a_r\\
1&\text{si }v|a_r.\end{cases}$$

Si $v\nmid a_r$, on a alors $t_v\in \G_r(\0_v)$. En utilisant la propriété de la fonction $\underline{\kappa}_v$, on a:
\begin{eqnarray*}
\underline{\kappa}_v(w_0t_v)&=&\underline{\kappa}_v(w_0)\underline{\kappa}_v(t_v)\chi_v(w_0,t_v)\\
&=&\chi_v(w_0,t_v)
\end{eqnarray*}

A l'aide de la formule de définition de $2$-cocycle $\chi_v$, on a : $\chi_v(w_0,t_v)=1$.
Par conséquent, $J_v(t,\underline{\alpha})=1$.\qedhere

\end{proof}

\item Deuxième d\'efinition :

Soit $t=\D(a_1,a_2/a_1,\dots,a_r/a_{r-1})$ où $a_i\in \0-\{0\}$. Pour tout $x \in {F}$, $\mathrm{sres}(x{\rm d}\varpi)$ est la somme des résidus en tous ses pôles  à distance finie (c.à.d en les points de $\mathrm{Spec}(\0)$).  On notera $\Psi$ le caractère de $F$ défini par $\Psi(x)=\psi(\mathrm{sres}(x{\rm d}\varpi))$. Pour tout $\underline{\alpha}\in {(k^*)}^{r-1}$, on note $\theta'_{\underline{\alpha}}: N_r(F)\rightarrow \overline{\mathbb{Q}}^{*}_{\ell}$ le caractère défini par $\theta'_{\underline{\alpha}}(n)=\Psi(\frac{1}{2}\sum^{r}_{i=2}\alpha_i n_{i-1,i})$. Sa restriction à $N_r(\0)$ étant triviale, elle induit une fonction $\theta'_{\underline{\alpha}}$ sur $N_r(F)/N_r(\0)$ à valeurs dans $\overline{\mathbb{Q}}^{*}_{\ell}$.

On introduit la variété de type fini $Y(t)$ dont l'ensemble des $k$ points est
$$Y(t)(k)=\{(n,n')\in (N_r(F)/N_r(\0)^2|{}^tntn'\in \mathfrak{gl}_r(\0)\}.$$
Cette variété est munie d'un morphisme $$h'_{\underline{\alpha}}(n,n')=\frac{1}{2}\sum_{i=2}^r\alpha_i\mathrm{sres}((n_{i-1,i}+n'_{i-1,i})d\varpi),$$
et d'un morphisme $\underline{\kappa} :Y(t)\ra \mg_m$ défini par
$$\underline{\kappa}(n,n')=\prod_{v|\prod_{i=1}^{r-1}} \underline{\kappa}_v(w_0{}^tn_vt_vn'_v),$$
où les fonctions $\underline{\kappa}_v$ sont celles définies dans la section 3.2.
\begin{defn}\label{glo2}On introduit l'intégrale $J$ globale :
$$J^{(2)}(t,\underline{\alpha})=\sum_{(n,n')\in Y(t)(k)}\zeta(\underline{\kappa}(w_0{}^tntn'))\theta'_{\overline{\alpha}}(w_0{}^tnw_0)\theta'_{\underline{\alpha}}(n').$$
\end{defn}

\end{enumerate}
\begin{proposition}\label{compint} Lorsque ${\rm pgcd}(a_r,\prod_{i=1}^{r-1}a_i)=1$,  les deux définitions dans ~\ref{glo1} et ~\ref{glo2} donnent le m\^eme r\'esultat.
\end{proposition}
\begin{proof}[Démonstration] 
On rappelle que $\mathrm{supp}(t)=\mathrm{Spm}(\0/\prod_{i=1}^{r-1} t_i)$.\\
D'après le lemme ~\ref{glo3}, on a  :  $J^{(1)}(t,\underline{\alpha})=\prod_{v\in \mathrm{supp}(t)}J_v(t,\underline{\alpha})$.\\
D'après la propriété multiplicative de la fonction $\underline{\kappa}$ globale et \cite[proposition 1.3.2]{N} on a :
$J^{(2)}(t,\underline{\alpha})=\prod_{v\in \mathrm{supp}(t)}J_v(t,\underline{\alpha}).$

Par conséquent, on a : $J^{(1)}(t,\underline{\alpha})=J^{(2)}(t,\underline{\alpha})$.
\end{proof}
On ajoute une barre pour indiquer le changement de corps de $k$ à $\overline{k}$. On définit les triples $(\overline{Y}(t),h'_{\underline{\alpha}},\underline{\kappa})$ et $(\overline{Y}_v(t),h'_{\underline{\alpha},v},\underline{\kappa}_v)$, où $v\in \mathrm{Spm}(\overline{\0})$ comme on a fait pour le cas local (4.1). On note encore $\mathrm{supp}(t)=\mathrm{Spm}(\overline{\0}/\break\prod_{i=1}^{r-1}a_i\overline{\0})$. En utilisant $\underline{\kappa}(w_0{}^tntn')=\prod_{v\in \mathrm{supp}(t)}\underline{\kappa}_v(w_0{}^tn_vt_vn'_v)$, on a une forme  cohomologique de la proposition~\ref{compint} (cf. \cite[corollaire 3.2.3]{N} pour la démonstration) : quand ${\rm pgcd}(a_r,\prod_{i=1}^{r-1}a_i)=1$,
$$\mathrm{R}\Gamma_c(\overline{Y}(t),{h'}_{\underline{\alpha}}^{*}\mathcal{L}_{\psi}\otimes \underline{\kappa}^*\mathcal{L}_{\zeta})=\bigotimes_{v\in \mathrm{supp}(t)}\mathrm{R}\Gamma_c(\overline{Y}_{v}(t),{h'}_{\underline{\alpha},v}^{*}\mathcal{L}_{\psi}\otimes\underline{\kappa}_v^*\mathcal{L}_{\zeta}).$$

Supposons $t=\D(a_1,\frac{a_2}{a_1},\dots,\frac{a_r}{a_{r-1}})$, où les $a_i$ sont des polynôme unitaires dont on fixera les degrés $d_i=\deg(a_i)$. On note $V_{\ud}=\{(a_1,\dots,a_r)\in Q_{\ud}\,|\mathrm{pgcd}(\prod_{i=1}^{r-1}a_i,a_r)=1\}$. Le triple $({Y}(t),h'_{\underline{\alpha}},\underline{\kappa})$ se mettent en famille  de  sorte  qu'on  obtient  une  variété ${Y}_{\ud}$ de  type  fini  sur ${k}$ munie de trois morphismes $f_{\ud}^Y:{Y}_{\ud}\times\mathbb{G}_m^{r-1}\ra V_{\ud}\times \mathbb{G}_m^{r-1}$,  $h'_{\ud}:{Y}_{\ud}\times\mathbb{G}_m^{r-1}\ra\mathbb{G}_a$, et $\underline{\kappa}_{\ud}:Y_{\ud}\times\mg_m^{r-1}\ra\mg_m$ tels que $\overline{Y}(t)$ et $ \mathrm{R}\Gamma_c(\overline{Y}(t),{h'}_{\underline{\alpha}}^{*}\mathcal{L}_{\psi}\otimes\underline{\kappa}^*\mathcal{L}_{\zeta})$ sont respectivement les fibres en $(t,\underline{\alpha})\in (V_{\ud}\times \mg_m^{r-1})(\overline{k})$ de $f_{\ud}^Y$ et de $\mathrm{R}f^Y_{\ud,!}({h'}_{\ud}^{*}\mathcal{L}_{\psi}\otimes\underline{\kappa}_{\ud}^*\mathcal{L}_{\zeta})$. En précisant, on a :
\begin{lemma}(cf. \cite[proposition 3.3.1]{N})
\begin{enumerate}
\item Pour tout $\underline{d}\in \mathbb{N}^r$ le foncteur $Y_{\underline{d}}$ qui associe à toute $k$-algèbre $R$ l'ensemble
\begin{eqnarray*}Y_{\underline{d}}(R)&=&{}^tN_r(\0\otimes_kR)\setminus\{g\in \mathfrak{gl}_r(\0\otimes_k R)|\det(g_i)\in Q_{d_i}(R),\\
&&\mathrm{pgcd}(\prod_{i=1}^{r-1}\det(g_i),\det(g_r))=1\}/N_r(\0\otimes_k R),
\end{eqnarray*}
    où $g_i$ est la sous-matrice de $g$ faite des $i$-premières lignes et des $i$-premières colonnes de $g$, est représenté par une variété affine de type fini sur $k$ qu'on note aussi $Y_{\underline{d}}$.
    Soit
    $$f^{Y}_{\underline{d}}: Y_{\underline{d}}\times \mg_m^{r-1}\ra V_{\underline{d}}\times \mg_m^{r-1}$$
    le morphisme défini par $f^{Y}_{\underline{d}}(g,\alpha)=((a_i)_{1\leq i\leq r},\alpha)$ où $a_i=\det(g_i)$
\item Pour tout $i$ avec $2\leq i \leq r$, l'application $h'_{i}:\mathfrak{gl}_r(\0\otimes_kR)\times (R^*)^{r-1}\ra R$ définie par
    \begin{eqnarray*}
    h'_{i}&=&\frac{1}{2}\,{\rm res}\left(a_{i-1}^{-1}\left((g_{i,1},\dots,g_{i,i-1})a_{i-1}g_{i-1}^{-1}\left(\begin{matrix}0\\\vdots\\0\\\alpha_i\end{matrix}\right)+\right.\right.\\
    &&\left.\left.+(0,\dots,0,\alpha_i)a_{i-1}g_{i-1}^{-1}\left(\begin{matrix}g_{1,i}\\\vdots\\g_{i-2,i}\\g_{i-1,i}\end{matrix}\right)\right)\right),
    \end{eqnarray*}
    où $a_ig_i^{-1}$ est la matrice des cofacteurs de $g_i$ lesquels sont dans $\0_{R}:=\0\otimes_kR$ et où $\mathrm{res}(a_{i-1}^{-1} b)$ est le coefficient de $\varpi^{d_{i-1}-1}$ dans l'expression polynomiale en variable $\varpi$ du reste de la
    division euclidienne de $b$ par $a_{i-1}$ (cette division euclidienne a un sens puisque le coefficient dominant de $a_{i-1}$ est égal à $1$), induit un morphisme $h'_{i} : Y_{\underline{d}}\times \mg_m^{r-1}\ra\mg_a$.
\item Soit $h_{\underline{d}}=\sum_{i=2}^rh_{i}$. Il existe un morphisme naturel $\underline{\kappa}_{\ud}: Y_{\ud}\to \mg_m$ 
    tel que pour tous $t\in Q_{\underline{d}}(k)$ et $\underline{\alpha} \in \mg_m^{r-1}$, le triplet $(Y(t),h'_{\alpha},\underline{\kappa})$ est isomorphe à la fibre $(f^{Y}_{\underline{d}})^{-1}(t,\underline{\alpha})$ munie de la restriction de $h_{\underline{d}}$ et de $\underline{\kappa}_{\ud}$ à cette fibre.
\end{enumerate}
\end{lemma}
\begin{proof}
On adapte la démonstration de \cite[proposition 3.3.1]{N} et il ne reste que le morphisme $\underline{\kappa}_{\ud}$ à construire. 

Soit $R$ une $k$-algèbre. 
Soit $g\in Y_{\ud}(R)$. On considère la ``décomposition de Bruhat'' de $g$ dans $R[\varpi,(\prod_{i=1}^ra_i)^{-1}]$ : $$g=nw_0tn',$$
où
\begin{itemize}
\item  les $a_i$ sont des polynômes unitaires à coefficients dans $R$,
\item $n,n'\in N_r(R[\varpi,(\prod_{i=1}^ra_i)^{-1}])$
\item et $t=\D(a_1,a_2/a_1,\dots,a_r/a_{r-1})$.
\end{itemize}
Comme $g\in\G_r(R[\varpi,a_r^{-1}])$, d'après Ngo \cite{N}, $n$ (resp. $n'$) s'écrit sous la forme :
$$n=\prod_{i=2}^r(\mathrm{Id}_r+y_i)\quad (\text{resp. }n'=\prod_{i=2}^r(\mathrm{Id}_r+y'_i)),$$
où $y_i\in a_{i-1}^{-1} R[\varpi,a_r^{-1}]$ (resp. $y'_i\in a_{i-1}^{-1} R[\varpi,a_r^{-1}]$).

Lorsque $M$, $M'$ et $M''$ sont trois $R[\varpi,a_r^{-1}]$- sous-modules de \break $R[\varpi,(\prod_{i=1}^ra_i)^{-1}]^r$ tels que $M''\subset M\cap M'$ et lorsque les quotients $M/M''$ et $M'/M''$ sont libres de type fini en tant que $R$-modules, on note
$$(M|M')=(\bigwedge M/M'')^{\otimes -1}\otimes (\bigwedge M'/M''),$$
où $\bigwedge$ est la puissance extérieure de degré maximal.

On note $P=\prod_{i=1}^ra_i$. On a alors : $P.R[\varpi,a_r^{-1}]^r\subset w_0tR[\varpi,a_r^{-1}]^r \cap R[\varpi,a_r^{-1}]^r$. Comme $P$ est un polynôme unitaire, $w_0tR[\varpi,a_r^{-1}]^r/\break P.R[\varpi,a_r^{-1}]^r$ et $R[\varpi,a_r^{-1}]^r/P.R[\varpi,a_r^{-1}]^r$ sont libres de type fini, de sorte qu'on définit :
$$D_{w_0t}^R=(\bigwedge R[\varpi,a_r^{-1}]^r/P.R[\varpi,a_r^{-1}]^r)^{\otimes -1}\otimes_R(\bigwedge w_0tR[\varpi,a_r^{-1}]^r/P.R[\varpi,a_r^{-1}]^r)$$
(on obtient de la même manière la $R$-droite $D_{t_i}$ où $t_i=a_i/a_{i-1}$ en convenant que $a_0=1$). En utilisant la même démarche que dans \ref{chi5}, on obtient le morphisme $\delta_{w_0g}: D_{w_0t} \to D_{\det(w_0t)}^R=(R[\varpi,a_r^{-1}]^r|\det(w_0t)R[\varpi,a_r^{-1}]^r)\can R$ (le morphisme de multiplication
$$D_{\prod_{j=1}^{r-i}t_{r+1-j}}\otimes_R D_{t_i}\to D_{\prod_{j=1}^{r-i+1}t_{r+1-j}}$$ est défini comme dans la proposition \ref{ACK1}), de sorte qu'on obtient un isomorphisme $\Delta_{w_0t}^R\simeq R$.

On considère $M=\D(1,a_1,a_1.a_2,\dots,\prod_{i=1}^{r-1}a_i)R[\varpi,a_r^{-1}]^r$. On a $nM=M$ et $n'M=M$. De plus, $R[\varpi,a_r^{-1}]^r/M \simeq n.R[\varpi,a_r^{-1}]^r/n.M =\break n.R[\varpi,a_r^{-1}]^r/M$ est un $R$-module libre, de sorte qu'on définit
$$D_n^R:=(\bigwedge R[\varpi,a_r^{-1}]^r/M)^{\otimes -1}\otimes_R(\bigwedge n.R[\varpi,a_r^{-1}]^r/M)$$
et
$$D_{n'}^R:=(\bigwedge R[\varpi,a_r^{-1}]^r/M)^{\otimes -1}\otimes_R(\bigwedge n'.R[\varpi,a_r^{-1}]^r/M).$$
L'isomorphisme canonique $n.R[\varpi,a_r^{-1}]^r/n.M\to R[\varpi,a_r^{-1}]^r/M$ (resp. \break $n'.R[\varpi,a_r^{-1}]^r/n'.M\to R[\varpi,a_r^{-1}]^r/M$) nous donne un isomorphisme $D_n\can R$ (resp. $D_{n'}\can R$).

On considère
\begin{eqnarray*}
&D_n^R\otimes D_{w_0t}^R\otimes D_{n'}^R\simeq (R[\varpi,a_r^{-1}]^r|M)\otimes (n.M|n.R[\varpi,a_r^{-1}]^r)\otimes \\
&(n . R[\varpi,a_r^{-1}]^r|n.P.R[\varpi,a_r^{-1}]^r)\otimes (n.P.R[\varpi,a_r^{-1}]^r|n.w_0tR[\varpi,a_r^{-1}]^r)\otimes\\
&(nw_0t.R[\varpi,a_r^{-1}]^r|nw_0t.M)\otimes (nw_0t.M|nw_0tn'R[\varpi,a_r^{-1}]^r).
\end{eqnarray*}

Comme $n.P.R[\varpi,a_r^{-1}]^r\subset n.M\subset n.R[\varpi,a_r^{-1}]^r$, on a un isomorphisme canonique
\begin{multline*}(n.M|n.R[\varpi,a_r^{-1}]^r)\otimes(n . R[\varpi,a_r^{-1}]^r|n.P.R[\varpi,a_r^{-1}]^r)\\
\simeq (n.M|n.P.R[\varpi,a_r^{-1}]^r).
\end{multline*}

On a aussi $n.P.R[\varpi,a_r^{-1}]^r\subset nw_0t. M\subset nw_0t.R[\varpi,a_r^{-1}]^r$, de sorte qu'on obtient un isomorphisme canonique
\begin{multline*}(n.P.R[\varpi,a_r^{-1}]^r|n.w_0tR[\varpi,a_r^{-1}]^r)\otimes(nw_0t.R[\varpi,a_r^{-1}]^r|nw_0t.M)\\
\simeq
(n.P.R[\varpi,a_r^{-1}]^r|nw_0t.M).
\end{multline*}

De plus, comme $n.P.R[\varpi,a_r^{-1}]^r\subset nw_0t. M\subset nw_0t.n'R[\varpi,a_r^{-1}]^r$, on obtient :
\begin{multline*}
(n.P.R[\varpi,a_r^{-1}]^r|nw_0t.M)\otimes (nw_0t.M|nw_0tn'R[\varpi,a_r^{-1}]^r)\\\simeq (n.P.R[\varpi,a_r^{-1}]^r|nw_0tn'R[\varpi,a_r^{-1}]^r).
\end{multline*}
D'autre part, on a $n.P.R[\varpi,a_r^{-1}]^r\subset n.M=M\subset R[\varpi,a_r^{-1}]^r$, donc
\begin{multline*}
(R[\varpi,a_r^{-1}]^r|M)\otimes (n.M|n.P.R[\varpi,a_r^{-1}]^r)\\\simeq
(R[\varpi,a_r^{-1}]^r|n.P.R[\varpi,a_r^{-1}]^r)
\end{multline*}

Par conséquent, on obtient un isomorphisme $\delta_g : R\simeq D_n^R\otimes D_{w_0t}^R\otimes D_{n'}^R \to (R[\varpi,a_r^{-1}]^r|gR[\varpi,a_r^{-1}]^r)$. On obtient alors un isomorphisme $\delta_g : R \to \Delta_g^R$, où $\Delta_g^R:= (R[\varpi,a_r^{-1}]^r|gR[\varpi,a_r^{-1}]^r)^{\otimes -1}\otimes(R[\varpi,a_r^{-1}]|\det(g)R[\varpi,a_r^{-1}])$. On définit  $\underline{\kappa}_{\ud}(g)=\mathrm{triv}/\delta_g$, où $\mathrm{triv}$ est l'isomorphisme $R\to \Delta_g^R$ provenant du fait que $gR[\varpi,a_r^{-1}]^r=R[\varpi,a_r^{-1}]^r$.

Cette définition commute au changement de base, on obtient donc bien un morphisme $\underline{\kappa}_{\ud}: Y_{\ud}\to \mg_m$ ; en prenant les $k$-points celui-ci redonne la fonction $\underline{\kappa}$ de la  section 3.2.

\end{proof}
\begin{theorem}\label{ka7}
Soit $g\in Y_{\ud}$. On a alors $$\underline{\kappa}_{\ud}(w_0g)\underline{\kappa}_{\ud}(w_0{}^tg)=(-1)^{\sum_{i=1}^{r-1}d_i+\sum_{i=1}^{r-1}d_id_{i+1}}{\rm result}(a_{r-1},a_r).$$
\end{theorem}
\begin{proof}
Comme $\underline{\kappa}_{\ud}$ est un morphisme de $Y_{\ud}$ dans $\mg_m$, le théorème résulte du fait (qu'on va ensuite vérifier) que
$$\underline{\kappa}_{\ud}(w_0g)\underline{\kappa}_{\ud}(w_0{}^tg)=(-1)^{\sum_{i=1}^{r-1}d_i+\sum_{i=1}^{r-1}d_id_{i+1}}{\rm result}(a_{r-1},a_r)$$ pour $g$ dans un ouvert dense de $Y_{\ud}$, qui n'est autre que l'image réciproque de $U_{\ud}$ par le morphisme $f_{\ud}^Y$. On va supprimer l'indice $\ud$ pour alléger les notations.

Soit $i\in \{1,\dots,r-1\}$. On note $\lambda_{i,j}$ où $j\in \{1,\dots,d_i\}$ les racines simples de $a_i$. En utilisant la formule de produit, on a :
$$\underline{\kappa}(w_0g)=\prod_{i=1}^{r-1}\prod_{j=1}^{d_i}\underline{\kappa}_{\varpi-\lambda_{i,j}}(w_0g)\  ;\quad \underline{\kappa}(w_0{}^tg)=\prod_{i=1}^{r-1}\prod_{j=1}^{d_i}\underline{\kappa}_{\varpi-\lambda_{i,j}}(w_0{}^tg).$$

En chaque place $\varpi-\lambda_{i,j}$, on écrit $n=u_2u_3\dots u_n$ et $n'=u_2'u_3'\dots u'_n$ où $u_k,u'_k$ appartiennent au-sous groupe $U_k(F_{\varpi-\lambda_{i,j}}$ de $N_r(F_{\varpi-\lambda_{i,j}})$ formé des matrice triangulaires supérieures unipotentes dont les coefficients non-diagonaux sont tous nuls sauf caux de la $k$-ième colonne. Soient $u_k=\mathrm{Id}_r+y_k$ et $u'_k=\mathrm{Id}_r+y'_k$. D'après \cite[lemme 1.1.3]{N}, $y_k,y'_k\in (a_{k-1}^{-1}\0_{\varpi-\lambda_{i,j}})^{k-1}$, donc $u_k,u'_k\in U_k(\0_{\varpi-\lambda_{i,j}})$ pour $j\neq i$.
On note $n_{i+1}=u_2\dots u_{i+1}u_i^{-1}\dots u_2^{-1}$ et $n'_{i+1}=u'_2\dots u'_{i+1}{u'}_i^{-1}\dots {u'}_2^{-1}$. En reprenant la démonstration de \cite[proposition 2.4.1]{N}, on obtient que $n_{i+1}, n'_{i+1}\in U_{i+1}$ et que les coefficients $n_{k,i+1},n'_{k,i+1}$ de $n_{i+1}$ et $n'_{i+1}$ sont dans $\0_{\varpi-\lambda_{i,j}}$ pour $k=1,\dots, i-1$. Ainsi $n_{i+1}$ (resp. $n'_{i+1}$) s'écrit-il  $n_{i+1}=(\mathrm{Id}+ z_{i+1})(\mathrm{Id}+\underline{z}_{i+1})$ (resp. $n'_{i+1}=(\mathrm{Id}+ z'_{i+1})(\mathrm{Id}+\underline{z}'_{i+1})$, où $z_{i+1}={}^t(0,\dots,0,n_{i,i+1})\in \0_{\varpi-\lambda_{i,j}}^{i}$ et $\underline{z}_{i+1}={}^t(n_{1,i+1},\dots,n_{i-1,i+1},0)\in \0_{\varpi-\lambda_{i,j}}^{i}$ (resp. $z'_{i+1}={}^t(0,\dots,0,n'_{i,i+1})\in \0_{\varpi-\lambda_{i,j}}^{i}$ et $\underline{z}'_{i+1}={}^t(n'_{1,i+1},\dots,n'_{i-1,i+1},0)\in \0_{\varpi-\lambda_{i,j}}^{i}$).
Soient $$\widehat{n}_{i+1}=(\mathrm{Id}+\underline{z}_{i+1})u_2\dots u_iu_{i+1}\dots u_r\in N_r(\0_{\varpi-\lambda_{i,j}})$$
et $$\widehat{n}'_{i+1}=(\mathrm{Id}+\underline{z}'_{i+1})u'_2\dots u'_iu'_{i+1}\dots u'_r\in N_r(\0_{\varpi-\lambda_{i,j}}).$$
On a alors $w_0g=(w_0{}^t\widehat{n}_{i+1}w_0)w_0{}^t(\mathrm{Id}+ z_{i+1})t(\mathrm{Id}+ z'_{i+1})\widehat{n}'_{i+1}$. Comme $w_0{}^t\widehat{n}_{i+1}w_0,\break\widehat{n}'_{i+1}\in N_r(\0_{\varpi-\lambda_{i,j}})$, on a $\underline{\kappa}(w_0{}^t\widehat{n}_{i+1}w_0)=\underline{\kappa}(\widehat{n}'_{i+1})=1$. Par ailleurs, en utilisant $\underline{\kappa}(g_1g_2)=\underline{\kappa}(g_1)\underline{\kappa}(g_2)\chi(g_1,g_2)$ et $\chi(n,g)=\chi(g,n)=1$, on a $$\underline{\kappa}(w_0g)=\underline{\kappa}(w_0{}^t(\mathrm{Id}+ z_{i+1})t(\mathrm{Id}+ z'_{i+1})).$$
Pareillement, on a aussi :
$$\underline{\kappa}(w_0{}^tg)=\underline{\kappa}(w_0{}^t({}^t(\mathrm{Id}+ z_{i+1})t(\mathrm{Id}+ z'_{i+1}))).$$
$t$ s'écrit $t=t_1t_2$ où $t_1=\D(1,\dots,a_i/a_{i-1},a_{i+1}/a_i,1,\dots,1)$ et $t_2=\D(a_1,\dots,a_{i-1}/a_{i-2},1,1,a_{i+1}/a_{i+2},\dots, a_r/a_{r-1})\in T_r(\0_{\varpi-\lambda_{i,j}})$. On a
$$(\mathrm{Id}+ z_{i+1})t(\mathrm{Id}+ z'_{i+1})=(\mathrm{Id}+ z_{i+1})t_1(\mathrm{Id}+ z''_{i+1})t_2$$
où $z''_{i+1}$ est une matrice triangulaire supérieure unipotente dont les coefficients non-diagonaux sont tous nuls sauf celui de la position $(i,i+1)$.

Comme $\underline{\kappa}(gt_2)=\underline{\kappa}(g)$ pour $g\in \G_r(\0_{\varpi-\lambda_{i,j}})$ et $t_2\in T_r(\0_{\varpi}-\lambda_{i,j})$ et que $\underline{\kappa}(\begin{smallmatrix}g_1&0\\0&g_2\end{smallmatrix})=\underline{\kappa}(\begin{smallmatrix}0&g_1\\g_2&0\end{smallmatrix})=\underline{\kappa}(g_1)\underline{\kappa}(g_2)$, le calcul $\underline{\kappa}$ se réduit au cas où $r=2$ avec \begin{eqnarray*}g&=&\begin{pmatrix}1&0\\x&1\end{pmatrix}\D(a_i/a_{i-1},a_{i+1}/a_i)\begin{pmatrix}1&x'\\0&1\end{pmatrix}\\
&=&\begin{pmatrix}\frac{a_i}{a_{i-1}}&x'\frac{a_i}{a_{i-1}}\\x\frac{a_i}{a_{i-1}}&xx'\frac{a_i}{a_{i-1}}+\frac{a_{i+1}}{a_{i}}\end{pmatrix}\in \G_2(\0_{\varpi-\lambda_{i,j}}).\end{eqnarray*}

En utilisant la formule de Kubota (cf. la proposition~\ref{kubo}), on a :
$$\underline{\kappa}_{\varpi-\lambda_{i,j}}(w_0g)=\left\{a_i/a_{i-1},\frac{x'a_i/a_{i-1}}{\det(w_0g)}\right\}_{\varpi-\lambda_{i,j}};$$ $$\underline{\kappa}_{\varpi-\lambda_{i,j}}(w_0{}^tg)=\left\{a_i/a_{i-1},\frac{xa_i/a_{i-1}}{\det(w_0{}^tg)}\right\}_{\varpi-\lambda_{i,j}}.$$
Comme $g\in \G_2(\0_{\varpi-\lambda_{i,j}})$, on a donc $v_{\varpi-\lambda_{i,j}}(x)=v_{\varpi-\lambda_{i,j}}(x')=-1$. Donc $\underline{\kappa}_{\varpi-\lambda_{i,j}}(w_0g)=\frac{\det(w_0g)}{x'a_i/a_{i-1}}(\lambda_{i,j})$ et $\underline{\kappa}_{\varpi-\lambda_{i,j}}(w_0{}^tg)=\frac{\det(w_0g)}{xa_i/a_{i-1}}(\lambda_{i,j})$.

Par ailleurs on a : $\det(w_0g)=x'\frac{a_i}{a_{i-1}}x\frac{a_i}{a_{i-1}}-\frac{a_i}{a_{i-1}}\left(xx'\frac{a_i}{a_{i-1}}+\frac{a_{i+1}}{a_{i}}\right)$. Le deuxième terme de la somme est de valuation supérieure ou égal à $1$, donc $\det(w_0g)(\lambda_{i,j})=\left(x'\frac{a_i}{a_{i-1}}x\frac{a_i}{a_{i-1}}\right)(\lambda_{i,j})$. Par conséquent, on a :
$$\underline{\kappa}_{\varpi-\lambda_{i,j}}(w_0g)\underline{\kappa}_{\varpi-\lambda_{i,j}}(w_0{}^tg)=\det(w_0g)(\lambda_{i,j})=-\frac{a_{i+1}}{a_{i-1}}(\lambda_{i,j}).$$

On en déduit que :
\begin{eqnarray*}
\underline{\kappa}(w_0g)\underline{\kappa}(w_0{}^tg)&=&\prod_{i=1}^{r-1}\prod_{j=1}^{d_i}(-1)\frac{a_{i+1}}{a_{i-1}}(\lambda_{i,j})\\
&=&\prod_{i=1}^{r-1}(-1)^{d_i}\frac{{\rm result}(a_{i+1},a_i)}{{\rm result}(a_{i-1},a_i)}\\
&=&(-1)^{\sum_{i=1}^{r-1}d_i+\sum_{i=1}^{r-1}d_id_{i+1}}{\rm result}(a_{r-1},a_r).
\end{eqnarray*}
La signe dans la dernière formule vient du fait que
$${\rm result}(P,Q)=(-1)^{\deg(P)\deg(Q)}{\rm result}(Q,P).$$
\end{proof}
\subsection{Le cas $\emph{d}=(1,2,\dots,r)$}
\setcounter{theorem}{0}
\setcounter{subsubsection}{1}

\begin{proposition}\label{ngo2}
Pour $\underline{d}=(1,2,\dots,r)$ le quadruple $(Y_{\underline{d}},f^{Y}_{\underline{d}},h'_{\underline{d}},\underline{\kappa}_{\ud})$ est isomorphe au quadruple $(\mathfrak{gl}_r^\circ,f^Y,h',\underline{\kappa})$ où :
\begin{itemize}
\item $\mathfrak{gl}_r=\{y\in \mathfrak{gl}_r|{\rm pgcd}(\prod_{i=1}^{r-1}a_i(y),a_r(y))=1\}$,
\item le morphisme $f^Y:\mathfrak{gl}_r^{\circ}\times \mg_m^{r-1}\ra\prod_{i=1}^rQ_i\times\mg_m^{r-1}$ est défini par
$$f^Y(y,\underline{\alpha})=(a_1(y),\dots,a_r(y),\underline{\alpha}),$$
\item le morphisme $h':\mathfrak{gl}_r^\circ\times \mg_m^{r-1}\ra \mg_a$ est défini par
$$h'(y,\underline{\alpha})=\sum_{i=2}^r\frac{1}{2}\alpha_i(y_{i-1,i}+y_{i,i-1})$$
\item le morphisme $\underline{\kappa}:\mathfrak{gl}_r^\circ\to \mg_m$
est défini par $\underline{\kappa}(y)=\underline{\kappa}_{\ud}(w_0(y+\varpi{\rm Id}_r))$.
\end{itemize}
\end{proposition}
\begin{proof}[Démonstration] On adapte la démonstration de \cite[proposition 4.2.1]{N}. 
\end{proof}Il reste à calculer la fonction $\underline{\kappa}$ sur $\mathfrak{gl}_r^\circ$. Soit $y=(y_{i,j})\in \mathfrak{gl}_r^\circ$. Pour cela on va utiliser le lemme suivant :
\begin{lemma}\label{ka2}
Soi $V'$ un espace vectoriel (en général de dimension infinie). Soient $A$ et $B$ deux sous-espaces vectoriels et $V=A\oplus B$. Soit $g\in \mathrm{Aut}(V')$ tel que $A$ et $gA$ soient commensurables ainsi que $B$ et $gB$. Les sous-espaces vectoriels $V$ et $gV$ sont alors commensurables et on a un isomorphisme canonique :
$$(V|gV)\can (A|gA)\otimes (B|gB).$$ Cet isomorphisme est compatible avec la convention de signe, c.à.d le diagramme
$$\xymatrix{&(V|gV)\ar[dr]\ar[dl]\\
(A|gA)\otimes (B|gB)\ar[rr]^{\mathrm{Sym}^{\bullet}}&&(B|gB)\otimes(A|gA)}$$ est commutatif.
\end{lemma}
\begin{proof}
On note $A'=A\cap gA$ et $B'=B\cap gB$.
\begin{eqnarray*}
(V|gV)&=&(A\oplus B|gA\oplus gB)\\
&\can&\left(\bigwedge(A\oplus B/A'\oplus B')\right)^*\otimes \bigwedge (gA\oplus gB/A'\oplus B')\\
&\can&\left(\bigwedge(A/A'\oplus B/ B')\right)^*\otimes \bigwedge (gA/A'\oplus gB/B')\\
&\can&\left(\bigwedge(A/A')\right)^*\otimes\left(\bigwedge(B/B')\right)^*\otimes \bigwedge(gA/A')\otimes\bigwedge(gB/B')\\
&\can&\left(\bigwedge(A/A')\right)^*\otimes \bigwedge(gA/A')\otimes\left(\bigwedge(B/B')\right)^*\otimes\bigwedge(gB/B')\\
&\can&(A|gA)\otimes(B|gB).
\end{eqnarray*}
\end{proof}

Soit $g$ une matrice de taille $r$. Soient $I,J$ deux sous-ensembles de $\{1,2,\dots,\break r\}$. On note $g_{I,J}$ la matrice extraite de $g$ dont les indices de ligne sont dans $I$ et les indices de colonnes sont dans $J$, en ordonnant par ordre croissant les éléments de $I$ et de $J$.
\begin{proposition}(cf. \cite[p. 394-395]{A2})\label{alain1}
Soit $g\in \G_r$. Supposons que $g$ s'écrit comme un produit $AB^{-1}$, où $A$ est une matrice triangulaire inférieure et $B$ est une matrice triangulaire supérieure dont tous les coefficiens diagonaux sont égaux à $1$. On écrit les matrices $A$ et $B$ sous la forme
$$A=\begin{pmatrix}
a_{1,1}&0&\dots&0\\
a_{2,1}&\ddots&\ddots&\vdots\\
\vdots&\ddots&\ddots&0\\
a_{r,1}&\dots&a_{r,r-1}&a_{r,r}
\end{pmatrix}\ ;\quad B=\begin{pmatrix}
1&b_{1,2}&\dots&b_{1,r}\\
0&\ddots&\ddots&\vdots\\
\vdots&\ddots&\ddots&b_{r-1,r}\\
0&\dots&0&1
\end{pmatrix}.$$
On note $\tilde{a}_{i,j}$ (resp. $\tilde{b}_{i,j}$) les coefficients de la matrice $A^{-1}$ (resp. de la matrice $B^{-1}$). On a alors :
$$a_{i,j}=\frac{\det(g_{[1,j-1]\cup \{i\},[1,j]})}{\det(g_{[1,j-1],[1,j-1]})},\quad
b_{i,j}=(-1)^{j-i}\frac{\det(g_{[1,j-1],[1,j]-\{i\}})}{\det(g_{[1,j-1],[1,j-1]})},$$
$$\tilde{a}_{i,j}=(-1)^{i-j}\frac{\det(g_{[1,i]-\{j\},[1,i-1]})}{\det(g_{[1,i],[1,i]})}
\text{ et }\quad \tilde{b}_{i,j}=\frac{\det(g_{[1,i],[1,i-1]\cup\{j\}})}{\det(g_{[1,i],[1,i]})}.$$
\end{proposition}
En utilisant la proposition~\ref{alain1} ci-dessus pour $g=y+\varpi{\rm Id}_r$ et en factorisant la matrice triangulaire inférieure sous la forme ${}^tN_rT_r$, on obtient le corollaire suivant
\begin{corollary}\label{ka5}
La décomposition de Bruhat de $w_0g=w_0(y+\varpi{\rm Id}_r)$ est $nw_0tn'$ avec
$n=(n_{i,j})$ où $$n_{i,j}=\frac{\det(g_{[1,r-j]\cup \{r+1-i\},[1,r+1-j]})}{\det(g_{[1,r+1-j],[1,r+1-j]})},$$ $t=\D(a_1(y),a_2(y)/a_1(y),\dots,a_r(y)/a_{r-1}(y))$ et
$n'=(n'_{i,j})$ où $$n'_{i,j}=(-1)^{j-i}\frac{\det(g_{[1,j-1],[1,j]-\{i\}})}{\det(g_{[1,j-1],[1,j-1]})}.
$$
De plus, si on note $n^{-1}=(\tilde{n}_{i,j})$, on a alors
$$\tilde{n}_{i,j}=(-1)^{j-i}\frac{\det(g_{[1,r+1-i]-\{r+1-j\},[1,r-i]})}{\det(g_{[1,r-i],[1,r-i]})}.$$
\end{corollary}
\begin{theorem}\label{pop} $\underline{\kappa}(y)$ est en fait un polyn\^ome en les coefficients de la matrice $y\in \mathfrak{gl}_r^\circ$. De plus on a : $$\underline{\kappa}(y)\underline{\kappa}({}^ty)=(-1)^{\sum_{i=1}^{r-1}[i+i(i+1)]}({\rm result}(a_{r-1}(y),a_r(y)).$$ \end{theorem}
\begin{proof}

Comme $\underline{\kappa}$ est un morphisme de $\mathfrak{gl}_r^\circ$ dans $\mg_m$, on peut se contenter de calculer $\underline{\kappa}(y)$ sur l'ouvert dense où $a_r(y)$ est à racines simples. On note $(\lambda_i)_{ i\in \{1,\dots,r\}}$ ($\lambda_i\in\overline{k}$) les racines de $a_r(y)$. En utilisant $S=\{\lambda_i\}_i$, la fonction $\underline{\kappa}_{\ud}$ provient de l'extension $\tG_r^S(F)$ de la partie ci-dessus. D'après le corollaire~\ref{ka5}, la décomposition de Bruhat de $w_0(y+\varpi{\rm Id}_r)$ est $nw_0tn'$ avec $t=\D(a_1(y),a_2(y)/a_1(y),\dots,a_r(y)/a_{r-1}(y))$.  La fonction $\underline{\kappa}_{\ud}$ provient alors du diagramme
$$\xymatrix{k\ar[r]^-{can} &\Delta_{w_0(y+\varpi{\rm Id}_r)}^S\ar[d]\\
&\Delta_{w_0t}^S\ar[ul],}$$
où l'isomorphisme vertical est celui figurant dans la définition de la section ensembliste $s_{\rm geo}:\G_r(F)\to\tG_r^S(F)$.

En utilisant le lemme~\ref{ka2} dans le cas particulier où $V=V'=\bigoplus_{i=1}^r\overline{k}((\varpi-\lambda_i))\oplus \overline{k}((\varpi^{-1}))$, $A=\overline{k}[\varpi,\Delta_r(y)^{-1}]$, $B=\bigoplus_{i=1}^r\overline{k}[[\varpi-\lambda_i]]\oplus \varpi^{-1}\overline{k}[[\varpi^{-1}]]$, on a $(V|\det(w_0(y+\varpi{\rm Id}_r))V)\can (A|\det(w_0(y+\varpi{\rm Id}_r))A)\otimes(B|\det(w_0(y+\varpi{\rm Id}_r))B)$ et aussi $(V^r|w_0(y+\varpi{\rm Id}_r)V^r)\can(A^r|w_0(y+\varpi{\rm Id}_r)A^r)\otimes(B^r|w_0(y+\varpi{\rm Id}_r)B^r)$. Par conséquent, on a $\Delta_{w_0(y+\varpi{\rm Id}_r)}^V\can \Delta_{w_0(y+\varpi{\rm Id}_r)}^A\otimes \Delta_{w_0(y+\varpi{\rm Id}_r)}^B$, où $\Delta_y^\bullet=(\bullet|\det(y)\bullet)\otimes(\bullet^r|y\bullet^r)^{\otimes -1}$ avec $\bullet =A,B \text{ ou } V$ (voir aussi la remarque \ref{rmk3} sur la trivialité de l'extension globale $\tG_r^{\emptyset}(F)$).

Pareillement, on a $\Delta_{w_0t}^V\can \Delta_{w_0t}^A\otimes \Delta_{w_0t}^B$.

Comme $w_0(y+\varpi{\rm Id}_r), w_0t\in\G_r(V)$, on a $\Delta_{w_0(y+\varpi{\rm Id}_r)}^V\can \overline{k}$ et $\Delta_{w_0t}^V\can \overline{k}$, de sorte qu'on a un isomorphisme canonique :
$\Delta_{w_0(y+\varpi{\rm Id}_r)}^A\can {\Delta_{w_0(y+\varpi{\rm Id}_r)}^B}^{\otimes-1}$ et $\Delta_{w_0t}^A\can {\Delta_{w_0t}^B}^{\otimes-1}$. On a alors le diagramme suivant :
$$\xymatrix{\overline{k}\ar[r]^-{can}&\Delta_{w_0(y+\varpi{\rm Id}_r)}^S=\Delta_{w_0(y+\varpi{\rm Id}_r)}^A\ar@{-}[r]^-{\sim}\ar[d]&{\Delta_{w_0(y+\varpi{\rm Id}_r)}^B}^{\otimes-1}\ar[d]\\
& \Delta_{w_0t}^S=\Delta_{w_0t}^A\ar[ul]&{\Delta_{w_0t}^B}^{\otimes-1}\ar@{-}[l]^-{\sim}}$$

\begin{enumerate}
\item \`A la place $v_i=\varpi-\lambda_i$, on note $\overline{\0}_{v_i}=\overline{k}[[\varpi-\lambda_i]]$. D'après le lemme~\ref{ka5} on a : $n\overline{\0}_{v_i}^r=n'\overline{\0}_{v_i}^r=\overline{\0}_{v_i}^r$. Comme $t\overline{\0}_{v_i}^r=\D(a_1(y),\break a_2(y)/a_1(y),\dots,a_r(y)/a_{r-1}(y))\overline{\0}_{v_i}^r=\bigoplus_{j=1}^{r-1}\overline{\0}_{v_i}e_j\oplus v_i\overline{\0}_{v_i}e_r$, on a $${\mathfrak{d}_{w_0(y+\varpi{\rm Id}_r)}}_{|v_i}={e_1^*}_{|v_i}.$$

    Comme $\det(w_0(y+\varpi{\rm Id}_r))=(-1)^{\frac{r(r-1)}{2}}a_r(y)$, on a $${\mathfrak{d}_{\det(w_0(y+\varpi{\rm Id}_r))}}_{|v_i}=1^*_{|v_i}.$$

    On a ${\overline{\delta}_{w_0(y+\varpi{\rm Id}_r)}}_{|v_i}=(-1)^{r-1}a_{r-1}(\lambda_i)$. Donc
    on a $${\delta_{w_0(y+\varpi{\rm Id}_r)}}_{|v_i}=(-1)^{r-1}a_{r-1}(\lambda_i)1^*_{|v_i}\otimes{e_1}_{|v_i}.$$
\item \`A la place $\infty$, on note $\overline{\mathfrak{m}}_\infty=\varpi^{-1}\overline{k}[[\varpi^{-1}]]$ et $\overline{\0}_\infty=\overline{k}[[\varpi^{-1}]]$. La droite qu'on va utiliser dans ce cas est $D_g^\infty=(\overline{\mathfrak{m}}_\infty^r|g\overline{\mathfrak{m}}_\infty^r)$ et le choix de l'élément ${\delta_{w_0(y+\varpi{\rm Id}_r)}}_{|\infty}$ est le même que celui fait dans la construction~\ref{chi5}. D'après le lemme~\ref{ka5} on a $v_{\infty}(n_{i,j})\geq 1$ et $v_{\infty}(n'_{i,j})\geq 1$, de sorte qu'on a $n\overline{\mathfrak{m}}^r_\infty=n'\overline{\mathfrak{m}}^r_{\infty}=\overline{\mathfrak{m}}^r_{\infty}$ et $t\overline{\mathfrak{m}}^r_{\infty}=\overline{\0}_\infty^r$. Donc ${\mathfrak{d}_{w_0(y+\varpi{\rm Id}_r)}}_{|\infty}=\bigwedge_{i=1}^rn{e_i}_{|\infty}$. En reprenant la démonstration du lemme \ref{n3} dans le cas où $M=\overline{\0}_\infty^r$ et $M'=\overline{\mathfrak{m}}_\infty^r$, on voit que l'isomorphisme $(\overline{\0}_\infty^r|\overline{\mathfrak{m}}_\infty^r)\buildrel {\times n}\over{\ra}(n\overline{\0}_\infty^r|n\overline{\mathfrak{m}}_\infty^r)$ est l'identité, de sorte qu'on obtient $${\mathfrak{d}_{w_0(y+\varpi{\rm Id}_r)}}_{|\infty}=\bigwedge_{i=1}^r{e_i}_{|\infty}.$$De plus ${\mathfrak{d}_{\det(w_0(y+\varpi{\rm Id}_r))}}_{|\infty}=(\bigwedge_{i=r-1}^{0}\varpi^i)_{|\infty}$. Par ailleurs ${\overline{\delta}_{w_0(y+\varpi{\rm Id}_r)}}_{|\infty}=(-1)^{\frac{(r-1)r}{2}}$. On a alors
    $${\delta_{w_0(y+\varpi{\rm Id}_r)}}_{|\infty}=(-1)^{\frac{(r-1)r}{2}}(\bigwedge_{i=r-1}^{0}\varpi^i)_{|\infty}\otimes(\bigwedge_{i=1}^re_i)^*_{|\infty}.$$
\end{enumerate}
Ensuite on va expliciter l'isomorphisme ${\Delta_{w_0(y+\varpi{\rm Id}_r)}^B}^{\otimes -1}\simeq \overline{k}$ du diagramme ci-dessus.
On note $B'=\bigoplus_{j=1}^rg\overline{\0}_{v_i}^r\oplus\overline{\mathfrak{m}}_\infty^r$. Il est clair que $B'=B^r\cap gB^r$. Comme $g\in \G_r(A)\subset\G_r(V)$, on a successivement les identifications suivantes :
\begin{eqnarray*}
V^r/B'&=&gV^r/B'\\
 (A^r\oplus B^r)/B'&\can&(gA^r\oplus gB^r)/B'\\
 A^r\oplus (B^r/B')&\can& gA^r\oplus (gB^r/B')\\
(B^r/B')&\can& (gB^r/B')\quad\quad (A^r=gA^r)\\
 (\bigoplus_{i=1}^r\0_{v_i}^r\oplus\mathfrak{m}_\infty^r)/(\bigoplus_{i=1}^rg\0_{v_i}^r\oplus\mathfrak{m}_\infty^r)
&\can&(\bigoplus_{i=1}^rg\0_{v_i}^r\oplus g\mathfrak{m}_\infty^r)/(\bigoplus_{i=1}^rg\0_{v_i}^r\oplus\mathfrak{m}_\infty^r)\\
 \bigoplus_{i=1}^r(\0_{v_i}^r/g\0_{v_i}^r)&\can& g\mathfrak{m}_\infty^r/\mathfrak{m}_\infty^r\can\0_\infty^r/\mathfrak{m}_\infty^r\  (**).
\end{eqnarray*}

$(e_i)_{|\infty}\in V^r$ s'écrit $(\bigoplus_{j=1}^r(e_i)_{|v_j}\oplus(e_i)_{|\infty})\oplus\bigoplus_{j=1}^r(-e_i)_{|v_j}\in A^r\oplus B^r$ 
Par conséquent, l'image de $(e_i)_{|\infty}\in \overline{\0}_\infty^r/\overline{\mathfrak{m}}_\infty^r$ par l'isomorphisme canonique (**) ci-dessus est $\bigoplus_{j=1}^r(-\overline{e_i})_{|v_j}\in \bigoplus_{j=1}^r(\overline{\0}_{v_j}^r/g\overline{\0}_{v_j}^r)$, où $(-\overline{e_i})_{|v_j}$ est l'image de $(-{e_i})_{|v_j}$ par la projection canonique $\overline{\0}_{v_j}^r\rightarrow\overline{\0}_{v_j}^r/g\overline{\0}_{v_j}^r$.

Soit $(1,\tilde{n}_{1,2},\dots,\tilde{n}_{1,r})$ la première ligne de la matrice $n^{-1}$. On a   $n^{-1}(-{e_i})_{|v_j}=-{\tilde{n}_{1,i}(\lambda_j)e_1}_{|v_j}\  ({\rm mod}\  w_0tn'\overline{\0}_{v_j})$ (en convenant que $\tilde{n}_{1,1}=1$), de sorte qu'on obtient $(-\overline{e_i})_{|v_j}=-\tilde{n}_{1,i}(\lambda_j)(e_1)_{|v_j}$. Par conséquent l'isomorphisme canonique $D_{g,\infty}\can \bigotimes_{i=1}^rD_{g,v_i}$ (qui vient de l'isomorphisme (**)) est $(\bigwedge_{i=1}^re_i)_{|\infty}\mapsto\det(\m)\bigotimes_{i=1}^r{e_1}_{|v_i}$, où
$$\m=\begin{pmatrix}-\tilde{n}_{1,1}(\lambda_1)&-\tilde{n}_{1,1}(\lambda_2)&\dots&-\tilde{n}_{1,1}(\lambda_r)\\
-\tilde{n}_{1,2}(\lambda_1)&-\tilde{n}_{1,2}(\lambda_2)&\dots&-\tilde{n}_{1,2}(\lambda_r)\\
\vdots&\vdots&\ddots&\vdots\\
-\tilde{n}_{1,r}(\lambda_1)&-\tilde{n}_{1,r}(\lambda_2)&\dots&-\tilde{n}_{1,r}(\lambda_r)
\end{pmatrix}.$$

De même, l'isomorphisme canonique $D_{\det(g),\infty}\can\bigotimes_{i=1}^rD_{\det(g),v_i}$ est
$(\bigwedge_{i=r-1}^{0}\varpi^i)_{|\infty}\mapsto\det(\m')\bigotimes_{i=1}^r 1_{|v_i}$ où
$$\m'=\begin{pmatrix}-\lambda_1^{r-1}&-\lambda_1^{r-2}&\dots&-1\\
-\lambda_2^{r-1}&-\lambda_2^{r-2}&\dots&-1\\
\vdots&\vdots&\ddots&\vdots\\
-\lambda_r^{r-1}&-\lambda_r^{r-2}&\dots&-1\\
\end{pmatrix}.$$

Par conséquent, on a
$$\underline{\kappa}_{\ud}(w_0(y+\varpi{\rm Id}_r))=(-1)^{r(r-1)+\frac{r(r-1)}{2}}{\rm result}(a_{r-1}(y),a_r(y))\det(\m)/\det(\m').$$

Par ailleurs, en utilisant le corollaire \ref{ka5}, on a
$$\det(\m)=(-1)^{r+\sum_{i=1}^r (i-1)}\left(\begin{smallmatrix}\frac{\det(g^{r,r})}{a_{r-1}(y)}(\lambda_1)&\frac{\det(g^{r,r})}{a_{r-1}(y)}(\lambda_2)&\dots&\frac{\det(g^{r,r})}{a_{r-1}(y)}(\lambda_r)\\
\frac{\det(g^{r-1,r})}{a_{r-1}(y)}(\lambda_1)&\frac{\det(g^{r-1,r})}{a_{r-1}(y)}(\lambda_2)&\dots&\frac{\det(g^{r-1,r})}{a_{r-1}(y)}(\lambda_r)\\
\vdots&\vdots&\ddots&\vdots\\
\frac{\det(g^{1,r})}{a_{r-1}(y)}(\lambda_1)&\frac{\det(g^{1,r})}{a_{r-1}(y)}(\lambda_2)&\dots&\frac{\det(g^{1,r})}{a_{r-1}(y)}(\lambda_r)
\end{smallmatrix}\right),$$
où $g^{i,r}$ est la sous-matrice de $g$ obtenue en supprimant sa $i$-ième ligne et sa $r$-ième colonne de $g$ (en effet, par définition on a $g^{i,r}=g_{[1,r]-\{i\},[1,r-1]}$).

Il est clair que $\det(g^{i,r})$ est un polynôme en la variable $\varpi$  et en les coefficients de $y$, de degré partiel $\leq r-1$ en la variable $\varpi$. On note
$\det(g^{i,r})=g^{i,r}[0]+g^{i,r}[1]\varpi+\dots+g^{i,r}[r-1]\varpi^{r-1}$, de sorte que la formule de $\det(\m)$ se reécrit  :
\begin{multline*}
\det(\m)=(-1)^{r+\frac{r(r-1)}{2}}\frac{1}{{\rm result}(a_{r-1}(y),a_r(y))}\times\\
\times \det\left(\left[\begin{smallmatrix}g^{r,r}[0]&g^{r,r}[1]&\dots&g^{r,r}[r-1]\\
g^{r-1,r}[0]&g^{r-1,r}[1]&\dots&g^{r-1,r}[r-1]\\
\vdots&\vdots&\ddots&\vdots\\
g^{1,r}[0]&g^{1,r}[1]&\dots&g^{1,r}[r-1]\end{smallmatrix}\right]\right)\det\left(\left[\begin{smallmatrix}1&1&\dots&1\\
\lambda_1&\lambda_2&\dots&\lambda_r\\
\vdots&\vdots&\ddots&\vdots\\
\lambda_1^{r-1}&\lambda_2^{r-1}&\dots&\lambda_r^{r-1}\\
\end{smallmatrix}\right]\right)
\end{multline*}
(comme on le voit en effectuant le produit de ces deux dernières matrices).

Par conséquent, \begin{eqnarray*}
\underline{\kappa}(y)&=&\underline{\kappa}_{\ud}(w_0(y+\varpi{\rm Id}_r))\\
&=&\det\left(\left[\begin{smallmatrix}g^{r,r}[0]&g^{r,r}[1]&\dots&g^{r,r}[r-1]\\
g^{r-1,r}[0]&g^{r-1,r}[1]&\dots&g^{r-1,r}[r-1]\\
\vdots&\vdots&\ddots&\vdots\\
g^{1,r}[0]&g^{1,r}[1]&\dots&g^{1,r}[r-1]\end{smallmatrix}\right]\right)
 \end{eqnarray*} est un polynôme en les coefficients de la matrice $y$.

La deuxième assertion du théorème est la conséquence du théorème \ref{ka7} dan le cas particulier où $\ud=(1,2,\dots,r)$ et $g=y+\varpi{\rm Id}_r$.
\end{proof}
D'après le théorème~\ref{pop} la fonction $\underline{\kappa}$ est un produit des facteurs irréductibles de ${\rm result}(a_{r-1},a_r)$, de sorte qu'on peut prolonger la fonction $\underline{\kappa}$ au-dessus de l'ouvert $V'_{\ud}=\{(a_1,\dots,a_r)\in Q_{\ud}|{\rm pgcd}(a_{r-1},a_r)=1\}$. De cette manière la famille $(Y_{\ud},f^Y_{\ud},h'_{\ud},\underline{\kappa}_{\ud})$ au-dessus de $V_{\ud}\times\mg_m^{r-1}$ se prolonge à $V'_{\ud}\times\mg_m^{r-1}$ tout entier. On note $\check{\mathfrak{gl}}_r$ l'ouvert de $\mathfrak{gl}_r$ au-dessus de $V'_{\ud}\times\mg_m^{r-1}$.
\begin{theorem}\label{py}
Pour $\ud=(1,2,\dots,r)$ le complexe de faisceaux\break
$\mathrm{R}f^Y_{\ud,!}({h'}_{\ud}^*\mathcal{L}_{\psi}\otimes\underline{\kappa}_{\ud}^*\mathcal{L}_{\zeta})[r^2+r-1]$ est un faisceau pervers sur $V'_{\ud}$, prolonge\-ment interm\'ediaire de sa restriction \`a l'ouvert $U_{\ud}$.
\end{theorem}

On définit les variétés $R_i$ en posant $R_r=\mathrm{Spec}(k)$ et $R_{i-1}=R_i\times Q_i \times \mg_m$. Soit $f^Y_i:\mg_m\times \mathfrak{gl}_i\times R_i\ra \mathfrak{gl}_{i-1}\times R_{i-1}$ le morphisme défini par $f^Y_i(\alpha_i,y_i,r_i)=(s_{i-1}(y_i),r_{i-1})$, où $r_{i-1}=(r_i,a_i(y_i),\alpha_i)$ (sauf $f^Y_r:\mg_m\times\check{\mathfrak{gl}}_r\times R_r\ra \mathfrak{gl}_{r-1}\times R_{r-1}$). Soit $h'_i :\mg_m\times \mathfrak{gl}_i \times R_i\ra \mg_a$ le morphisme défini par $h'_i(\alpha_i,y_i,r_i)=\frac{1}{2}\alpha_i(y_{i-1,i}+y_{i,i-1})$. Soit $\mathrm{pr}_i:\mg_m\times S_i\times R_i\ra S_i\times R_i$ la projection évidente. On définit les complexes $\mathcal{J}_i$ sur $\mathfrak{gl}_i\times R_i$ en posant $\mathcal{J}_r=\underline{\kappa}^*\mathcal{L}_\zeta[r^2]$ et $\mathcal{J}_{i-1}=Rf^Y_{i,!}(\mathcal{J}_i\otimes {h'}^*_{i}\lsi)$ (voir \cite{N}). On a un isomorphisme $\mathfrak{gl}_1\times R_1 \simeq Q_{\ud}\times\mg_m^{r-1}$ via lequel $\mathcal{J}_1\simeq \mathrm{R}f^Y_{\ud,!}({h'}_{\ud}^*\mathcal{L}_{\psi}\otimes \underline{\kappa}^*_{\ud}\mathcal{L}_\zeta)[r^2+r-1]$.

On note $G_i$ l'extension de $\G_i$ obtenue en extrayant une racine carrée de $\det(g)\,\forall g\in\G_i$ (i.e $G_i=\{(g,\delta)|g\in \G_i, \det(g)=\delta^2\}$. Ce groupe agit sur $\mg_m\times \mathfrak{gl}_i\times R_i$ par l'action adjointe de $\G_i$ sur $\mathfrak{gl}_i$ et par l'action triviale sur les autres facteurs. On identifiera le groupe $G_{i-1}$ au sous-groupe $\D(G_{i-1},1)$ de $G_i$ (puisque $\det(g)=\det(\D(g,1))$ de sorte que $G_{i-1}$ agit aussi sur $\mg_m\times \mathfrak{gl}_i\times R_i$ par l'action induite. Comme l'action adjointe laisse invariant le polynôme caractéristique, le morphisme $f^Y_i$ est $G_{i-1}$-équivariant. Le morphisme $h'_i$ n'est pas $G_{i-1}$-équivariant mais est néanmoins $G_{i-2}$-équivariant.

Lorsque $P$ un polynôme de plusieurs variables, on note $V(P)$ le schéma des zéros de $P$. D'après le théorème~\ref{pop} ci-dessus, la fonction $\underline{\kappa}$ est un produit de facteurs irréductibles de ${\rm result}(a_{r-1},a_r)$, donc on peut écrire ${\rm result}(a_{r-1},a_r)=\prod_{i}P_i^{m_i}$ et $\underline{\kappa}(w_0(y+\varpi{\rm Id}_r)=\prod_i P_i(y)^{n_i}$ avec $0\leq n_i\leq m_i$. On a $V({\rm result}(a_{r-1},a_r))=\bigcup_i V(P_i^{m_i})$. Soit $h\in \G_{r-1}$, comme l'action adjointe $y\mapsto h.y$  laisse invariant $a_{r-1}(y)$ et $a_r(y)$, cette action laisse invariant ${\rm result}(a_{r-1}(y),a_r(y))$, i.e
   $${\rm result}(a_{r-1}(y),a_r(y))={\rm result}(a_{r-1}(h.y),a_r(h.y)),$$ donc encore
   $$V({\rm result}(a_{r-1}(y),a_r(y)))=V({\rm result}(a_{r-1}(h.y),a_r(h.y))).$$
On considère le morphisme $\rho :\G_{r-1}\times \mathfrak{gl}_r\to \mathfrak{gl}_r\ (h,y)\mapsto h.y$ défini par l'action. Comme $V({\rm result}(a_{r-1}(y),a_r(y)))=V({\rm result}(a_{r-1}(h.y),a_r(h.y)))=V$, on a donc un morphisme $\rho_{|V(P_i)}: \G_{r-1}\times V(P_i)\to V$. Puisque $\G_{r-1}\times V(P_i)$ est irréductible, son image dans $V$ est aussi irréductible. Par ailleurs, cette image contient $V(P_i)$ et lui est donc égale. Par conséquent, on a $V(P_i(h.y))=V(P_i(y))\, \forall\, h\in \G_{r-1},\,\forall y\in \mathfrak{gl}_{r}$, de sorte que $P_i(h.y)=c_i(h,y).P_i(y)$ où $c_i(h,y)$ est inversible (i.e $c_i(h,y)\in (\0_{\G_{r-1}}[(y_{i,j})])^*)$.On a donc $c_i(h,y)=c_i(h)\in \0_{\G_{r-1}}^*$ puisque les inversibles d'un anneau de polynômes $A[x_1,\dots,x_n]$ sont ceux de $A$. On obtient alors $\underline{\kappa}_{\ud}(w_0(h.y+\varpi{\rm Id}_r))=c(h)\underline{\kappa}_{\ud}(w_0(y+\varpi{\rm Id}_r))$ avec $c(h)=\prod_i(c_i(h))^{n_i}$.
\begin{lemma}
On a $c(h.h')=c(h).c(h')$. En particulier, on a un morphisme $c: \G_{r-1}\to \mg_m$.
\end{lemma}
\begin{proof}La démonstration est évidente.
\end{proof}

On considère le diagramme suivant
$$\xymatrix{1\ar[r]&\mathrm{SL}_{r-1}\ar@{^{(}->}[r]&\G_{r-1}\ar[r]^-{\det}\ar[d]_-{c}&\mg_m.\\
&&\mg_m&}$$
En utilisant ${\rm Hom} (\mathrm{SL}_{r-1},\mg_m)=\{1\}$ et la propriété universelle du conoyau , on obtient le diagramme commutatif
$$\xymatrix{\G_{r-1}\ar[r]^-{\det}\ar[d]_-{c}&\mg_m,\ar[dl]\\
\mg_m&}$$
de sorte qu'on a $c(h)=\det(h)^{s}$ pour un certain $s\in \mathbb{Z}$. Donc $\underline{\kappa}$ est $G_{r-1}$-équivariant.
Le théorème~\ref{py} résulte alors de la proposition suivante
\begin{proposition}(cf. \cite[proposition 5.2.2]{N})
Soient $U_i$ et $U_{i-1}$ les images réciproques de $U_{\ud}$ dans $\mathfrak{gl}_i\times R_i$ et dans $\mathfrak{gl}_{i-1}\times R_{i-1}$. Si $\mathcal{J}$ est un faisceau pervers sur $\mathfrak{gl}_i\times R_i$, $G_{i-1}$-équivariant et isomorphe au prolongement intermédiaire à restriction à l'ouvert $U_i$, alors
$$\mathcal{J}'=Rf^Y_{i,!}(\mathcal{J}\otimes h^*_{i}\lsi)[1]$$ est aussi un faisceau pervers sur $\mathfrak{gl}_{i-1}\times R_{i-1}$, $G_{i-2}$-équivariant et isomorphe au prolongement intermédiaire de sa restriction à l'ouvert $U_{i-1}$.
\end{proposition}
\begin{proof}[Démonstration]
Les morphismes qui interviennent dans la formation de $\mathcal{J}'$ sont tous $G_{i-2}$-équivariants donc $\mathcal{J}'$ l'est aussi.

On utilise la transformation de Fourier-Deligne (\cite{L}) pour démontrer la perversité et le prolongement intermédiaire. Soit $E=\mathfrak{gl}_i\times Q_i\times R_i$. Il est clair que $E\times \mg_m\simeq \mathfrak{gl}_{i-1}\times R_{i-1}$. Soient $V$ le fibré trivial $E\times (\mathds{A}^{i-1})^2$ et $V^{\vee}$ son fibré dual. On note $\iota:\mathfrak{gl}_i\times R_i\ra V$ l'immersion fermée définie par $\iota(y_i,r_i)=((y_{i-1},\Delta_i(y_i),r_i),z,z')$, où $y_i$ est de la forme
$\left(\begin{smallmatrix}
y_{i-1}&z'\\
{}^tz&*
\end{smallmatrix}\right)$. On note $\epsilon :E\times\mg_m\ra V^{\vee}$ l'immersion fermée définie par $\epsilon(e,\alpha_i)=(e,{}^t(0,\dots,0,\alpha_i),{}^t(0,\dots,0,\alpha_i))$. On vérifie que $\mathcal{J}'=\epsilon^*\fsi(\iota_*\mathcal{J})[1-2(i-1)]$ (cf. \cite[proposition 5.3.2]{N}), où $\fsi$ est la transformation de Fourier-Deligne. Puisque $\mathcal{J}$ est un faisceau pervers et $\iota$ est une immersion fermée, $\iota_*\mathcal{J}$ est un faisceau pervers. D'après \cite{L}, $\fsi(\iota_*\mathcal{J})$ en est un aussi. L'action de $G_{i-1}$ sur $\mathfrak{gl}_i\times R_i$ s'\'etend à $V$ de la manière suivante :
$$\pi(g,(y_{i-1},a_i,r_i),(z,z'))=((g^{-1}y_{i-1}g,a_i,r_i),({}^tgz,g^{-1}z').$$
Cela induit donc une action sur $V^{\vee}$
$$\check{\pi}(g,(y_{i-1},a_i,r_i),(\check{z},\check{z}')=(({}^tgy_{i-1}g,a_i,r_i),({}^tg^{-1}\check{y},g\check{y}')).$$
Par rapport à cette action, $\fsi(\iota_*\mathcal{J})$ est $G_{i-1}$-équivariant.
On utilise alors le lemme suivant
\renewcommand{\qedsymbol}{}
\end{proof}
\begin{lemma}(cf. \cite[lemme 5.4.3]{N})\label{i2}
Le morphisme composé
$$\xymatrix{G_{i-1}\times E\times \mg_m\ar@{^{(}->}[r]^-{\epsilon}&G_{i-1}\times V^{\vee}\ar[r]^-{\check{\pi}}&V^{\vee}}$$ est un morphisme lisse de dimension relative $(i-1)^2+1-2(i-1)$.
\end{lemma}
\begin{proof}[Démonstration du lemme]
On note $Z$ l'image de l'immersion localement fermée $\epsilon$ dans $V^{\vee}$. Le morphisme composé s'écrit alors :
$$G_{i-1}\times Z\buildrel {\check{\pi}}\over{\ra} V^{\vee}.$$

En oubliant les composantes $Q_i$ et $R_i$ on obtient un diagramme cartésien évident
$$\xymatrix{
G_{i-1}\times Z\ar[r]^{\check{\pi}}\ar[d]&V^{\vee}\ar[d]\\
G_{i-1}\times \mg_m\times \mathfrak{gl}_{i-1}\ar[r]^-{\xi}&\mathfrak{gl}_{i-1}\times (\mathbb{A}^{i-1})^2
}$$
où $\xi(g_{i-1},\alpha_i,y_{i-1})=(g_{i-1}^{-1}y_{i-1}g_{i-1},g_{i-1}{}^t(0,\dots,0,\alpha_i),{}^tg_{i-1}^{-1}{}^t(0,\dots,0,\alpha_i))$.

En factorisant le morphisme $\zeta$ et oubliant le facteur $\mathfrak{gl}_{i-1}$, l'assertion se ram\`ene à démontrer que
le morphisme
$$G_{i-1}\times\mg_m\ra (\mathbb{A}^{i-1})^2,\ (g_{i-1},\alpha_i)\mapsto (g_{i-1}{}^t(0,\dots,0,\alpha_i),{}^tg_{i-1}^{-1}{}^t(0,\dots,0,\alpha_i))$$
est lisse et de dimension relative $(i-1)^2+1-2(i-1)$. Cela résulte de l'assertion suivante :
``L'orbite de l'\'el\'ement  ${}^t(0,\dots,0,1)\in \mathbb A^{{i-1}}$ sous l'action du  groupe $\mathrm{G}_{i-1}\times\mg_m$  sur $(\mathbb{A}^{i-1})^2$ d\'efinie par
$$ ((g_{i-1},\lambda),(z,z'))\mapsto  (\lambda g_{i-1}z,\lambda {}^tg_{i-1}^{-1}z')$$ est l'ouvert défini par l'équation ${}^tzz'\neq 0$."
Cette assertion est démontrée dans \cite[lemme 5.3.5]{N}.
\renewcommand{\qedsymbol}{}
\end{proof}
\begin{proof}[Fin de la démonstration]
D'après \cite[4.2.5]{BBP}, $\check{\pi}^*(\fsi(\iota_*\mathcal{J}))[(i-1)^2+1-2(i-1)]$ est un faisceau pervers sur $G_{i-1}\times Z$. Grâce à la $G_{i-1}$-équivariance on a un isomorphisme
  $$\check{\pi}^*(\fsi(\iota_*\mathcal{J}))=\mathrm{pr}_Z^*(\fsi(\iota_*\mathcal{J})_{|Z}).$$
D'après loc. cit. $\fsi(\iota_*\mathcal{I})_{|Z}[1-2(i-1)]$ est un faisceau pervers sur $Z$, la rojection $\mathrm{pr}_Z: G_{i-1}\times Z \to Z$ étant clairement un morphisme lisse de dimension relative $\dim(G_{i-1})=(i-1)^{2}$. Alors $\mathcal{J}'$ est un faisceau pervers sur $G_{i-1}\times R_{i-1}$ .

De la m\^eme mani\`ere, $\mathcal{J}'$ est le prolongement intermédiaire de sa restriction à $U_{i-1}$.
\end{proof}
\section{Facteur de tranfert}
\subsection{Rappels sur le symbole de Hilbert et les constantes de Weil}
\setcounter{theorem}{0}
\setcounter{subsubsection}{1}

\begin{proposition}(cf. \cite[proposition 1]{M})\label{hilbert}
Soient $a,b\in \fp$. On a :
\begin{enumerate}
\item Si $v(a)$ et $v(b)$ sont impaires, $[a,b]=[-ab,\varpi]$.
\item Si $v(a)$ est paire et $v(b)=1$ , $[a,b]=\zeta(a)$.
\end{enumerate}
\end{proposition}
On peut retrouver les formules de la proposition~\ref{hilbert} ci-dessus en considérant le symbole de Hilbert comme le composé du symbole modéré et du caractère $\zeta$.
\begin{proposition}\label{weil1}(cf. \cite[proposition 2]{M})
Soit $\psi$ un caractère additif d'ordre $0$. On a :
\begin{enumerate}
\item $\gamma(a,\psi)=1 $ si $v(a)$ est paire.
\item $\gamma(a,\psi)=q^{-1/2}\sum_{b\in\op/\varpi\op}\psi(a\varpi^{2r}b/2)[b,\varpi]$, si $|a|=q^{2r+1}$.
\item $\gamma(ab,\psi)=\gamma(a,\psi)\gamma(b,\psi)[a,b]$.
\end{enumerate}
\end{proposition}
\begin{proposition}\label{weil}(c.f \cite[proposition 5, p. 179]{W}) Soient $a\in k(\varpi)$ et $\psi: \mathbb{A}/k(\varpi)\to \mathbb{C}$ un caractère additif. On a : $\prod_{v\in |\mathbb{P}^1|}\gamma_v(a,\psi_v)=1$.
\end{proposition}
\subsection{Géométrisation d'un calcul de Jacquet \cite[paragraphe 8]{J1}}
\setcounter{theorem}{0}
\setcounter{subsubsection}{1}

On commence cette partie en démontrant le cas très particulier du théorè\-me $A'$ de l'introduction où $r=2$ et $t=\D(t_1,t_2)$ avec $v_{\varpi}(t_1)=1, v_{\varpi}(t_2)=-1$. Plus précisement, on montre la proposition suivante :

\begin{proposition}\label{fatf}Pour $r=2$ et $t=\D(t_1,t_2)$ avec $v_{\varpi}(t_1)=1, v_{\varpi}(t_2)=-1$, on a
${\cal J}_{\varpi}(t)\simeq{\cal T}_{\varpi}(t)\otimes{\cal I}_{\varpi}(t)\simeq
{\cal T'}_{\varpi}(t)\otimes{\cal I}_{\varpi}(t),$ o\`u ${\cal T}_{\varpi}(t)$ et ${\cal T'}_{\varpi}(t)$ sont des $\overline{\mathbb{Q}}_{\ell}$-espaces vectoriels de rang $1$ plac\'es en degr\'e $1$ tels que $\mathrm{Tr}(\mathrm{Fr},{\cal T}_{\varpi}(t))=|t_1|^{-1/2}\gamma_\varpi(-t_1,\psi)$ et  $\mathrm{Tr}(\mathrm{Fr},{\cal T'}_{\varpi}(t))=|t_2|^{1/2}\gamma_\varpi(t_2,\psi)$. De plus ${\cal I}_{\varpi}(t)$ et
${\cal J}_{\varpi}(t)$ sont alors des $\overline{{\mathbb{Q}}}_{\ell}^*$-espaces vectoriels de rang $2$ plac\'es respectivement en degr\'e $0$ et $1$.
\end{proposition}
\begin{proof}[Démonstration]
On a besoin du lemme suivant :
\begin{lemma} \label{zeta}
Soient $X'_1=\{(u,\epsilon,c)\in \mathbb{A}^2\times \mg_m|\epsilon^2=c\}$ munie du morphisme 
$\mathrm{pr}_X: X'_1\to \mg_a\times \mg_m$ défini par $(u,\epsilon,c)\mapsto (u,c)$, $Y'_1=\{(u,v,c)\in \mathbb{A}^2\times \mg_m|v^2-uv+c=0\}$ munie du morphisme $\mathrm{pr}_Y:Y'_1 \to \mg_a\times \mg_m$ défini par $(u,c,v)\mapsto (u,c)$ et  $j :\mg_m \hookrightarrow \mg_a$ l'immersion canonique. On a alors un isomorphisme canonique
$$R\mathrm{pr}_{X,!}(Rj_!\lze(u+2\epsilon))\simeq R\mathrm{pr}_{Y,!}(Rj_{!}\lze(v)),$$
où pour tout $k$-schéma $S$ muni d'un morphisme $S\to \mg_m$, on note $\lze(s)=s^*\lze$.

\end{lemma}
\begin{proof}
On note $$X'_2=\{(u,\xi,\xi',\epsilon,c)\in \mathbb{A}^4\times \mg_m|\epsilon^2=c,\,\xi^2=u+2\epsilon,\,{\xi'}^2=u-2\epsilon\},$$
$$Y'_2=\{(u,v,w,w',c)\in \mathbb{A}^4\times \mg_m|v^2-uv+c=0,\,w^2=v,\,{w'}^2=v'=u-v
\}$$
et $Z=\{(u,c)\in \mathbb{A}^1\times \mg_m\}$.
On considère le diagramme suivant
$$\xymatrix{
X'_2
\ar[d]^-\circlearrowright_-{\begin{smallmatrix}\mathbb{Z}/2\mathbb{Z}& x\mapsto -x\\
\times&\\ \mathbb{Z}/2\mathbb{Z}&y\mapsto -y\end{smallmatrix}}&&
Y'_2
\ar[d]_-{\circlearrowright}^-{\begin{smallmatrix} \mathbb{Z}/2\mathbb{Z}& w\mapsto -w\\
\times&\\ \mathbb{Z}/2\mathbb{Z}&w'\mapsto -w'\end{smallmatrix}}\\
X'_1
\ar[dr]^-{\circlearrowright}_-{\mathbb{Z}/2\mathbb{Z}\,\epsilon\mapsto -\epsilon}&&
Y'_1\ar[dl]_-{\circlearrowright }^-{\mathbb{Z}/2\mathbb{Z}\, v\mapsto v'=u-v}\\
&Z.&}$$

Soit $\buildrel\circ\over{Z}=Z-\{(u,c)|u^2-4c=0\}$. On ajoute un $\circ$ pour indiquer le changement de base de $Z$ à $\buildrel\circ\over{Z}$. On note $j_X :\buildrel\circ\over{X'}_1 \hookrightarrow X'_1$,
$j_Y :\buildrel\circ\over{Y'}_1 \hookrightarrow Y'_1$ et $j_Z:\buildrel\circ\over{Z}\to Z$ . On obtient les isomorphismes suivants :
$$j_{X,*}j_X^*(Rj_!\lze(u+2\epsilon))\simeq Rj_!\lze(u+2\epsilon),$$
$$j_{Y,*}j_Y^*(Rj_!\lze(v))\simeq Rj_!\lze(v),$$
desquels résultent les isomorphismes
$$pr_{X,*}j_!\lze(u+2\epsilon)\simeq j_{Z,*}j_Z^*(pr_{X,*}j_!\lze(u+2\epsilon)),$$
$$pr_{Y,*}j_!\lze(v)\simeq j_{Z,*}j_Z^*(pr_{Y,*}j_!\lze(v)),$$
de sorte qu'il suffira de construire l'isomorphisme au-dessus de $\buildrel\circ\over{Z}$.
Les morphismes $\buildrel\circ\over{X'}_2\to \buildrel\circ\over{X'}_1\to \buildrel\circ\over{Z}$ et $\buildrel\circ\over{Y'}_2\to \buildrel\circ\over{Y'}_1\to \buildrel\circ\over{Z}$
sont étales et les deux membres de l'isomorphisme qu'on veut obtenir sont donc des systèmes locaux lorsqu'on les restreint à $\buildrel\circ\over{Z}$. On considérera ces systèmes locaux comme des représentations des groupes de Galois des revêtements $\buildrel\circ\over{X'}_2/\buildrel\circ\over{Z}$ et $\buildrel\circ\over{Y'}_2/\buildrel\circ\over{Z}$.


On a un isomorphisme $\buildrel\circ\over{X'}_2\buildrel \thicksim\over{\to}\buildrel\circ\over{Y'}_2$ :
\begin{itemize}
\item $\buildrel\circ\over{X'}_2{\to}\buildrel\circ\over{Y'}_2$ :
$$\begin{matrix}v=\frac{1}{2}(u+\xi\xi'),\\
w=\frac{1}{2}(\xi+\xi'),\\
w'=\frac{1}{2}(\xi-\xi').
\end{matrix}$$
\item $\buildrel\circ\over{Y'}_2\buildrel \thicksim\over{\to}\buildrel\circ\over{X'}_2$ :
$$\begin{matrix}\epsilon=ww',\\
\xi=w+w',\\
\xi'=w-w'.
\end{matrix}$$
\end{itemize}
Cet isomorphisme est équivariant vis à vis de l'isomorphisme des groupes de Galois
$$\mathbb{Z}/2\mathbb{Z}\ltimes(\mathbb{Z}/2\mathbb{Z})^2\to \mathbb{Z}/2\mathbb{Z}\ltimes(\mathbb{Z}/2\mathbb{Z})^2$$ donné par
$$\begin{matrix}(\epsilon\mapsto -\epsilon, \xi\leftrightarrows\xi')&\mapsto&(w\mapsto w, w\mapsto w')\\
(\epsilon\mapsto \epsilon, \xi\mapsto \xi,\xi'\mapsto -\xi')&\mapsto&(w\leftrightarrows w').
\end{matrix}$$

Nos deux systèmes locaux correspondent respectivement aux représen\-tations des groupes de Galois
sur $\overline{\mathbb{Q}}_{\ell}^2$ données par

$$\begin{matrix}(\epsilon\mapsto -\epsilon, \xi\leftrightarrows\xi')&\mapsto& \left(\begin{smallmatrix}0&1\\1&0\end{smallmatrix}\right)\\
(\epsilon\mapsto \epsilon, \xi\mapsto \xi,\xi'\mapsto -\xi')&\mapsto &\left(\begin{smallmatrix}1&0\\0&-1
\end{smallmatrix}\right)\end{matrix}$$
et
$$\begin{matrix}(w\mapsto w, w\mapsto w')&\mapsto& \left(\begin{smallmatrix}1&0\\0&-1\end{smallmatrix}\right)\\
(w\leftrightarrows w')&\mapsto&\left(\begin{smallmatrix}0&1\\1&0\end{smallmatrix}\right).\end{matrix}$$


Par conséquent, le lemme résulte du diagramme commutatif suivant
$$\xymatrix{
\G_2(\overline{\mathbb{Q}}_{\ell}^2)\ar[r]^-{A}&\G_2(\overline{\mathbb{Q}}_{\ell}^2)\\
\mathbb{Z}/2\mathbb{Z}\ltimes(\mathbb{Z}/2\mathbb{Z})^2\ar[u]\ar[r]^{\thicksim} &\mathbb{Z}/2\mathbb{Z}\ltimes(\mathbb{Z}/2\mathbb{Z})^2\ar[u]}$$
où $A$ est l'isomorphisme défini par $M\mapsto AMA^{-1}$ avec $A=\left(\begin{smallmatrix}1&1\\1&-1\end{smallmatrix}\right)$.


\end{proof}
\begin{rmk}
En utilisant la formule des traces de Grothendieck-Lefsch\-etz pour les deux systèmes locaux dans le lemme ci-dessus sur $(c,u)$, on obtient :
$$\sum_{v+c/v=u}\zeta(v)=\sum_{\epsilon^2=c}\zeta(u-2\epsilon).$$

Pour $c=1$, on obtient la démonstration d'une formule de \cite[paragraphe 8, p.145 (entre les formules 127 et 128)]{J1}
$$\sum_{v+1/v=u}\zeta(v)=\zeta(u-2)+\zeta(u+2),$$
 à l'aide de laquelle Jacquet obtient essentiellement le cas particulier du lemme fondamental de Jacquet et Mao obtenu en prenant les traces de Frobenius dans la proposition \ref{fatf}.
\end{rmk}
\renewcommand{\qedsymbol}{}
\end{proof}
\begin{proof}[Démonstration de la proposition \ref{fatf}]
Soient $t_1=t_1[1]\varpi+t_1[2]\varpi^2+\dots$ et $t_2=t_2[-1]\varpi^{-1}+t_2[0]\varpi^0+\dots$. On note $c=-t_2[-1]/t_1[1]$. Soient $X(c)={\rm Spec}\,k[\epsilon]/(\epsilon^2-c)$ munie du morphisme $h :X(c)\to \mg_a$ défini par $h(\epsilon)=x$ et $Y(c)={\rm Spec}\,k[v,v']/(vv'-c)$ munie du morphisme $h': Y(c)\to \mg_a$ défini par $h'(v,v')=\frac{v+v'}{2}$ et du morphisme $\underline{\kappa} : Y(c)\to \mg_m$ défini par $\underline{\kappa}(v,v')=-vt_1[1]$ (la définition de $\underline{\kappa}$ vient de la formule de Kubota (cf. la proposition \ref{kubo})). D'après les définitions des variétés $X_\varpi(t)$ et $Y_\varpi(t)$ (voir les sections 2 et 4), $(X_\varpi(t),h_\alpha)$ est donc isomorphe à $(X(c),h)$ et $(Y_\varpi(t),h'_\alpha,\underline{\kappa})$ est isomorphe à $(Y(c),h',\underline{\kappa})$.
En faisant varier $c$ et $d:=t_1[1]$, on obtient le diagramme suivant :

$$\xymatrix{
X_1\ar[d]\ar[dr]&&Y_1\ar[dl]\\
X_0\ar[dr]&Z_1\ar[d]&\\
&Z_0
}$$ où
\begin{itemize}
\item $X_1=\{(c,d,u,\epsilon,x)\in \mg_m^2\times \mathbb{A}^3|\epsilon^2=c,x=u+2\epsilon\}$,
\item $Y_1=\{(c,d,u,v)\in \mg_m^2\times \mathbb{A}^2|v^2-uv+c=0\}$,
\item $X_0=\{(c,d,x)\in \mg_m^2\times\mathbb{A}\}$,
\item $Z_1=\{(c,d,u)\in \mg_m^2\times\mathbb{A}\}$,
\item $Z_0=\{(c,d)\in \mg_m^2\}$.
\end{itemize}
On a alors les isomorphismes canoniques suivants
\begin{eqnarray*}
\mathcal{J}_\varpi&\simeq& R(Y_1\to Z_0)_{!}\lze(-dv)\otimes \lsi(\frac{u}{2})\\
&\simeq&\lze(-d)\otimes R(Z_1\to Z_0)_!\lsi(\frac{u}{2})\otimes R(Y_1\to Z_1)_!\lze(v)
\end{eqnarray*}
En utilisant le lemme \ref{zeta}, on obtient
$$R(Y_1\to Z_1)_!\lze(v)\simeq R(X_1\to Z_1)_!\lsi(u+2\epsilon).$$
Par conséquent, on a :
\begin{eqnarray*}
\mathcal{J}_\varpi&\simeq&\lze(-d)\otimes R(Z_1\to Z_0)_!\lsi(\frac{u}{2})\otimes R(X_1\to Z_1)_!\lsi(u+2\epsilon)\\
&\simeq&\lze(-d)\otimes R(X_1\to Z_0)_!(\lsi(\frac{u}{2})\otimes \lsi(u+2\epsilon))\\
&\simeq&\lze(-d)\otimes R(X_1\to Z_0)_!(\lsi(-\epsilon)\otimes\lsi(\frac{x}{2})\otimes \lsi(x))\\
&\simeq&\lze(-d)\otimes R(X_0\to Z_0)_!(\lsi(\frac{x}{2})\otimes \lsi(x))\otimes R(X_1\to X_0)_!\lsi(-\epsilon)\\
&\simeq&\lze(-d)\otimes R(X_0\to Z_0)_!(\lsi(\frac{x}{2})\otimes \lsi(x))\otimes \mathcal{I}_\varpi
\end{eqnarray*}
On pose $\mathcal{T}_{\varpi}=\lze(-d)\otimes R(X_0\to Z_0)_!(\lsi(\frac{x}{2})\otimes \lsi(x))$ et $\mathcal{T}'_{\varpi}=\mathcal{T}_\varpi\otimes \lze(c)$. On a donc $\mathcal{J}_{\varpi}\simeq\mathcal{T}_{\varpi}\otimes \mathcal{I}_{\varpi}$. Comme par ailleurs, on a isomorphisme canonique $\mathcal{I}_\varpi\simeq \mathcal{I}_\varpi\otimes \lze(c)$.

En utilisant la formule des traces de Grothendieck-Lefschetz, on obtient bien :
\begin{eqnarray*}
\mathrm{Tr}(\mathrm{Fr}_{(c,d)},\mathcal{T}_\varpi)&=&\sum_{x\in k}\psi(x/2)\zeta(-t_1[1]x)\\
&=&\sum_{x\in k}\psi(-t_1[1]x/2)\zeta(x)\\
&=&|t_1|_{\varpi}^{-1/2}\gamma_\varpi(-t_1,\Psi)
\end{eqnarray*}
et
\begin{eqnarray*}
\mathrm{Tr}(\mathrm{Fr}_{(c,d)},\mathcal{T}'_\varpi)&=&\zeta(c)\mathrm{Tr}(\mathcal{T}_\varpi)\\
&=&\zeta(-t_2[-1]/t_1[1])|t_1|_{\varpi}^{-1/2}\gamma_\varpi(-t_1,\Psi)\\
&=&|t_2|^{1/2}\gamma_{\varpi}(t_2,\Psi).
\end{eqnarray*}

\end{proof}

La proposition suivante ramène le cas simple des matrices $t\in T_r(\fp)$ telles que $v(\prod_{i=1}^{r-1}a_i)=1$ et $v(a_r)=0$ à celui envisagé dans la proposition \ref{fatf}.
\begin{proposition}\label{ssp} Soient $t\in T_r(\fp)$ et $i$ un nombre entier compris entre $1$ et $r-1$, tels que $v(a_i)=1$ et $v(a_j)=0$ pour $j\neq i$. Alors, pour tout $\underline{\alpha}\in (k^*)^{r-1}$, on a un isomorphisme entre $(X_\varpi(t),h_{\underline{\alpha}})$ et $(X_\varpi(\D(\frac{a_i}{a_{i-1}},\frac{a_{i+1}}{a_i})),\break h_{\alpha_{i+1}})$ et un isomorphimse entre $(Y_\varpi(t),h'_{\underline{\alpha}},\underline{\kappa}_\varpi)$ et $(Y_\varpi(\D(\frac{a_i}{a_{i-1}},\frac{a_{i+1}}{a_i})),\break h'_{\alpha_{i+1}},\underline{\kappa})$.
\end{proposition}
\begin{proof}
On adapte la démonstration de \cite[proposition 2.4.1]{N} pour le premier isomorphisme. Pour le deuxième isomorphisme, on adapte encore cette démonstration en utilisant la formule : $$\underline{\kappa}\left(\begin{matrix} &g_1\\g_2&\end{matrix}\right)=\underline{\kappa}\left(\begin{matrix} g_1&\\&g_2\end{matrix}\right)=\underline{\kappa}(g_1)\underline{\kappa}(g_2).$$
\end{proof}
\subsection{Facteur de transfert}
\setcounter{theorem}{0}
\setcounter{subsubsection}{1}

\begin{proposition}\label{prop1}
Pour $\ud=(d_1,\dots,d_r)$, la restriction \`a $U_{\ud}$ du complexe ${\cal I}=\mathrm{R}f^X_{\ud,!}h_{\ud}^*\mathcal{L}_{\psi}$ (resp. ${\cal J}=\mathrm{R}f^Y_{\ud,!}({h'}_{\ud}^*\mathcal{L}_{\psi}\otimes\underline{\kappa}_{\ud}^*\mathcal{L}_{\zeta})$) est un syst\`eme local de rang $2^{\sum_{i=1}^rd_i}$ plac\'e en degr\'e $0$ (resp. plac\'e en degr\'e $\sum_{i=1}^rd_i$). De plus
${\cal J}_{|U_{\ud}}={\cal T}\otimes {\cal I}_{|U_{\ud}}$, o\`u $\cal T$ est un syst\`eme local de rang $1$ plac\'e en degr\'e $\sum_{i=1}^rd_i$ au-dessus de $U_{\ud}$, g\'eom\'etriquement constant et provenant d'un caract\`ere $\tau $ de ${\rm Gal}_{\overline{k}/k}$ tel que
$$\tau({\rm Fr}_q)=\left\{\begin{smallmatrix}(-1)^{\sum_{i=1}^rd_i}q^{\sum_{i=1}^rd_i/2}\zeta(-1)^{\sum_{i=0}^{s-1}d_{2i+1}}\gamma_\infty(\varpi,\Psi_\infty)^{-\sum_{i=1}^{s-1}p(d_{2i}-d_{2i+1})}\,\text{si $r=2s$},\\
(-1)^{\sum_{i=1}^rd_i}q^{\sum_{i=1}^rd_i/2}\zeta(-1)^{\sum_{i=1}^{s}d_{2i}}\gamma_\infty(\varpi,\Psi_\infty)^{-\sum_{i=1}^{s}p(d_{2i}-d_{2i-1})}\,\text{si $r=2s+1$}.
\end{smallmatrix}\right.$$
\end{proposition}
\begin{proof}
En utilisant la proposition~\ref{fatf}, la proposition~\ref{ssp} et la formule de produit, on obtient les assertions sur les systèmes locaux $\mathcal{I}_{|U_{\ud}}$ et $\mathcal{J}_{|U_{\ud}}$. Pour le facteur de tranfert on envisage deux cas suivants :
\begin{enumerate}
\item Si $r=2s$, on écrit le facteur de transfert sous la forme :
\begin{multline*}{\bf T}=\prod_{v|a_1}(q^{1/2}\gamma_v(-a_1,\Psi_v))\times\\
\times\prod_{i=1}^{s-1}\left(\prod_{v|a_{2i}}(q^{1/2}\gamma_v(a_{2i+1}/a_{2i},\Psi_v))\prod_{v|a_{2i+1}}(q^{1/2}\gamma_v(-a_{2i+1}/a_{2i},\Psi_v))\right).
\end{multline*}
\item Si $r=2s+1$ , on écrit le facteur de transfert sous la forme :
$${\bf T}=\prod_{i=1}^{s}\left(\prod_{v|a_{2i-1}}(q^{1/2}\gamma_v(a_{2i}/a_{2i-1},\Psi_v))\prod_{v|a_{2i}}(q^{1/2}\gamma_v(-a_{2i}/a_{2i-1},\Psi_v))\right).$$
\end{enumerate}
On se contente de donner le calcul pour le premier cas où $r=2s$ (pour le cas où $r=2s+1$, on utilise le même argument). En fait,
\begin{eqnarray*}
\prod_{v|a_1}(q^{1/2}\gamma_v(-a_1,\Psi_v))&=&\prod_{v|a_1}(q^{1/2}\gamma_v(a_1,\Psi_v)\gamma_v(-1,\Psi_v)[-1,a_1]_v)\\
&=&\prod_{v|a_1}(q^{1/2}\gamma_v(a_1,\Psi_v)\zeta(-1))\\
&=&q^{d_1/2}\zeta(-1)^{d_1}\prod_{v|a_1}\gamma_v(a_1,\Psi_v)\\
&=&q^{d_1/2}\zeta(-1)^{d_1}\gamma_\infty(a_1,\Psi_\infty)^{-1}.
\end{eqnarray*}
On a aussi (comme $\gamma_v(-a_{2i+1}/a_{2i},\Psi_v)=\zeta(-1)\gamma_v(a_{2i+1}/a_{2i},\Psi_v)$, où la démonstration est la même que ci-dessus avec $a_1$)
\begin{eqnarray*}
&&\prod_{v|a_{2i}}(q^{1/2}\gamma_v(a_{2i+1}/a_{2i},\Psi_v))\prod_{v|a_{2i+1}}(q^{1/2}\gamma_v(-a_{2i+1}/a_{2i},\Psi_v))\\
&=&\prod_{v|a_{2i}}(q^{1/2}\gamma_v(a_{2i+1}/a_{2i},\Psi_v))\prod_{v|a_{2i+1}}(q^{1/2}\gamma_v(a_{2i+1}/a_{2i},\Psi_v)\zeta(-1)\\
&=&q^{d_{2i}+d_{2i+1}}\zeta(-1)^{d_{2i+1}}\prod_{v|a_{2i}a_{2i+1}}\gamma_v(a_{2i+1}/a_{2i},\Psi_v)\\
&=&q^{d_{2i}+d_{2i+1}}\zeta(-1)^{d_{2i+1}}\gamma_\infty(a_{2i+1}/a_{2i},\Psi_\infty)^{-1}
\end{eqnarray*}
Par ailleurs, les $a_i$ sont des polynôme unitaire, donc $a_{2i+1}/a_{2i}=(\varpi^{-1})^{d_{2i}-d_{2i+1}}+...$ (c'est à dire $v_\infty(a_{2i+1}/a_{2i})=d_{2i}-d_{2i+1}$). En utilisant les deux assertions premières la proposition ~\ref{weil1} on a alors
$$\gamma_\infty(a_{2i+1}/a_{2i},\Psi_\infty)=\begin{cases}1&\text{ si $d_{2i}-d_{2i+1}$ est paire,}\\
\gamma_\infty(\varpi,\Psi_\infty)& \text{ si $d_{2i}-d_{2i+1}$ est impaire.}\end{cases}$$
Par conséquent, on obtient la formule de trace de $\mathcal{T}_{|U_{\ud}}$. D'après le théorème de Chebotarev, on voit que $\mathcal{T}$ est géométriquement constant.
\end{proof}
Le facteur de tranfert est géométriquement constant au-dessus de $U_{\ud}$ et se prologne alors de manière évidente $\mathcal{T}$ à $V'_{\ud}\times \mg_m^{r-1}$.

En utilisant les théorèmes~\ref{px} et~\ref{py}, on obtient l'énoncé global pour $\ud=(1,2\dots,r)$, qui prolonge l'égalité de la proposition~\ref{prop1} ci-dessus au-dessus de  $V'_{\ud}\times \mg_m^{r-1}$ tout entier.
\begin{thmB}\label{thm4} Pour $\ud=(1,2\dots,r)$, $\mathrm{R}f^Y_{\ud,!}(h_{Y,\ud}^*\mathcal{L}_{\psi}\otimes\underline{\kappa}_{\ud}^*\mathcal{L}_{\zeta})={\cal T}\otimes \mathrm{R}f^X_{\ud,!}h_{X,\ud}^*\mathcal{L}_{\psi}$. Les deux membres de cette \'egalit\'e sont, \`a d\'ecalage pr\`es, des faisceaux pervers isomorphes au prolongement interm\'ediaire de leur restriction \`a $U_{\ud}\times\mg_m^{r-1}$.
\end{thmB}
\section{L'énoncé local résulte de l'énoncé global}
\setcounter{subsubsection}{1}
\setcounter{theorem}{0}
\setcounter{subsection}{1}
Dans cette partie, on va montrer que le théorème B entraine une version géométrique du théorème A.
\begin{thmA'} Soit $t'=\D(t'_1,\dots,t'_s)\in T_s(\fp)$. On a alors
${\cal J}_{\varpi}(t')\simeq{\cal T}_{\varpi}(t')\otimes{\cal I}_{\varpi}(t')\simeq
{\cal T'}_{\varpi}(t')\otimes{\cal I}_{\varpi}(t'),$ o\`u ${\cal T}_{\varpi}(t')$ et ${\cal T'}_{\varpi}(t)$ sont des $\overline{\mathbb{Q}}_{\ell}$-espaces vectoriels de rang $1$ plac\'es en degr\'e $v_{\varpi}(\prod_{i=1}^{r-1}a_i)$ tels que $\mathrm{Tr}(\mathrm{Fr},{\cal T}_{\varpi}(t'))=\mathfrak{t}_{\varpi}(t',\alpha)$ et $\mathrm{Tr}(\mathrm{Fr},{\cal T'}_{\varpi}(t'))=\mathfrak{ t'}_{\varpi}(t',\alpha)$ avec $\mathfrak{t}_\varpi$ et $\mathfrak{t}'_\varpi$ sont analogues celles du théorème $A$ (cf. la section d'introduction).
\end{thmA'}
D'apr\`es \cite[prop 3.5.1, p. 505]{N}, pour $r$ assez grand, il existe $t^{\circ}=\D(a_1^{\circ},\break a_2^{\circ}/a_1^{\circ}, \dots,a_r^{\circ}/a_{r-1}^{\circ})$, avec $(a_1^{\circ},\dots,a_r^{\circ})\in V_{(1,2,\dots,r)}(k)$ tel que ${\cal I}_\varpi(t^{\circ})\simeq {\cal I}_\varpi(t')$ et ${\cal J}_\varpi(t^{\circ})\simeq {\cal J}_\varpi(t')$. De plus, en utilisant la proposition~\ref{intro3} on a :
$${\cal I}(t^\circ)\simeq{\cal I}_\varpi(t')\otimes\bigotimes_{{\small \begin{smallmatrix} v\neq\varpi \\v\in{\rm supp}(t^\circ)\end{smallmatrix}}}{\cal I}_v(t^\circ)\text{ et }{\cal J}(t^\circ)\simeq{\cal J}_\varpi(t')\otimes\bigotimes_{{\small \begin{smallmatrix} v\neq \varpi\\v\in{\rm supp}(t^\circ)\end{smallmatrix}}}{\cal J}_v(t^\circ). $$
On a $a^{\circ}_i={a'}^{\circ}_i{a''}^{\circ}_i$, o\`u ${a'}^{\circ}_i$ est \`a racines simples non nulles et o\`u  ${a''}^{\circ}_i$ a toutes ses racines nulles. Soient alors $\ud'=(\deg({a'}_i^{\circ}))_i$ et $\ud''=(\deg({a''}_i^{\circ}))_i$. On fait varier $(a'_i)_i$ et $(a''_i)_i$ en introduisant l'ouvert $(V_{\ud'}\times V_{\ud''})^{{\rm dist}}$ de $V_{\ud'}\times V_{\ud''}$ au-dessus duquel ${\rm pgcd}(\prod_{i=1}^ra'_i,\prod_{i=1}^ra''_i)=1$ et les $a'_i$ sont \`a racines simples.  On a alors un morphisme \'etale $\mu:(V_{{\ud}^{ '}}\times V_{{\ud}^{''}})^{{\rm dist}}\ra V_{\ud}$.

Du côté de l'intégrale $I$, pour tout couple $(t',t'')=(\D(a'_1,\dots,a'_r),\break\D(a'_1,\dots,a'_r))$ où $((a'_1,\dots,a'_r),(a''_1,\dots,a''_r))\in (V_{\ud'}\times V_{\ud''})^{{\rm dist}}$, on introduit les deux données géométriques, $(X_1(t',t''),h_{1,\underline{\alpha}})$ et $(X_2(t',t''),h_{2,\underline{\alpha}})$, où
{\small $$X_1(t',t'')(k)=\{n\in N_r(F)/N_r(\0[\det(t'')^{-1}])|{}^tnt't''n\in \mathfrak{gl}_r(\0[\det(t'')^{-1}])\},$$
$$X_2(t',t'')(k)=\{n\in N_r(F)/N_r(\0[\det(t')^{-1}])|{}^tnt't''n\in \mathfrak{gl}_r(\0[\det(t')^{-1}])\}
$$}
et les morphismes $$h_{1,\underline{\alpha}}(n)=\sum_{i=2}^r\left(\alpha_i\sum_{v\nshortmid\prod_{i=1}^ra''_i}{\rm tr}_{k_v/k}{\rm res}_v(n_{i-1,i}{\rm d}\varpi)\right),$$ $$h_{2,\underline{\alpha}}(n)=\sum_{i=2}^r\left(\alpha_i\sum_{v\nshortmid\prod_{i=1}^ra'_i}{\rm tr}_{k_v/k}{\rm res}_v(n_{i-1,i}{\rm d}\varpi)\right).$$
Du côté de l'intégrale $J$, pour tout couple $(t',t'')=(\D(a'_1,\dots,a'_r),\break\D(a'_1,\dots,a'_r))$ où $((a'_1,\dots,a'_r),(a''_1,\dots,a''_r))\in (V_{\ud'}\times V_{\ud''})^{{\rm dist}}$, on introduit aussi les deux données géométriques  $(Y_1(t',t''),h'_{1,\underline{\alpha}},\underline{\kappa}_1)$ et $(Y_2(t',t''),h'_{2,\underline{\alpha}},\underline{\kappa}_2)$ où
{\small $$
Y_1(t',t'')(k)=\{(n,n')\in (N_r(F)/N_r(\0[\det(t'')^{-1}]))^2|{}^tnt't''n'\in \mathfrak{gl}_r(\0[\det(t'')^{-1}])\},
$$
$$
Y_2(t',t'')(k)=\{(n,n')\in (N_r(F)/N_r(\0[\det(t')^{-1}]))^2|{}^tnt't''n'\in \mathfrak{gl}_r(\0[\det(t')^{-1}])\}
$$}
et les morphismes
$$h'_{1,\underline{\alpha}}(n,n')=\frac{1}{2}\sum_{i=2}^r\left(\alpha_i\sum_{v\nshortmid\prod_{i=1}^ra''_i}{\rm tr}_{k_v/k}{\rm res}_v((n_{i-1,i}+n'_{i-1,i}){\rm d}\varpi)\right),$$
$$\underline{\kappa}_1=\prod_{v\nshortmid\, a'_r\prod_{i=1}^ra''_i}\underline{\kappa}_v,$$
$$h'_{2,\underline{\alpha}}(n,n')=\frac{1}{2}\sum_{i=2}^r\left(\alpha_i\sum_{v\nshortmid\prod_{i=1}^ra'_i}{\rm tr}_{k_v/k}{\rm res}_v((n_{i-1,i}+n'_{i-1,i}){\rm d}\varpi)\right),$$
$$\underline{\kappa}_2=\prod_{v\nshortmid\, a''_r\prod_{i=1}^ra'_i}\underline{\kappa}_v$$
(le morphisme $\underline{\kappa}_1$ (resp. $\underline{\kappa}_2$) a une présentation géométrique comme dans la sous-section \ref{glokap} où $S=\{v|v\setminus a'_r\prod_{i=1}^ra''_i\}$ (resp. où $S=\{v| v\setminus a''_r\prod_{i=1}^ra'_i\}$).

Ces données géométriques se mettent en familles (on va les nommer respectivement $(X_{1,\ud'},f^{X_1}_{\ud'},h_{1,\ud'})$, $(X_{2,\ud''},f^{X_2}_{\ud''},h_{2,\ud''})$, $(Y_{1,\ud'},f^{Y_1}_{\ud'},h'_{1,\ud'},\underline{\kappa}_{1,\ud'})$ et $(Y_{2,\ud''},f^{Y_2}_{\ud''},h'_{2,\ud''},\underline{\kappa}_{2,\ud''})$) quand $\ud'$ et $\ud''$ sont fixés. Plus précisément on a les lemmes suivants :
\begin{lemma}
\begin{enumerate}
\item Pour tout $\underline{d}',\ud''\in \mathbb{N}^r$ le foncteur $X_{1,\underline{d}'}$ qui associe à toute $k[(a'_{i,j})_{i,j},(a''_{i,j})_{i,j}]$-algèbre ($a'_{i,j}$ (resp. $a''_{i,j}$) sont les coefficients de $a'_i$ (resp. de $a''_i$) - c'est à dire $a'_i=\varpi^{d'_i}+a'_{i,d'_i-1}\varpi^{d'_i-1}+\dots+a'_{i,0}$) l'ensemble
    \begin{eqnarray*}X_{1,\underline{d}'}(R)&=&\{g\in S_r((\0\otimes_k R)[(\prod_{i=1}^r a''_i)^{-1}])|\\&&
    \det(g_i)=a'_ia''_i\}/N_r((\0\otimes_k R)[(\prod_{i=1}^r a''_i)^{-1}]),
    \end{eqnarray*}
    où $g_i$ est la sous-matrice de $g$ faite des $i$-premières lignes et des $i$-premières colonnes de $g$, est représenté par une variété affine de type fini sur $k$, qu'on note aussi $X_{1,\underline{d}'}$ (resp. $X_{2,\underline{d}''}$).
    Soient
    $$f^{X_1}_{\underline{d}'}: X_{1,\underline{d}'}\times \mg_m^{r-1}\ra (V_{\ud'}\times V_{\ud''})^{\rm dist}\times \mg_m^{r-1}$$
    les morphismes définis par $f^{X_1}_{\underline{d}'}(g,\underline{\alpha})=((a'_j)_{1\leq j\leq r},(a''_j)_{1\leq i\leq r},\underline{\alpha})$ où $a'_ja''_j=\det(g_j)$.
\item Pour tout $i $ avec $2\leq i \leq r$, l'application $h_{i}:S_r((\0\otimes_kR)[(\prod_{j=1}^ra''_j)^{-1}])\times (R^*)^{r-1}\ra R$ définie par
    $$    h_{1,i}=\mathrm{res}\left((a'_{i-1}a''_{i-1})^{-1}\left((g_{i,1},\dots,g_{i,i-1})(a'_{i-1}a''_{i-1})g_{i-1}^{-1}\left(\begin{matrix}0\\\vdots\\0\\\alpha_i\end{matrix}\right)\right)\right),
    $$
    où $(a'_ia''_i)g_i^{-1}$ est la matrice des cofacteurs de $g_i$, lesquels sont dans $(\0\otimes_kR)[(\prod_{i=1}^ra''_i)^{-1}]$ et où $\mathrm{res}((a'_{i-1}a''_{i-1})^{-1} \frac{b}{(\prod_{j=1}^r{a''_j})^s})$ est le coefficient de $\varpi^{(d'_i+d''_i)+s\sum_{j=1}^rd''_j-1}$ dans l'expression polynomiale en la variable $\varpi$ du reste de la
    division euclidienne de $b$ par $a'_{i-1}a''_{i-1}(\prod_{j=1}^ra''_j)^s$ (cette division euclidienne a un sens puisque le coefficient dominant des $a'_i$ et des $a''_i$ sont égals à $1$), induit un morphisme $h_{1,i} : X_{1,\underline{d}'}\times \mg_m^{r-1}\ra\mg_a$.
\item Soit $h_{1,\underline{d}'}=\sum_{i=2}^rh_{1,i}$. Alors pour tous $(t',t'')\in (V_{\ud'}\times V_{\ud''})^{\rm dist}(k)$ et $\underline{\alpha} \in \mg_m^{r-1}$, le couple $(X_1(t',t''),h_{1,\underline{\alpha}})$ est isomorphe à la fibre $(f^{X_1}_{\underline{d}'})^{-1}(t',t'',\underline{\alpha})$ munie de la restriction de $h_{1,\underline{d}'}$ à cette fibre.
\end{enumerate}

\end{lemma}
\begin{lemma}(cf. \cite[proposition 3.3.1]{N})
\begin{enumerate}
\item Pour tout $\underline{d}',\ud''\in \mathbb{N}^r$ le foncteur $Y_{1,\underline{d}'}$ qui associe à toute $k[(a'_i)_i,(a''_i)_i]$-algèbre $R$ l'ensemble
{\small\begin{multline*}Y_{1,\underline{d}'}(R)={}^tN_r((\0\otimes_kR)[(\prod_{i=1}^ra''_i)^{-1}])\setminus\{g\in \mathfrak{gl}_r((\0\otimes_k R)[(\prod_{i=1}^ra''_i)^{-1}])|\\\det(g_i)=a'_ia''_i,
\mathrm{pgcd}(\prod_{i=1}^{r-1}\det(g_i),\det(g_r))=1\}/N_r((\0\otimes_k R)[(\prod_{i=1}^ra''_i)^{-1}]),
\end{multline*}}
    où $g_i$ est la sous-matrice de $g$ faite des $i$-premières lignes et des $i$-premières colonnes de $g$, est représenté par une variété affine de type fini sur $k$ qu'on note aussi $Y_{1,\underline{d}'}$.
    Soit
    $$f^{Y_1}_{\underline{d}'}: Y_{1,\underline{d}'}\times \mg_m^{r-1}\ra (V_{\underline{d}'}\times V_{\ud''})^{\rm dist}\times \mg_m^{r-1}$$
    le morphisme défini par $f^{Y_1}_{\underline{d}'}(g,\underline{\alpha})=((a'_i)_{1\leq i\leq r},(a''_i)_{1\leq i\leq r},\underline{\alpha})$ où $a'_ia''_i=\det(g_i)$
\item Pour tout $i$ avec $2\leq i \leq r$, l'application $h'_{1,i}:\mathfrak{gl}_r((\0\otimes_kR)[(\prod_{j=1}^ra''_j)^{-1}])\times (R^*)^{r-1}\ra R$ définie par
    \begin{eqnarray*}
    h'_{i}&=&\frac{1}{2}\,{\rm res}\left(a_{i-1}^{-1}\left((g_{i,1},\dots,g_{i,i-1})a_{i-1}g_{i-1}^{-1}\left(\begin{matrix}0\\\vdots\\0\\\alpha_i\end{matrix}\right)+\right.\right.\\
    &&\left.\left.+(0,\dots,0,\alpha_i)a_{i-1}g_{i-1}^{-1}\left(\begin{matrix}g_{1,i}\\\vdots\\g_{i-2,i}\\g_{i-1,i}\end{matrix}\right)\right)\right),
    \end{eqnarray*}
    où $(a'_ia''_i)g_i^{-1}$ est la matrice des cofacteurs de $g_i$, lesquels sont dans $(\0\otimes_kR)[(\prod_{i=1}^ra''_i)^{-1}]$ et où $\mathrm{res}((a'_{i-1}a''_{i-1})^{-1} \frac{b}{(\prod_{j=1}^r{a''_j})^s})$ est le coefficient de $\varpi^{(d'_i+d''_i)+s\sum_{j=1}^rd''_j-1}$ dans l'expression polynomiale en la variable $\varpi$ du reste de la
    division euclidienne de $b$ par $a'_{i-1}a''_{i-1}(\prod_{j=1}^ra''_j)^s$ (cette division euclidienne a un sens puisque le coefficient dominant des $a'_i$ et des $a''_i$ sont égals à $1$), induit un morphisme $h'_{1,i} : Y_{1,\underline{d}'}\times \mg_m^{r-1}\ra\mg_a$.
\item Soit $h'_{1,\underline{d}'}=\sum_{i=2}^rh'_{1,i}$. Il existe un morphisme naturel $\underline{\kappa}_{1,\ud'}: Y_{1,\ud'}\to \mg_m$ 
    tel que pour tous $(t',t'')\in (V_{\underline{d}'}\times V_{\ud''})^{\rm dist}(k)$ et $\underline{\alpha} \in \mg_m^{r-1}$, le triplet $(Y_1(t',t''),h'_{1,\underline{\alpha}},\underline{\kappa}_1)$ est isomorphe à la fibre $(f^{Y_1}_{\underline{d}'})^{-1}(t',t'',\underline{\alpha})$ munie de la restriction de $h'_{1,\underline{d}'}$ et de $\underline{\kappa}_{1,\ud'}$ à cette fibre.
\end{enumerate}
\end{lemma}
Les données géométriques de $X_{2,\ud''}$ et de $Y_{2,\ud''}$ sont définies par des propositions analogues.

Ensuite, on d\'efinit des complexes ${\cal I}_1=\mathrm{R}f^{X_1}_{\ud',!}h_{1,\ud'}^*\mathcal{L}_{\psi},$
${\cal I}_2=\mathrm{R}f^{X_2}_{\ud'',!}h_{2,\ud'}^*\mathcal{L}_{\psi}$
${\cal J}_1=\mathrm{R}f^{Y_1}_{\ud',!}({h'}_{1,\ud'}^{*}\mathcal{L}_{\psi}\otimes\underline{\kappa}_{1,\ud'}^{*}\mathcal{L}_{\zeta})$ et
${\cal J}_2=\mathrm{R}f^{Y_2}_{\ud'',!}({h'}_{2,\ud''}^{*}\mathcal{L}_{\psi}\otimes\underline{\kappa}_{2,\ud''}^{*}\mathcal{L}_{\zeta})$ v\'erifiant $\mu^*{\cal I}={\cal I}_1\otimes^{{\mathbb L}}{\cal I}_2$ et $\mu^*{\cal J}={\cal J}_1\otimes^{{\mathbb L}}{\cal J}_2$.

Comme les $a'_i$ sont à racines simples, en utilisant la proposition~\ref{fatf} et la proposition~\ref{ssp}, on obtient que ${\cal I}_1$ et ${\cal J}_1$ sont des produits de syst\`emes locaux, donc sont eux mêmes des systèmes locaux. De plus, en utilisant la formule du produit des constantes de Weil (cf. la proposition \ref{weil}), on définit un système local $\mathcal{T}_1$ telle que $\mathcal{J}_1=\mathcal{T}_1\otimes\mathcal{I}_1$. Plus précisé
ent, on a la proposition suivante
 \begin{proposition} \label{fatfg1}(cf. la proposition~\ref{fatf}) ${\cal T}_1$ est un syst\`eme local de rang $1$ plac\'e en degr\'e $\sum_{i=1}^rd'_i$ au-dessus de $(V_{\ud'}\times V_{\ud''})^{\rm dist}$ et provenant d'un caract\`ere $\tau $ de ${\rm Gal}_{\overline{k}/k}$ tel que $(-1)^{\sum_{i=1}^rd'_i}\tau({\rm Fr}_q)$ est égal à $$\begin{cases}q^{\sum_{i=1}^rd'_i/2}\zeta(-1)^{\sum_{i=0}^{s-1}d'_{2i+1}}\gamma_\infty(\varpi,\Psi_\infty)^{-\sum_{i=1}^{s-1}p(d'_{2i}-d'_{2i+1})}&\text{si $r=2s$},\\
q^{\sum_{i=1}^rd'_i/2}\zeta(-1)^{\sum_{i=1}^{s}d'_{2i}}\gamma_\infty(\varpi,\Psi_\infty)^{-\sum_{i=1}^{s}p(d'_{2i}-d'_{2i-1})}&\text{si $r=2s+1$},
\end{cases}$$
où $p(x)=\begin{cases} 1&\text{si $x$ est impair}\\
0&\text{si $x$ est pair}.\end{cases}$

 \end{proposition}
 En utilisant le fait que la perversit\'e et le prolongement interm\'ediaire sont stables par changement de base \'etale et que le produit tensoriel d'un complexe avec un syst\`eme local est pervers et prolongement interm\'ediaire de sa restriction \`a un ouvert si et seulement si ce complexe l'est d\'ej\`a, on obtient que ${\cal I}_2$ et ${\cal J}_2$ sont pervers et prolongement interm\'ediaire de leur restriction \`a l'ouvert $\mu^*U_{\ud}$.

Le syst\`eme local ${\cal T}$ s'\'ecrit lui aussi comme un produit ${\cal T}_1\otimes {\cal T}_2$ (${\cal T}_1$ et ${\cal T}_2$ ne sont plus g\'eom\'etriquement constants, on d\'efinit en fait ${\cal T}_2={\cal T}\otimes {\cal T}_1^{\otimes -1}$) et à l'aide des propositions \ref{fatf} et \ref{fatfg1}, la formule du produit pour les constantes de Weil permet de calculer ${\rm Tr}({\rm Fr}_t,{\cal T}_2)={\rm Tr}({\rm Fr}_t,{\cal T})/{\rm Tr}({\rm Fr}_t,{\cal T}_1)$. En sp\'ecialisant en $t=t^{\circ}$ on obtient alors le th\'eor\`eme A$'$.

\end{document}